\documentclass[final,onefignum,onetabnum,letterpaper]{siamonline190516}
\usepackage{xr-hyper} 
\usepackage{xr}

\usepackage{hyperref}

\usepackage[utf8]{inputenc}
\usepackage{xr-hyper}
\usepackage{url}
\usepackage{soul}
\usepackage{lipsum}
\usepackage{amsfonts}
\usepackage{epstopdf}
\ifpdf
  \DeclareGraphicsExtensions{.eps,.pdf,.png,.jpg}
\else
  \DeclareGraphicsExtensions{.eps}
\fi

\usepackage{datetime}
\newdateformat{monthyeardate}{%
  \monthname[\THEMONTH] \THEDAY, \THEYEAR}

\usepackage{academicons}
\usepackage{xcolor}
\renewcommand{\orcid}[1]{\href{https://orcid.org/#1}{\textcolor[HTML]{A6CE39}{orcid.org/#1}}}

\usepackage{amsmath}
\allowdisplaybreaks
\usepackage{amssymb}
\usepackage{hyperref,cleveref}
\usepackage{commath}
\usepackage{mathtools}
\usepackage{bbm}

\usepackage{color}
\usepackage{graphicx}
\usepackage[small]{caption}
\usepackage{subcaption}

\usepackage{relsize}
\usepackage{adjustbox}
\usepackage{algorithm,algpseudocode}
\usepackage{booktabs}
\usepackage{tikz}

\usepackage{verbatim}

\usepackage{enumitem}
\setlist[enumerate]{leftmargin=.5in}
\setlist[itemize]{leftmargin=.5in}

\newsiamthm{problem}{Problem}
\newsiamremark{remark}{Remark}
\newsiamremark{hypothesis}{Hypothesis} 
\crefname{hypothesis}{Hypothesis}{Hypotheses}
\newsiamremark{example}{Example}
\newsiamthm{claim}{Claim}
\newsiamthm{conjecture}{Conjecture}

\headers{Priorconditioned Sparsity-Promoting Projection Methods}{J.\ Lindbloom, M.\ Pasha, J.\ Glaubitz, and Y.\ Marzouk}

\title{Priorconditioned Sparsity-Promoting Projection Methods for {Deterministic} and Bayesian Linear Inverse Problems
\thanks{
\monthyeardate\today 
\corresponding{Jonathan Lindbloom} 
}}

\author{
Jonathan Lindbloom\thanks{Department of Mathematics, Dartmouth College, USA (\email{jonathan.t.lindbloom.gr@dartmouth.edu}, \orcid{0000-0002- 1789-2629})}
\and 
Mirjeta Pasha\thanks{Department of Mathematics \& Academy of Data Science, Virginia Tech, USA ( \email{mpasha@vt.edu}, \orcid{0000-0003-4249-2421})} 
\and
Jan Glaubitz\thanks{Department of Mathematics, Link\"oping University, Sweden, (\email{jan.glaubitz@liu.se}, \orcid{0000-0002-3434-5563})}
\and
Youssef Marzouk\thanks{Department of Aeronautics and Astronautics \& Laboratory for Information and Decision Systems, Massachusetts Institute of Technology, USA (\email{ymarz@mit.edu}, \orcid{0000-0001-8242-3290})} }
\newcommand{\diag} [1]  {{\rm diag\!}\left( #1 \right)}    
\usepackage{amsopn}

\usepackage{ifthen}
\usepackage{float}
\floatstyle{ruled}
\newfloat{algorithm}{htb}{alg}
\floatname{algorithm}{Algorithm}
\newcounter{algo@row}
\newcounter{algo@rowindent}
\newcommand{\algofont}[1]{\textbf{#1}}
\newcommand{\algonumbersize}[1]{\scriptsize{#1}}
\newcommand{\algopreitem}[1][\arabic{algo@row}]{\texttt{\algonumbersize{#1}}}
\newcommand{\algoitemskip}{\hspace{\value{algo@rowindent}cc}}
\newcommand{\algonewnestedopen}[2]{
	\newcommand{#1}[1][]{%
		\ifthenelse{\equal{##1}{}}{\item}{\item[{\algopreitem[##1]}]}
		\algoitemskip\algofont{#2}%
		\addtocounter{algo@rowindent}{1}%
		\ignorespaces
	}
}
\newcommand{\algonewnestedaux}[2]{
	\newcommand{#1}[1][]{
		\addtocounter{algo@rowindent}{-1}
		\ifthenelse{\equal{##1}{}}{\item}{\item[{\algopreitem[##1]}]}
		\algoitemskip\algofont{#2}%
		\addtocounter{algo@rowindent}{+1}%
		\ignorespaces
	}
}
\newcommand{\algonewnestedclose}[2]{
	\newcommand{#1}[1][]{
		\addtocounter{algo@rowindent}{-1}
		\ifthenelse{\equal{##1}{}}{\item}{\item[{\algopreitem[##1]}]}
		\algoitemskip\algofont{#2}%
		\ignorespaces
	}
}
\newcommand{\algonewcommand}[2]{
	\newcommand{#1}[1][default]{
		\ifthenelse{\equal{##1}{default}}{\item}{\item[{\algopreitem[##1]}]}%
		\algoitemskip\algofont{#2}%
		\ignorespaces
	}%
}
\newcommand{\algonewkeyword}[2]{\newcommand{#1}{\algofont{#2}}}
\algonewcommand{\STATE}{\ignorespaces}
\algonewcommand{\INPUT}{Input: }
\algonewcommand{\pINPUT}{\phantom{Input: }}
\algonewcommand{\COMPUTE}{Compute: }
\algonewcommand{\OUTPUT}{Output: }
\algonewcommand{\pOUTPUT}{\phantom{Output: }}

\algonewnestedopen{\IF}{if }
\algonewnestedaux{\ELSEIF}{else if }
\algonewnestedaux{\ELSE}{else }
\algonewnestedclose{\ENDIF}{end if }
\algonewnestedopen{\FOR}{for }
\algonewnestedclose{\ENDFOR}{end for }
\algonewnestedopen{\WHILE}{while }
\algonewnestedclose{\ENDWHILE}{end while }
\algonewcommand{\BREAK}{break}%

\algonewkeyword{\To}{to }%
\algonewkeyword{\Do}{do }%
\algonewkeyword{\Then}{then }%
\algonewkeyword{\End}{end }%
\algonewkeyword{\AND}{and }%
\algonewkeyword{\True}{true }%
\algonewkeyword{\False}{false }%
\algonewkeyword{\irbleigs}{irbleigs }%
\algonewkeyword{\tridiag}{tridiag}%
\algonewkeyword{\reorth}{reorth}%

\newcommand{\MAP}{\mathrm{MAP}}
\DeclareMathOperator*{\argmin}{arg\,min}

\newcommand{\R}{\mathbb{R}}

\newcommand{\bA}{{\bf A}}

\newcommand{\bI}{{\bf I}}

\newcommand{\bQ}{{\bf Q}}
\newcommand{\bR}{{\bf R}}

\newcommand{\bV}{{\bf V}}
\newcommand{\bW}{{\bf W}}

\newcommand{\bb}{{\bf b}}

\newcommand{\be}{{\bf e}}

\newcommand{\br}{{\bf r}}

\newcommand{\bv}{{\bf v}}
\newcommand{\bw}{{\bf w}}
\newcommand{\bx}{{\bf x}}

\newcommand{\bz}{{\bf z}}

\newcommand{\regparam}{\mu}
\newcommand{\itidx}{\ell}

\makeatletter
\newcommand*{\addFileDependency}[1]{
  \typeout{(#1)}
  \@addtofilelist{#1}
  \IfFileExists{#1}{}{\typeout{No file #1.}}
}
\makeatother

\usepackage{lineno}
\linenumbers

\ifpdf
\hypersetup{
  pdftitle={Priorconditioned projection methods},
  pdfauthor={Lindbloom}
}
\fi

\begin{document}
\nolinenumbers
\maketitle
\begin{abstract}
{
High-quality reconstructions of signals and images with sharp edges are needed in a wide range of applications.
} To overcome the large dimensionality of the parameter space and the complexity of the regularization functional, {sparisty-promoting} techniques for both deterministic and hierarchical Bayesian regularization rely on solving a sequence of high-dimensional iteratively reweighted least squares (IRLS) problems on a lower-dimensional subspace. 
Generalized Krylov subspace (GKS) methods are a particularly potent class of hybrid Krylov schemes that efficiently solve sequences of IRLS problems by projecting large-scale problems into a relatively small subspace and successively enlarging it.
We refer to methods that promote sparsity and use GKS as S-GKS. 
A disadvantage of S-GKS methods is their slow convergence. 
In this work, we propose techniques that improve the convergence of S-GKS methods by combining them with priorconditioning, which we refer to as PS-GKS. 
Specifically, integrating the PS-GKS method into the IAS algorithm allows us to automatically select the shape/rate parameter of the involved generalized gamma hyper-prior, which is often fine-tuned otherwise. 
Furthermore, we proposed and investigated variations of the proposed PS-GKS method, including restarting and recycling (resPS-GKS and recPS-GKS). 
These respectively leverage restarted and recycled subspaces to overcome situations when memory limitations of storing the basis vectors are a concern.
We provide a thorough theoretical analysis showing the benefits of priorconditioning for sparsity-promoting inverse problems. Numerical experiment are used to illustrate that the proposed PS-GKS method and its variants are competitive with or outperform other existing hybrid Krylov methods.
\end{abstract}

\begin{keywords}
    priorconditioning, generalized Krylov subspace, sparsity, majorization minimization, generalized sparse Bayesian learning, {linear} inverse problems
\end{keywords}

\begin{AMS}

\end{AMS}

\begin{Code}
    \url{https://github.com/mpasha3/IRLS_prec_GSBL}
\end{Code}

\begin{DOI}
    Not yet assigned
\end{DOI}

\section{Introduction} 
\label{sec:intro} 
Recovering high-quality signals and images from indirect, incomplete, and noisy observations is a common yet challenging problem in various applications. The task is often modeled as a linear inverse problem
\begin{equation}\label{eq:IP}
    \bb = \bA \bx + \be, 
\end{equation}
where $\bb \in \R^M$ denotes the observed data, $\bx \in \R^N$ is the unknown parameter vector (e.g., the signal or the vectorized image), $\bA \in \R^{M \times N}$ is a known linear forward operator, and $\be \in \R^M$ corresponds to noise. 
 
The inverse problem \cref{eq:IP} is typically \emph{ill-posed}, resulting in the solution of \cref{eq:IP} being non-unique, not existing at all, or being highly sensitive. One way to overcome ill-posedness is through \emph{regularization}, wherein one instead seeks the solution to a nearby regularized inverse problem:

\begin{equation}\label{eq:regIP}
    \argmin_{\bx \in \R^N} \left\{ 
        \mathcal{F}(\bx {; \bb})
        + \mathcal{R}(\mathbf{x})
    \right\},
\end{equation}

where $\mathcal{F}$ and $\mathcal{R}$ denote the data-fidelity and regularization term, respectively.
The choice of $\mathcal{F}$ is informed by the assumptions about the data-generating process and noise characteristics. 
For simplicity, we assume that $\be$ is a realization of standard normal noise, i.e., $\be \sim \mathcal{N}(\mathbf{0},\mathbf{I})$, yielding $\mathcal{F}(\bx {; \bb}) = \|\bA\bx - \bb\|_2^2$.\footnote{
If $\be \sim \mathcal{N}(\mathbf{0},\boldsymbol{\Sigma})$ with a symmetric positive definite covariance matrix $\boldsymbol{\Sigma} \neq \mathbf{I}$, there exists a Cholesky decomposition $\boldsymbol{\Sigma} = \mathbf{C} \mathbf{C}^T$, and the problem can be whitened by substituting $\mathbf{A} \gets \mathbf{C}^{-1} \mathbf{A}$ and $\mathbf{b} \gets \mathbf{C}^{-1} \mathbf{b}$.}
{
The regularization term $\mathcal{R}$ encodes one's prior belief about the structure of the otherwise unknown parameter vector $\bx$. 
A common assumption is that $\bx$ is sparse or has a sparse representation $\boldsymbol{\Psi} \bx$ with sparsifying transform $\boldsymbol{\Psi} \in \mathbb{R}^{K \times N}$. 
For instance, $\boldsymbol{\Psi}$ can be a discrete gradient operator or a wavelet transformation.
} 
Sparsity for $\boldsymbol{\Psi} \bx$ can be promoted by selecting $\mathcal{R}$ as some scaled surrogate for the $\ell_0$-``norm" $\| \boldsymbol{\Psi} \bx \|_0$, which counts the number of non-zero components. 

An alternative to the above deterministic regularization setting is the \emph{Bayesian approach} to inverse problems \cite{stuart2010inverse,calvetti2023bayesian}. 
where we treat the unknown parameters as random variables and impose a prior distribution on them. 
For instance, assuming additive standard normal noise $\be \sim \mathcal{N}(\mathbf{0},\mathbf{I})$ in \cref{eq:IP} corresponds to a likelihood function $\pi( \mathbf{b} \vert \bx) \ \propto \  \exp \left (- \frac{1}{2}\| \bA \bx - \mathbf{b} \|^2_2 \right )$. 
Assuming further a prior distribution $\pi^{0}(\bx)$ for the random variable of interest $\bx$, Bayes' rule prescribes a formula for the posterior density $\pi_{\rm pos}(\bx|\bb) \ \propto \ \pi(\bb \vert \bx) \pi^{0}(\bx)$. 
Furthermore, the role of the regularization term is now taken by a prior distribution $\pi^{0}(\bx)$, encoding our structural belief about $\bx$---in this case, that it has a sparse representation. 
Finally, the sought-after posterior density $\pi_{\rm pos}(\bx|\bb)$ is provided by Bayes' rule as $\pi_{\rm pos}(\bx|\bb) \ \propto \ \pi(\bb \vert \bx) \pi^{0}(\bx)$.
A particularly potent class of sparsity-promoting priors is the generalized sparse Bayesian learning (GSBL) priors \cite{tipping2001sparse,calvetti2020sparse,glaubitz2023generalized}, where the main idea is to consider a joint prior $\pi^0( \mathbf{x}, \boldsymbol{\theta} ) = \pi^0( \mathbf{x} | \boldsymbol{\theta}) \pi^0( \boldsymbol{\theta} )$ that combines a conditional Gaussian prior $\pi^0( \mathbf{x} | \boldsymbol{\theta} )$ and a generalized gamma hyper-prior $\pi^0( \boldsymbol{\theta} )$. 
Here, $\boldsymbol{\theta} = [\theta_1,\dots,\theta_K]^T$ is a vector of auxiliary hyper-parameters that encode the sparsity profile of $\boldsymbol{\Psi} \bx$. 
In this paper, we focus on developing prior-conditioning strategies for S-GKS methods in both deterministic and Bayesian settings. 
Specifically, to handle GSBL priors in the Bayesian setting, we consider the iterative alternating sequential (IAS) algorithm \cite{calvetti2007gaussian,calvetti2019hierachical,calvetti2020sparse,lindbloom2024generalized}. 
A comprehensive discussion on deterministic regularization for linear inverse problems can be found in \cite{vogel2002computational,hansen2006deblurring,hansen2010discrete} and for the Bayesian setting {in} \cite{stuart2010inverse,calvetti2023bayesian,glaubitz2023generalized,dong2023inducing}. 

Both the S-GKS and the IAS {methods} aim to solve the inverse problem \cref{eq:IP} with {sparisty-promoting} assumptions and rely on efficiently performing the IRLS iterations
\begin{equation}\label{eq:IRLS}
    \bx_{\itidx+1} = \argmin_{\bx \in \R^{N}} \left\{ 
        \norm{ \bA \bx - \bb }_2^2 
        + \regparam \norm{ \bW_{\itidx+1} \boldsymbol{\Psi} \bx }_2^2
    \right\}, \quad \itidx = 0, 1, 2, \ldots,
\end{equation}
where the weight matrix $\bW_{\itidx+1}$ depends on the previous approximate solution $\bx_{\itidx}$. Here, we assume $\mathbf{W}_{\itidx+1} = \operatorname{diag}(\mathbf{w}_{\itidx+1})$ with $\mathbf{w}_{\itidx+1} \in \mathbb{R}_{++}^K$ containing strictly positive \emph{weights}. 
Solving \cref{eq:IRLS} can be computationally challenging and may easily exceed the memory limits of the device being used \cite{pasha2023recycling,pasha2023computational}. 
One approach {for overcoming these computational bottlenecks} is to solve \cref{eq:IRLS} on smaller-dimensional subspaces using hybrid projection methods \cite{pasha2023computational,pasha2023recycling,buccini2023limited,chung2024computational}. 
For instance, for a fixed $\regparam$, for which we can utilize the CGLS iterative method to solve \cref{eq:IRLS}. 
However, estimating $\regparam$ is crucial for ill-posed inverse problems to set a balance between the data fidelity and the regularization term. 
On the other hand, hybrid projection methods have potential to efficiently solve massive scale ill-posed inverse problems with complex regularization terms $\mathcal{R}(\bx)$, see for instance \cite{pasha2023computational, pasha2023recycling, buccini2023limited} and define the regularization parameter automatically, see for instance \cite{chung2024computational, pasha2024trips}. Furthermore, in the Bayesian setting, alternating between the $\bx$- and $\boldsymbol{\theta}$-updates can be computationally demanding and involves solving iteratively reweighted least squares problems. 
Moreover, estimating the hyper-parameters $\boldsymbol{\theta}$ requires solving the problem several times to fine-tune the desired parameter. Similarly, even though it has shown potential in many practical settings, S-GKS approach \cite{huang2017majorization, lanza2015generalized, buccini2020modulus} used for complex regularization terms can exhibit slow convergence and other computational limitations, as pointed out in several recent manuscripts \cite{pasha2023computational, pasha2023recycling, buccini2023limited}.

\subsection*{Our contribution}

We propose a novel priorconditioning {strategy} that can be used in sparsity-promoting techniques in both deterministic and Bayesian settings for ill-posed linear inverse problems. For each case, we present the weights and describe how priorconditioning can be used in the context of reweighting. {Furthermore, we propose variations of our method that employ restarting and recycling to overcome memory limitations. Specifically, this paper's main contributions can be summarized as follows:}
\begin{enumerate}
\item Driven by the need for computationally feasible methods for large-scale linear inverse problems, we propose a priorconditioning strategy {(referred to as ``PS-GKS")} that is based on GKS and can be used to efficiently solve {IRLS} problems arising from the deterministic or Bayesian setting. 
This allows us to efficiently and automatically select the model parameters (see below for details).
    
\item To overcome the computational bottleneck of computing the pseudoinverse needed for the priorconditioning, we propose a prior conditioned CG method that is then further accelerated by GPU, {making it} applicable for large-scale inverse problems. 
    
\item A comprehensive comparison to other sparsity-promoting methods that utilize Krylov subspaces is provided. 
In particular, we compare our PS-GKS method with competing methods based on the Golub-Kahan bidiagonalization and the flexible Golub-Kahan process. 
{Numerical experiments in 1D and 2D (X-ray computerized tomography (CT) applications)} illustrate our proposed method's performance.

\end{enumerate} 
We observe that our PS-GKS method significantly improves the existing S-GKS method by substantially reducing the number of iterations required for convergence. 
Such improvement is observed throughout several reweighting strategies used in both the deterministic and Bayesian formulation. 

\subsection*{Outline}

We begin in \Cref{sec:application} by motivating IRLS problems and the need for priorconditioning in two distinct settings: 

{The deterministic framework employing the majorization minimization (MM) weights \cite{huang2017majorization} combined with GKS and the GSBL framework using the IAS algorithm for efficient MAP estimation.} In \Cref{sec:background}, we provide background material on existing methods that rely on generalized Krylov subspaces for computational efficiency. The core contribution of the paper, including the proposed priorconditioning method, its theoretical properties, and various algorithmic refinements, are presented in \Cref{sec:proposed}. Finally, we demonstrate the effectiveness and versatility of our approach through a series of numerical experiments and comparative studies in \Cref{sec:numerics}.
In \Cref{sec:conclusion}, we conclude with a summary and outlook on possible directions for future research. 
All test problems and algorithm implementations in Python will be made publicly available at \href{https://github.com/mpasha3/IRLS_prec_GSBL}{https://github.com/mpasha3/IRLS\_prec\_GSBL} once the manuscript is accepted to the journal.


\section{Application to deterministic and Bayesian inverse problems}
\label{sec:application} 

We outline two motivating examples of sparsity-promoting methods for linear inverse problems that rely on efficiently solving a sequence of IRLS problems: One deterministic method arising from the MM approach to $\ell^p$-regularization and one Bayesian method arising from MAP estimation with sparsity-promoting GSBL priors. 
\subsection{Deterministic regularization and the MM approach}\label{sub:problem_MM} 

A common technique to promote sparsity in $\boldsymbol{\Psi} \mathbf{x}$ is to seek the solution to the deterministic problem
\begin{equation}\label{eq:general_tikhonov_problem}
    \argmin_{\bx\in \R^N} \left\{ \mathcal{J}(\bx) \right\}, \quad 
    \mathcal{J}(\bx) = \norm{ \mathbf{A} \mathbf{x} - \mathbf{b} }_2^2 
        + \frac{\regparam}{p} \norm{  \boldsymbol{\Psi}\mathbf{x} }_p^p,
\end{equation}
with $0<p\leq 1$. 
Notably, the objective $\mathcal{J}$ is generally neither smooth nor convex. 
The popular MM approach \cite{hunter2004tutorial,huang2017majorization} addresses this challenging structure of $\mathcal{J}$ by successively minimizing a sequence of smooth approximations--- so-called quadratic tangent majorants---of the original functional, resulting in an IRLS problem of the form \cref{eq:IRLS}. 
Specifically, a smooth approximation of the $p$-norm 
in \cref{eq:general_tikhonov_problem} is given by $\|\bz\|_p^p \approx \sum_{k} \phi_{p, \varepsilon}(\bz_k)$ with $\phi_{p, \varepsilon}(\bz_i) = (\bz_k^2+\varepsilon^2)^{p/2}$ being an approximation of $|\bz_k|^p$ and $\varepsilon>0$. 
The first step to building a quadratic tangent majorant consists of constructing the smoothed functional 
$\mathcal{J}_{\varepsilon}(\bx) = \|\bA\bx-\bb\|_2^2 + (\regparam/p) \sum_{k} \phi_{p, \varepsilon}([\boldsymbol{\Psi}\bx]_{k}).$ 
Let $\bx_{\itidx}$ be an available approximation of the desired solution. Assuming an adaptive quadratic majorant, we select the weighting matrix as 
\begin{equation}\label{eq:MM_original}
\bW_{\itidx+1} = \text{diag}\left({ (\boldsymbol{\Psi}\bx_{\itidx})^2+\varepsilon^2}\right)^{\frac{p-2}{4}},
\end{equation}
which yields the following quadratic tangent majorant for $\mathcal{J}_\varepsilon(\bx)$:
\begin{equation}\label{eq: QuadraticMajorantQ}
\mathcal{M}(\bx, \bx_{\itidx})  = \displaystyle{\frac{1}{2}} \|\bA\bx-\bb\|^{2}_{2}
+\displaystyle{\frac{\regparam}{2}} \|\bW_{\itidx+1}\boldsymbol{\Psi}\bx\|^{2}_{2}+c,
\end{equation}
where $c$ is a constant that is independent of $\bx$ and $\bx_{\itidx}$. 

The new approximation of the solution, $\mathbf{x}_{\itidx+1}$, is then obtained by minimizing \cref{eq: QuadraticMajorantQ} by a standard method such as CGLS. {The process of defining and minimizing a new quadratic tangent majorant is repeated, resulting in a sequence of IRLS problems as in \cref{eq:IRLS} where the weights are given as in \cref{eq:MM_original}. 

\subsection{Bayesian inverse problems: GSBL and the IAS approach} 
\label{sub:problem_Bayesian}

We next demonstrate how IRLS naturally arise in sparsity-promoting Bayesian approaches to linear inverse problems using hierarchical priors. 
For simplicity, we focus on efficient MAP estimation within the GSBL approach \cite{glaubitz2023generalized,xiao2023sequential,glaubitz2024leveraging} using the popular IAS algorithm \cite{calvetti2007gaussian,calvetti2019hierachical,calvetti2020sparse,lindbloom2024generalized}. 
Notably, combining the IAS algorithm with the proposed PS-GKS method later allows us to automate the selection of the rate parameter $\vartheta$ (which serves as a regularization parameter) of the generalized gamma hyper-prior, reducing the need for manual fine-tuning. 

In the Bayesian approach \cite{stuart2010inverse,calvetti2023bayesian}, the inverse problem \cref{eq:IP} is framed as a statistical inference problem based on the posterior distribution, which combines the likelihood function $f( \mathbf{x}; \mathbf{b} )$ implied by \cref{eq:IP} with a prior density $\pi^0$ that encodes our structural beliefs about $\mathbf{x}$. 
Consider the data model \cref{eq:IP} with whitened noise $\mathbf{e} \sim \mathcal{N}(\mathbf{0}, \mathbf{I})$. 
In this case, the likelihood is $f(\mathbf{x};\mathbf{b}) \propto \exp( -\frac{1}{2} \norm{ \bA \mathbf{x} - \mathbf{b} }_2^2 )$. 
To formulate the prior density $\pi^0$, we again assume that $\boldsymbol{\Psi} \mathbf{x} \in \R^{K}$ is sparse. 
A particularly potent class of sparsity-promoting priors that we consider in this work are the GSBL priors $\pi^0( \mathbf{x}, \boldsymbol{\theta} ) = \pi^0( \mathbf{x} | \boldsymbol{\theta} ) \pi^0( \boldsymbol{\theta} )$, combining a conditional Gaussian prior $\pi^0( \mathbf{x} | \boldsymbol{\theta} )$ and a generalized gamma hyper-prior $\pi^0( \boldsymbol{\theta} )$, where $\boldsymbol{\theta} = [\theta_1,\dots,\theta_K]$ is a vector of auxiliary hyper-parameters. 
Specifically, we assume that the $k$th component of $\boldsymbol{\Psi} \mathbf{x} \in \R^K$ is independently normal-distributed with mean zero and variance $\theta_k$, i.e., $[ \boldsymbol{\Psi} \mathbf{x} ]_k | \theta_k \sim \mathcal{N}(0,\theta_k)$.
The variance $\theta_k$ is also modeled as a random variable, which is generalized gamma-distributed, i.e., $\theta_k \sim \mathcal{GG}(r,\beta,\vartheta)$ with parameters $r \in \R \setminus \{0\}$, $\beta > 0$, and $\vartheta > 0$.
The resulting GSBL posterior density $\pi^{\mathbf{b}}$ for $(\mathbf{x}, \boldsymbol{\theta})$ conditioned on $\mathbf{b}$ follows from Bayes' theorem and is given by  
\begin{equation}\label{eq:posterior_GSBL}
\resizebox{.92\textwidth}{!}{$\displaystyle 
    \pi^{\mathbf{b}}(\mathbf{x},\boldsymbol{\theta}) 
	\propto \exp\left( 
		-\frac{1}{2} \norm{ \bA \mathbf{x} - \mathbf{b} }_2^2 
            -\frac{1}{2} \norm{ \mathbf{D}_{\boldsymbol{\theta}}^{-1/2} \boldsymbol{\Psi} \mathbf{x} }_2^2
            - \sum_{k=1}^K \left[ 
			\left( \frac{\theta_k}{\vartheta} \right)^r 
		      - \left(r \beta - \frac{3}{2} \right) \log \theta_k 
		\right]
	\right)
$}    
\end{equation}
with diagonal matrix $\mathbf{D}_{\boldsymbol{\theta}} = \diag{\boldsymbol{\theta}}$.

The \emph{MAP estimate} $(\mathbf{x}^{\MAP},\boldsymbol{\theta}^{\MAP})$ of $\pi^{\mathbf{b}}(\mathbf{x},\boldsymbol{\theta})$ is the maximizer of the joint posterior in \cref{eq:posterior_GSBL}. 
Equivalently, the MAP estimate is the minimizer of the negative logarithm of the joint posterior, i.e., $(\mathbf{x}^{\MAP},\boldsymbol{\theta}^{\MAP}) 
	= \argmin_{ \mathbf{x}, \boldsymbol{\theta} } \left\{ 
        \mathcal{J}( \mathbf{x}, \boldsymbol{\theta} ) 
    \right\}$ with $\mathcal{J} = - \log \pi^{\mathbf{b}}(\mathbf{x},\boldsymbol{\theta})$.
A prevalent strategy to approximate the minimizer of $\mathcal{J}$ is to use block-coordinate descent methods \cite{wright2015coordinate,beck2017first} that aim to minimize $\mathcal{J}$ by alternatingly minimizing $\mathbf{x}$ and $\boldsymbol{\theta}$. 
For MAP estimation of the GSBL posterior, the same strategy is leveraged by IAS. 
We refer to \cite{calvetti2020sparsity,lindbloom2024generalized} for details on the $\boldsymbol{\theta}$-update, which can be efficiently performed by finding the root of a simple quadratic function if $r = \pm 1$ and by solving an ordinary differential equation in all other cases. 
Furthermore, the $\mathbf{x}$-update reduces to solving a quadratic optimization problem 
\begin{equation}\label{eq:IAS_x_update} 
	\mathbf{x}_{\itidx+1} = \argmin_{\mathbf{x} \in \mathbb{R}^N} \left\{ 
        \norm{ \bA \mathbf{x} - \mathbf{b} }_2^2 
        + \norm{ \mathbf{D}_{\boldsymbol{\theta}_{\itidx+1}}^{-1/2} \boldsymbol{\Psi} \mathbf{x} }_2^2 
    \right\}.
\end{equation}

To place the IAS algorithm into the IRLS form of \cref{eq:IRLS}, we perform the change of variables $\boldsymbol{\xi}_{\itidx+1} = \boldsymbol{\theta}_{\itidx+1}/\vartheta$, transforming \cref{eq:IAS_x_update} into 
\begin{equation}\label{eq:IAS_GKS_x_update} 
		\mathbf{x}_{\itidx+1} = \argmin_{\mathbf{x} \in \mathbb{R}^N} \left\{ 
            \norm{ \bA \mathbf{x} - \mathbf{b} }_2^2 
            + \vartheta^{-1} \norm{ \mathbf{D}_{\boldsymbol{\xi}_{\itidx+1}}^{-1/2} \boldsymbol{\Psi} \mathbf{x} }_2^2 
        \right\}.
\end{equation} 
Notably, \cref{eq:IAS_GKS_x_update} will allow us to introduce the proposed PS-GKS method into the IAS algorithm, resulting in the automated selection of the rate parameter $\vartheta$ of the generalized gamma hyper-prior, removing the need for manually fine-tuning it. 

\begin{remark}
    While the above discussion focuses on the GSBL framework, it extends to any sparsity-promoting hierarchical prior based on scale-mixtures of normals \cite{beale1959scale,andrews1974scale}, including Laplace \cite{west1987scale,park2008bayesian,flock2024continuous},  horseshoe priors \cite{carvalho2009handling,uribe2023horseshoe,dong2023inducing}, and potentially Besov priors \cite{lan2023spatiotemporal}. 
\end{remark}


\section{Background} 
\label{sec:background}

As demonstrated above, IRLS problems arise naturally in various approaches to linear inverse problems. 
We now review hybrid projection methods, specifically the GKS approach, for efficiently solving IRLS problems of the form \cref{eq:IRLS}. 
Furthermore, we demonstrate the limitations of existing GKS methods, which subsequently motivate the development of {priorconditioning strategies in the context of hybrid methods}.

\subsection{Projected IRLS and GKS}
\label{sub:sgks}

For large-scale problems with thousands or millions of unknowns, the computational bottleneck of the IRLS problem \cref{eq:IRLS} is that we have to repeatedly solve high-dimensional least squares problems. 
This becomes extremely costly as iterative methods may require a large number of matrix-vector products (matvecs) with $\mathbf{A}$ and $\boldsymbol{\Psi}$ to produce a solution of sufficient quality \cite{pasha2023computational}. 
Moreover, traditional techniques \cite{vogel2002computational, hansen2010discrete} for selecting an appropriate regularization parameter $\regparam$ in \cref{eq:general_tikhonov_problem} rely on solving \cref{eq:general_tikhonov_problem} for a potentially large number of different $\regparam$-values, further increasing computational costs. 
To alleviate the computational burdens of solving IRLS problems, various hybrid projection methods have been introduced \cite{chung2024computational}, which replace \cref{eq:IRLS} with a sequence of projected IRLS problems:
\begin{equation}\label{eq:projected_IRLS}
    \bx_{\itidx+1} = \argmin_{\bx \in \mathcal{V}_\itidx} \left\{ 
        \norm{ \bA \bx - \bb }_2^2 
        + \regparam_{\itidx+1} \norm{ \bW_{\itidx+1} \boldsymbol{\Psi} \bx }_2^2
    \right\}, \quad \itidx = 0, 1, 2, \ldots, 
\end{equation}
where $\{ \mathcal{V}_\itidx \}_{\itidx \geq 0}$ is a nested sequence of low-dimensional approximation spaces. 
We denote their dimensions by $D_\itidx = \operatorname{dim}(\mathcal{V}_\itidx)$, also called \emph{basis sizes}.
The heuristic behind this approach is that solving each projected, $D_\itidx$-dimensional problem---including regularization parameter selection---is significantly cheaper compared to the original problem \cref{eq:IRLS}.

Here, we focus on projection methods based on generalized Krylov subspaces (GKS) for $\{ \mathcal{V}_\itidx \}_{\itidx \geq 0}$. 
The GKS approach was introduced in \cite{lampe2012large} to solve \cref{eq:general_tikhonov_problem} with $p = 2$. 
Subsequently, \cite{lanza2015generalized,huang2017majorization} combined the GKS method with a projected IRLS scheme of the form \cref{eq:projected_IRLS} to solve \cref{eq:general_tikhonov_problem} for any $0 < p < 2$, called the \emph{MM-GKS} approach. 

For generality, we consider MM-GKS as a subset of the broader class of \emph{sparsity-promoting GKS (S-GKS) methods} for solving \cref{eq:projected_IRLS}. 
These methods employ GKS for the subspaces and allow for any sparsity-promoting weights, including those derived from the GSBL approach in \Cref{sub:problem_Bayesian}. 
Specifically, the S-GKS method begins by selecting an initial $\mathcal{V}_0$ and $\mathbf{x}_0$. 
A typical choice for $\mathcal{V}_0$ is the standard Krylov subspace $\mathcal{V}_0 = \mathcal{K}_{h}(\mathbf{A}^T \mathbf{A}, \mathbf{A}^T \mathbf{b}) = {\rm span}\{(\bA^T\bA)^{0}\bA^T\bb,\ldots,(\bA^T\bA)^{h-1}\bA^T\bb\}$ with a relatively small $h$ (for instance, $h = 5$ see \cite{pasha2023computational,pasha2023recycling}). At the $(\itidx+1)$th iteration, one then computes $\mathbf{W}_{\itidx+1}$, $\boldsymbol{\Psi}_{\itidx+1} = \mathbf{W}_{\itidx+1} \boldsymbol{\Psi}$, and the economic QR factorizations $ \mathbf{A} \mathbf{V}_\itidx = \mathbf{Q}_{\mathbf{A}} \mathbf{R}_{\boldsymbol{\Psi}}, \boldsymbol{\Psi}_{\itidx+1} \mathbf{V}_\itidx = \mathbf{Q}_{\boldsymbol{\Psi}} \mathbf{R}_{\boldsymbol{\Psi}}$.  Here, $\mathbf{V}_{l}$ is a matrix whose columns form an orthonormal basis for $\mathcal{V}_{l}$. 
It can be computed using, for instance, the Golub--Kahan bidiagonalization \cite[\S 10.4]{golub2013matrix}. 
Using $\mathcal{V}_{l}$ allows to re-write the constraint problem \cref{eq:projected_IRLS} as the unconstrained problem 
\begin{align}\label{eq:sgks_projected_problem}
    \min_{\mathbf{z} \in \mathbb{R}^{D_\itidx} } \left\{ \| \mathbf{A} \mathbf{V}_{\itidx} \mathbf{z} - \mathbf{b} \|_2^2 + \regparam_{\itidx+1} \| \boldsymbol{\Psi}_{\itidx+1} \mathbf{V}_{\itidx} \mathbf{z} \|_2^2  \right\}.
\end{align}
One then substitutes the QR decompositions into \cref{eq:sgks_projected_problem} to obtain 
\begin{equation}\label{eq:l2_gks_qr_factored_problem}
    \min_{\mathbf{z} \in \mathbb{R}^{D_\itidx} } \left\{ \| \mathbf{R}_{\mathbf{A}} \mathbf{z} - \mathbf{Q}_{\mathbf{A}}^T \mathbf{b} \|_2^2 + \regparam_{\itidx+1} \| \mathbf{R}_{\boldsymbol{\Psi}} \mathbf{z} \|_2^2  \right\}.
\end{equation}
A regularization parameter selection method is then applied to choose $\regparam_{\itidx+1}$ in \cref{eq:l2_gks_qr_factored_problem}, which we comment on in \Cref{sub:dp_priorconditioned_projected}. 
Next, \cref{eq:l2_gks_qr_factored_problem} is solved and the solution $\mathbf{z}_{\itidx+1}$ is mapped back to the original $N$-dimensional space via $\mathbf{x}_{\itidx+1} = \mathbf{V}_{\itidx} \mathbf{z}_{\itidx+1}$. If convergence has not yet been confirmed, the residual vector $\mathbf{r}_{\itidx+1} = \mathbf{A}^T (\mathbf{A} \mathbf{V}_{\itidx} \mathbf{z}_{\itidx+1} - \mathbf{b}) + \regparam_{\itidx+1} \boldsymbol{\Psi}_{\itidx+1}^T (\boldsymbol{\Psi}_{\itidx+1} \mathbf{V}_{\itidx} \mathbf{z}_{\itidx+1})$ is incorporated into the subspace of the next iteration. 
This is done by by computing $\bv_{\rm new} = (\mathbf{I}_N - \mathbf{V}_\itidx \mathbf{V}_\itidx^T) \mathbf{r}_{\itidx+1}/\| (\mathbf{I}_N - \mathbf{V}_\itidx \mathbf{V}_\itidx^T) \mathbf{r}_{\itidx+1} \|_2 $, setting $\mathbf{V}_{\itidx+1} = [\mathbf{V}_{\itidx}, \mathbf{v}_{\text{new}}]$, and choosing $\mathcal{V}_{\itidx+1} = \operatorname{col}(\mathbf{V}_{\itidx+1})$. 
The above S-GKS iteration is repeated until a pre-specified convergence criterion is satisfied. 
We summarize the S-GKS method in \cref{alg:sgks}.

\begin{remark}[Restarting and recycling]\label{rem:restarting} 
    Several computational experiments \cite{pasha2023computational, pasha2023recycling, lindbloom2024generalized} have demonstrated that many iterations may be necessary for convergence when S-GKS is applied to large-scale problems.
    In this case, the computational costs of using a high-dimensional subspace quickly become prohibitive. 
    Furthermore, the associated storage requirement can easily exceed the memory capacity. 
    Recently, restarted and recycled variants of S-GKS \cite{buccini2023limited, pasha2023recycling} have been proposed: once the basis size reaches a dimension limit $D_{\text{max}}$, the basis is compressed into a smaller subspace of dimension $D_{\text{min}}$. 
    The resulting ``restarted'' S-GKS (resS-GKS) and ``recycled'' S-GKS (recS-GKS) method have a $\mathcal{O}(  D_{\itidx} N )$ memory requirement and $\mathcal{O}( D_{\itidx}^2 M )$  computational cost per iteration with $D_\itidx = D_{\text{min}} + \itidx \bmod (D_{\text{max}} + 1)$, which can be significantly cheaper than S-GKS for a large iteration index $\itidx$. 
\end{remark}

\subsection{An illustrative example}
\label{sub:illustrative_example} We present an illustrative one-dimensional example to highlight the current limitations of the S-GKS method. 
Specifically, we consider reconstructing a piecewise-constant discrete signal, $\mathbf{x}_{\text{true}} \in \mathbb{R}^{1000}$, from noisy observations of its first $50$ discrete cosine transform (DCT) coefficients. 
We define the noise level as $\sigma_{\text{NL}} = \sqrt{M} / \| \mathbf{A} \mathbf{x}_{\text{true}} \|_2$, which we set to 3\%. 
As the sparsifying operator $\boldsymbol{\Psi} \in \R^{K \times N}$, we use a usual discrete first derivative operator with right-hand side homogenous Dirichlet boundary condition, i.e., $[\boldsymbol{\Psi} \mathbf{x}]_k = x_{k} - x_{k+1}$ for $k<K$ and $[\boldsymbol{\Psi} \mathbf{x}]_K = x_K$. 
We examine reconstructions obtained using S-GKS with two different choices of weights: (1) MM weights with $p = 1$, $\varepsilon = 10^{-3}$, corresponding to $\ell_1$ regularization with weights as in \cref{eq:MM_original}, and (2) IAS weights with $r = -1$ and $\beta = -1$, which promote sparsity more aggressively than $\ell_1$-regularization \cite{calvetti2020sparse, calvetti2020sparsity}. 
As a baseline, we also consider the ``vanilla'' GKS method with equal weights $\mathbf{w}_{\itidx+1} = \mathbf{1}_K$.

\begin{figure}[tb]
    \centering
    \includegraphics[width=.99\textwidth]{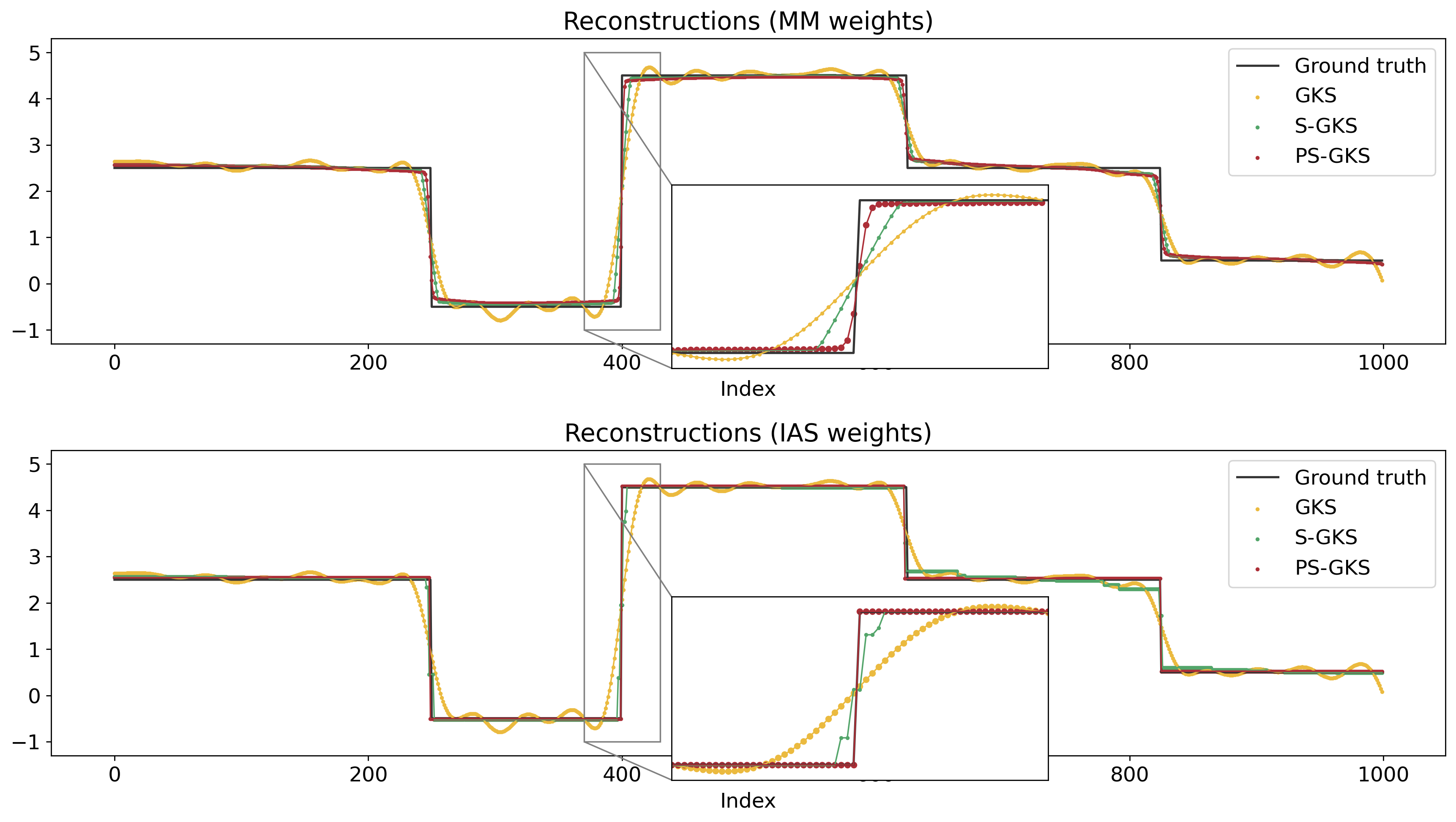}
    \caption{Comparison of GKS (with equal weights), S-GKS, and PS-GKS reconstructions for the 1D cosine problem with MM weights (top row) and IAS weights (bottom row).
    }
    \label{fig:1d_cosine_results_mm_and_ias}
\end{figure}

\begin{figure}[tb]
    \centering
    \includegraphics[width=.99\textwidth]{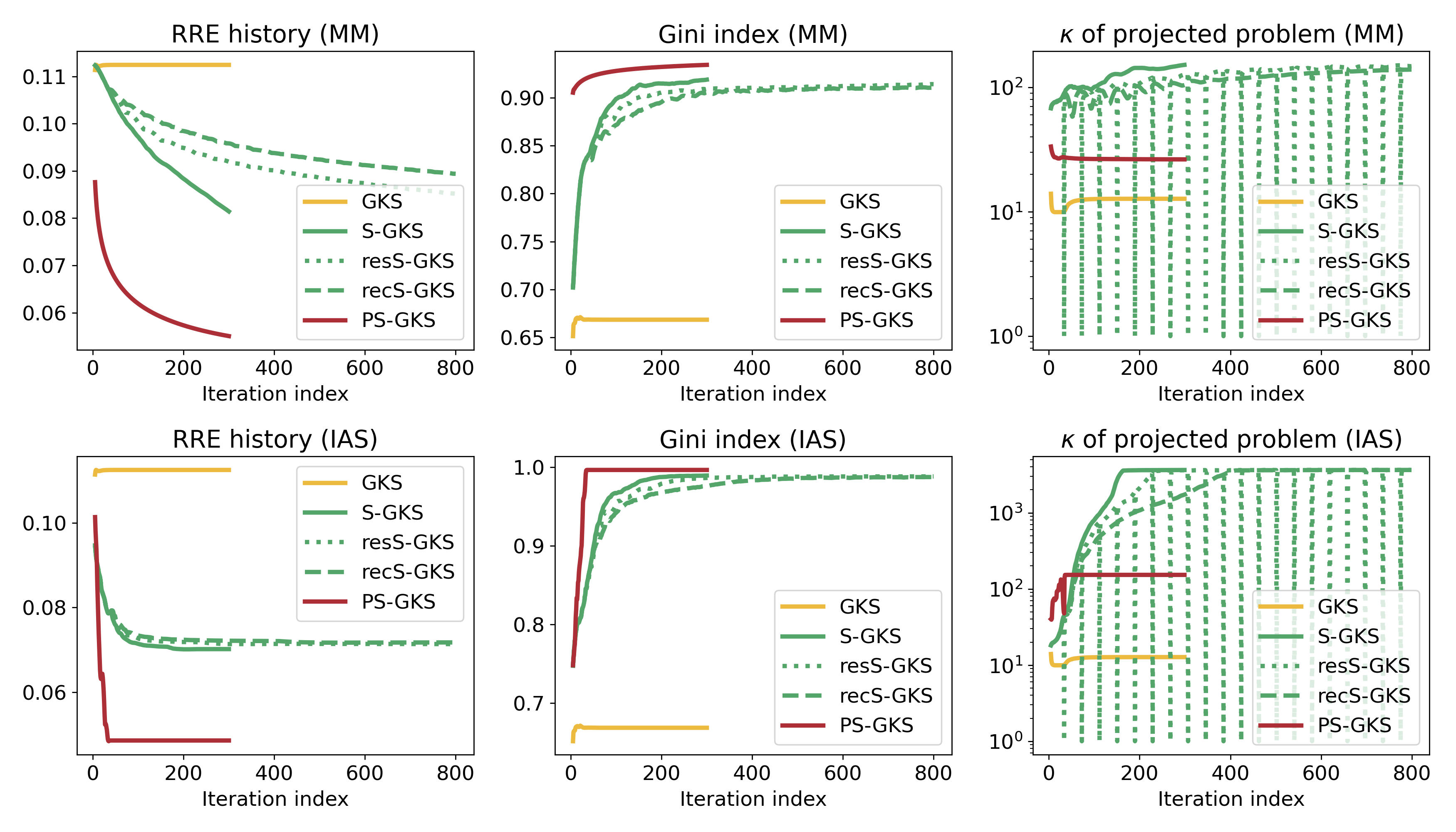}
    \caption{Performance comparison of the GKS (with equal weights), S-GKS (including restarting and recycling), and the proposed PS-GKS methods for the 1D cosine problem with MM (top row) and IAS (bottom row) weights. 
    Reported are the RRE (first column), the Gini index (measuring sparsity) of $\boldsymbol{\Psi} \mathbf{x}$ (second column), and the condition number of the projected least squares problem  (third column).
    }
\label{fig:1d_cosine_results_mm_and_ias_performance}
\end{figure}

\cref{fig:1d_cosine_results_mm_and_ias,fig:1d_cosine_results_mm_and_ias_performance} show the reconstructions and performance results for the GKS, S-GKS, and our new PS-GKS method, which we describe in detail in \Cref{sec:proposed}. 
Each method is run for 300 iterations. 
Additionally, we also show the results of resS-GKS with $D_{\text{max}} = 40$ and recS-GKS with $D_{\text{max}} = 40$ and $D_{\text{min}} = 20$, each run for 800 iterations. 
We observe in \cref{fig:1d_cosine_results_mm_and_ias_performance} that GKS converges rapidly to a smooth solution. 
At the same time, we observe that the S-GKS method fails to recover each of the discontinuities and does not appear to converge with either weight formulation. 
This might be explained by the condition number $\kappa$ of the projected least squares problem of S-GKS blowing up as a sparser solution is found. 
We provide a theoretical explanation for this phenomenon in  \Cref{sub:analysis_of_priorconditioning}. 

Although restarting and recycling make it possible to perform more iterations, res/recS-GKS still does not provide reconstructions competitive with our new PS-GKS method.


\section{Proposed method}
\label{sec:proposed} 

We observed above that neither S-GKS nor its restarted or recycled variants yielded satisfactory results for reconstructing a piecewise constant signal. 
To overcome this limitation, we introduce a new projection method that constructs prior conditioned generalized Krylov subspaces within a transformed space defined by the sparsifying transformation $\boldsymbol{\Psi}$.  
The resulting method, which we call \emph{priorconditioned S-GKS (PS-GKS)}, incurs a higher computational cost per iteration than S-GKS due to operations involving the precondition. 
However, this additional expense is offset by the method’s ability to produce sparser and more accurate solutions within an approximation subspace of modest dimension.

\subsection{Priorconditioning for the full-scale problem}
\label{sub:priorconditioning_full_scale}

 First, we review the priorconditioning technique applied to a single full-scale least squares problem in \cref{eq:IRLS}. 
 The crux of the technique is to seek a transformation under which the solution can be expressed in terms of a Tikhonov problem with regularization transformation equal to the identity. 
 Such a transformation performs a whitening by the prior and hence is referred to as \emph{priorconditioning}. 

It is well known that a transformation satisfying this property is provided by the standard form transformation for least squares problems \cite{elden1977algorithms, hansen2013oblique}: 
Let $\boldsymbol{\Psi}_{\itidx+1} = \mathbf{W}_{\itidx+1} \boldsymbol{\Psi}$ and $\mathbf{K} \in \mathbb{R}^{N \times P}$ be a full-rank matrix whose columns form an orthonormal basis for $\operatorname{ker}(\boldsymbol{\Psi})$.
Then, the standard form transformation applied to \cref{eq:IRLS} yields
\begin{equation}\label{eq:priorconditioning_transformation}
\resizebox{.92\textwidth}{!}{$\displaystyle 
    \argmin_{\mathbf{x} \in \mathbb{R}^N} \left\{ \| \mathbf{A} \mathbf{x} - \mathbf{b} \|_2^2 + \regparam \|  \boldsymbol{\Psi}_{\itidx+1} \mathbf{x} \|_2^2 \right\}  
    = (\boldsymbol{\Psi}_{\itidx+1} )_{\mathbf{A}}^\dagger \left( \argmin_{\mathbf{z} \in \mathbb{R}^K} \left\{ \| \overline{\mathbf{A}}_{\itidx+1} \mathbf{z} - \overline{\mathbf{b}} \|_2^2 + \regparam \| \mathbf{z} \|_2^2 \right\} \right) + \mathbf{x}_{\text{ker}},
$}
\end{equation}
where $\mathbf{x}_{\text{ker}} = \mathbf{K} (\mathbf{A} \mathbf{K})^\dagger \mathbf{b} \in \operatorname{ker}(\boldsymbol{\Psi})$, $\overline{\mathbf{A}}_{\itidx+1} = \mathbf{A} (\boldsymbol{\Psi}_{\itidx+1})_{\mathbf{A}}^\dagger$, $\overline{\mathbf{b}} = \mathbf{b} - \mathbf{A} \mathbf{x}_{\text{ker}}$, $(\boldsymbol{\Psi}_{\itidx+1})_{\mathbf{A}}^\dagger = \mathbf{E} \boldsymbol{\Psi}_{\itidx+1}^\dagger$, and $\mathbf{E} = \mathbf{I}_N - \mathbf{K} (\mathbf{A} \mathbf{K})^\dagger \mathbf{A}$. 
Furthermore, $(\boldsymbol{\Psi}_{\itidx+1})_{\mathbf{A}}^\dagger$ is known as the oblique ($\mathbf{A}$-weighted) pseudoinverse of $\boldsymbol{\Psi}_{\itidx+1}$; see \cite{hansen2013oblique}.  
We note that the singular values associated with the least squares problems appearing in \cref{eq:priorconditioning_transformation} may be very different; hence $(\boldsymbol{\Psi}_{\itidx+1} )_{\mathbf{A}}^\dagger$ functions as a precondition. 

\begin{remark}
Notable simplifications arise if $\boldsymbol{\Psi}^{-1}$ exists, in which case $(\boldsymbol{\Psi}_{\itidx+1})_{\mathbf{A}}^\dagger = \boldsymbol{\Psi}^{-1} \mathbf{W}_{\itidx+1}^{-1}$, or if $\boldsymbol{\Psi}$ has full column rank, in which case $(\boldsymbol{\Psi}_{\itidx+1})_{\mathbf{A}}^\dagger = \boldsymbol{\Psi}_{\itidx+1}^\dagger$. 
In either case, $\mathbf{x}_{\text{ker}} = \mathbf{0}_{N}$, eliminating the need for the matrix $\mathbf{K}$.  
Assuming $P > 0$, we observe that $\mathbf{A} \mathbf{K}$ forms a ``skinny'' full-rank matrix, whose pseudoinverse can be efficiently computed—even for large-scale problems—using the economic QR decomposition of $\mathbf{A} \mathbf{K}$. 
In contrast, computing $\boldsymbol{\Psi}_{\itidx+1}^\dagger$ is more demanding for large problems and may, itself, require an iterative approach. 
We discuss such efficient iterative approaches in \Cref{sec:pseudoinverses}.
\end{remark}

\subsection{Analysis of priorconditioning}
\label{sub:analysis_of_priorconditioning} 

We provide a theoretical analysis showing the benefits of priorconditioning in \cref{eq:priorconditioning_transformation} for sparsity-promoting inverse problems. 
To the best of our knowledge, such benefits have only been acknowledged heuristically in the literature \cite{calvetti2018bayes, uribe2023horseshoe, dong2023inducing, lindbloom2024generalized} except for \cite{nishimura2023prior}, where an analysis was provided for the case $\boldsymbol{\Psi} = \mathbf{I}_N$. 
Importantly, our investigation makes no assumptions on $\boldsymbol{\Psi}$ and permits it to be rank-deficient. 

Henceforth, let $\lambda_i(\mathbf{X})$ denote the $i$th largest eigenvalue of a square matrix $\mathbf{X}$ in descending order (counting multiplicities).
We begin by examining the suboptimal performance of the existing S-GKS method. 
By the Poincar\'{e} separation theorem \cite[Corollary 4.3.16]{johnson1985matrix}, the condition number of the projected least-squares problem in \cref{eq:sgks_projected_problem} cannot exceed that of the original (full-scale) problem in \cref{eq:IRLS}. 
However, the condition number of \cref{eq:IRLS} itself can be extremely large, especially when the weight vector $\mathbf{w}$ captures the sparsity pattern accurately. 
\cref{thm:original_eigenvalue_bounds} below sheds light on this phenomenon.

\begin{theorem}\label{thm:original_eigenvalue_bounds} 
    Let $\mathbf{Q}^{\text{st}}_{\regparam} =  \mathbf{A}^T \mathbf{A} + \regparam \boldsymbol{\Psi}^T \mathbf{W}^2 \boldsymbol{\Psi}$, $\mathbf{W} = \operatorname{diag}(\mathbf{w})$, and  $R = \operatorname{rank}(\boldsymbol{\Psi})$. 
    Then, the first $R$ largest eigenvalues of $\mathbf{Q}^{\text{st}}_{\regparam}$ satisfy 
    \begin{align}
        \lambda_N(\mathbf{A}^T \mathbf{A}) + \regparam \lambda_{R}(\boldsymbol{\Psi}^T \boldsymbol{\Psi}) \lambda_{i+(N-R)}(\mathbf{W}^2) 
            \leq \lambda_i( \mathbf{Q}^{\text{st}}_{\regparam} ) \leq \lambda_1(\mathbf{A}^T \mathbf{A}) + \mu \lambda_1( \boldsymbol{\Psi}^T \boldsymbol{\Psi}) \lambda_i(\mathbf{W}^2)
    \end{align}
    for $i = 1, \ldots, R$, and the remaining $N-R$ eigenvalues satisfy 
    \begin{align}
        \lambda_N(\mathbf{A}^T \mathbf{A}) \leq \lambda_i(\mathbf{Q}^{\text{st}}_{\regparam}) \leq \lambda_1(\mathbf{A}^T \mathbf{A}) + \regparam \lambda_1(\boldsymbol{\Psi}^T \boldsymbol{\Psi}) \lambda_i(\mathbf{W}^2)
    \end{align}
    for $i = R+1, \ldots, N$.
\end{theorem}

\begin{proof} 
    The statement follows from standard bounds on the eigenvalues of sums of symmetric matrices and a generalization of Ostrowski’s theorem (see \cref{thm:ostrowski}).  
\end{proof}

\cref{thm:original_eigenvalue_bounds} implies the following lower bound for the condition number of $\mathbf{Q}_{\regparam}^{\text{st}}$:
\begin{equation}\label{eq:cond_number_Qst}
    \kappa\left( \mathbf{Q}_{\regparam}^{\text{st}} \right) 
    \geq \frac{
        \lambda_N(\mathbf{A}^T \mathbf{A}) + \regparam \lambda_{N}(\boldsymbol{\Psi}^T \boldsymbol{\Psi}) \lambda_1(\mathbf{W}^{2}) 
    }{
        \lambda_1(\mathbf{A}^T \mathbf{A}) + \regparam \lambda_{1}(\boldsymbol{\Psi}^T \boldsymbol{\Psi}) \lambda_N(\mathbf{W}^{2})
    },
\end{equation} 
assuming $\operatorname{rank}(\boldsymbol{\Psi}) = N$.
The lower bound \cref{eq:cond_number_Qst} reveals the dependency of the condition number on the scaling in $\mathbf{W}^2$ and the regularization parameter $\regparam$.
Consider the typical situation where $\mathbf{W} = \diag{\mathbf{w}}$ accurately encodes the sparsity profile, i.e., $w_k \approx 0$ if $[\boldsymbol{\Psi} \mathbf{x}_{\text{truth}}]_k = 0$ (which is true for most of the entries) and $w_k \gg 0$ otherwise.
In this case, $\lambda_1(\mathbf{W}^{2}) \gg \lambda_N(\mathbf{W}^{2})$ and the right hand side of \cref{eq:cond_number_Qst} implies a detremantally large condition number for $\mathbf{Q}_{\regparam}^{\text{st}}$ and the S-GKS method, provided that $\mu$ is not too small. 

In contrast, we next show in \cref{thm:main_sparsity_theorem} that the normal equations for the priorconditioned formulation of the RHS of \cref{eq:priorconditioning_transformation} enjoys a clustered spectrum.

\begin{theorem}\label{thm:main_sparsity_theorem} 
    Let $\mathbf{Q}_{\regparam}^{\text{pr}} = \overline{\mathbf{A}}^T \overline{\mathbf{A}} + \mu \mathbf{I}_K$ with  $\overline{\mathbf{A}} = \mathbf{A} (\mathbf{W} \boldsymbol{\Psi})_{\mathbf{A}}^\dagger$ where  $(\mathbf{W} \boldsymbol{\Psi})_{\mathbf{A}}^\dagger = \mathbf{E} (\mathbf{W} \boldsymbol{\Psi})^\dagger$ is the oblique ($\mathbf{A}$-weighted) pseudoinverse of $\mathbf{W} \boldsymbol{\Psi}$ as in \Cref{sub:priorconditioning_full_scale}, and let $R = \operatorname{rank}(\boldsymbol{\Psi})$. 
    Then, the first $R$ largest eigenvalues of $\mathbf{Q}_{\regparam}^{\text{pr}}$ satisfy 
    \begin{align}
        \mu \leq \lambda_i(\mathbf{Q}_{\regparam}^{\text{pr}}) \leq \mu +  \min\left\{ c_1 \lambda_i(\mathbf{A}^T \mathbf{A}), c_2   \lambda_i(\mathbf{W}^{-2}) \right\}
    \end{align}
    for $i = 1, \ldots, R$, where $c_1 = \lambda_1(\mathbf{W}^{-2})/\lambda_R(\boldsymbol{\Psi}^T \boldsymbol{\Psi})$ and $c_2 = \lambda_1(\mathbf{A}^T \mathbf{A})/\lambda_{R}(\boldsymbol{\Psi}^T \boldsymbol{\Psi})$ are constants independent of $i$. 
    Moreover, the remaining $K-R$ eigenvalues of $\mathbf{Q}_{\mu}^{\text{pr}}$ are all equal to $\mu$.
\end{theorem}

\begin{proof}
    See \Cref{app:proof_eigenvalue_thm}. 
\end{proof}

\cref{thm:main_sparsity_theorem} implies the following upper bound for the condition number of $\mathbf{Q}_{\mu}^{\text{pr}}$:
\begin{align}\label{eq:kappa_pr_upper_bound}
    \kappa(\mathbf{Q}_{\mu}^{\text{pr}}) \leq 1 + \frac{\lambda_1(\mathbf{A}^T \mathbf{A}) \lambda_1(\mathbf{W}^{-2}) }{ \regparam \lambda_R(\boldsymbol{\Psi}^T \boldsymbol{\Psi}) }.
\end{align}
Comparing \cref{eq:kappa_pr_upper_bound} with \cref{eq:cond_number_Qst}, we see that the linear system resulting from priorconditioning is typically significantly better conditioned than the original one. 
In particular, the upper bound \cref{eq:kappa_pr_upper_bound} for priorconditioned linear systems is independent of the relative scaling of the weights $\mathbf{w}$ and improves for increasing $\regparam$. 
This observation indicates that the proposed priorconditioning is particularly advantageous for strongly sparsity-promoting approaches. 

It is important to note that not only the condition number but also the clustering of eigenvalues will significantly influence the performance of iterative methods such as CG. 

This is due to the polynomial best approximation property of CG \cite[Lemma 3.14]{ciaramella2022iterative}, which states for a generic system $\mathbf{Q} \mathbf{u} = \mathbf{v}$ that the $\itidx$th iteration satisfies 
\begin{align}\label{eq:cg_best_approximation}
    \| \mathbf{u} - \mathbf{u}_{\itidx} \|_{\mathbf{Q}} = \min_{ \substack{p \in \mathcal{P}_{\itidx} \\ p(0) = 1} }  \| p(\mathbf{Q}) (\mathbf{u} - \mathbf{u}_0) \|_{\mathbf{Q}},
\end{align}
where $\mathcal{P}_{\itidx}$ denotes the set of polynomials of degree at most $\itidx$. That is, CG implicitly fits a polynomial to the spectrum of $\mathbf{Q}$. 
The resulting polynomial best approximation is expected to be more accurate when the spectrum of $\mathbf{Q}$ is clustered. Hence, better clustering leads to accelerated convergence for CG with fewer basis vectors. 
From \cref{thm:main_sparsity_theorem}, the spectrum of $\mathbf{Q}_{\regparam}^{\text{pr}}$ is comprised of $R$ eigenvalues decaying to $\regparam$ at least as rapidly as  $\regparam + c_1 \lambda_i(\mathbf{A}^T \mathbf{A})$ or $\regparam + c_2 \lambda_i(\mathbf{W}^{-2})$ (whichever is faster), along with the eigenvalue $\regparam$ repeated $K-R$ times. 
Consequently, $\mathbf{Q}_{\regparam}^{\text{pr}}$ can have at most $\operatorname{rank}(\mathbf{A})+1$ distinct eigenvalues. 
Furthermore, if the weights $\mathbf{w}$ are close to encoding a sparse solution with $S$ ``small'' components (identifying the support) and $K-S$ ``large'' components, then we expect the spectrum to exhibit at most $S$+1 clusters. 
In summary, our above analysis indicates that priorconditioning becomes particularly effective when (i) the weights strongly promote sparsity, (ii) the singular values of $\mathbf{A}$ decay fast, or (iii) the regularization parameter $\regparam$ is increased. This should be compared to \cref{thm:original_eigenvalue_bounds}, which  indicates that the eigenvalues of $\mathbf{Q}_{\regparam}^{\text{st}}$ span a wide range of values with no particular clustering.

\subsection{Projected IRLS via priorconditioned generalized Krylov subspaces}
\label{sub:gks_main}

We now introduce our new PS-GKS method, which incorporates priorconditioned generalized Krylov subspaces into the projected IRLS scheme in \cref{eq:projected_IRLS}. 
The main idea is to replace \cref{eq:projected_IRLS} with 
\begin{equation}\label{eq:priorconditioned_projected_IRLS}
    \bx_{\itidx+1} = \mathbf{x}_{\text{ker}} + (\boldsymbol{\Psi}_{\itidx+1})_{\mathbf{A}}^\dagger \left(  \argmin_{\bz \in \mathcal{V}_\itidx} \left\{ 
        \norm{ \overline{\bA}_{\itidx+1} \bz - \bb }_2^2 
        + \regparam_{\itidx+1} \norm{  \bz }_2^2
    \right\} \right), 
\end{equation}
where $\{  \mathcal{V}_{\itidx} \}_{\itidx \geq 0}$ is a nested sequence of low-dimensional subspaces. A key distinction of our proposed PS-GKS from existing S-GKS methods is that we define the subspace in the $K$-dimensional space of ``increments'' rather than in the $N$-dimensional native space of $\bx$.
Consequently, we refer to these subspaces as \emph{priorconditioned generalized Krylov subspaces}.

Our proposed PS-GKS method operates as follows: 
We begin by selecting an initialization $\mathbf{x}_0 \in \mathbb{R}^{N}$. 
We then compute the associated weight matrix $\mathbf{W}_0$ and weighted sparsifying transformation $\boldsymbol{\Psi}_0 = \mathbf{W}_0 \boldsymbol{\Psi}$.
Next, we determine the contribution to the solution from $\ker(\boldsymbol{\Psi})$ as $\mathbf{x}_{\text{ker}} = \mathbf{K}(\mathbf{A} \mathbf{K})^\dagger \mathbf{b}$ and define the pseudoinverses $\boldsymbol{\Psi}_0^\dagger$ and  $(\boldsymbol{\Psi}_0)_{\mathbf{A}}^\dagger$. 
Letting  $\overline{\mathbf{A}}_\itidx \coloneqq \mathbf{A} (\boldsymbol{\Psi}_{\itidx})_{\mathbf{A}}^\dagger$, we generate an initial subspace $\mathcal{V}_0 = \mathcal{K}_{h}( \overline{\mathbf{A}}_0^T \overline{\mathbf{A}}_0, \overline{\mathbf{A}}_0^T \overline{\mathbf{b}}) \subset \mathbb{R}^K$ for a relatively small $h$, e.g., $h = 5$. 
(See \cref{rem:incorporating_x0_priorconditioned} for more details on choosing the initial subspace.).

At the $(\itidx+1)$th iteration of the PS-GKS method, we begin by computing the new weight matrix $\mathbf{W}_{\itidx+1}$ and the weighted sparsifying transformation $\boldsymbol{\Psi}_{\itidx+1} = \mathbf{W}_{\itidx+1} \boldsymbol{\Psi}$---as in the S-GKS approach. 
We then project the stand form IRLS subproblem on the right-hand side of  \cref{eq:priorconditioning_transformation} onto the lower-dimensional subspace $\mathcal{V}_\itidx$ to obtain the projected problem 
\begin{align}\label{eq:priorconditioned_projected_problem}
    \argmin_{\mathbf{z} \in \mathcal{V}_\itidx} \left\{ \| \overline{\mathbf{A}}_{\itidx+1} \mathbf{z} - \overline{\mathbf{b}} \|_2^2 + \regparam_{\itidx+1} \| \mathbf{z} \|_2^2 \right\}.
\end{align}

To solve \cref{eq:priorconditioned_projected_problem}, we insert the economic QR factorization $\overline{\mathbf{A}}_{\itidx+1} \mathbf{V}_\itidx = \mathbf{Q}_{\overline{\mathbf{A}}} \mathbf{R}_{\overline{\mathbf{A}}}$, yielding 
\begin{align}\label{eq:qr_factored_priorconditioned_projected_problem}
    \argmin_{\mathbf{u} \in \mathbb{R}^{D_\itidx} } \left\{ \| \mathbf{R}_{\overline{\mathbf{A}}} \mathbf{u} - \mathbf{Q}_{\overline{\mathbf{A}}}^T \overline{\mathbf{b}} \|_2^2 + \regparam_{\itidx+1} \| \mathbf{u} \|_2^2  \right\}.
\end{align}
After selecting $\regparam_{\itidx+1}$ using a regularization parameter selection method (see \Cref{sub:dp_priorconditioned_projected}), we recover the full-scale solution to the transformed problem as $\mathbf{z}_{\itidx+1} = \mathbf{V}_{\itidx} \mathbf{u}_{\itidx+1}$. 
Furthermore, we get an estimate of the solution to the original problem as $\mathbf{x}_{\itidx+1} = (\boldsymbol{\Psi}_{\itidx+1})_{\mathbf{A}}^\dagger \mathbf{z}_{\itidx+1} + \mathbf{x}_{\text{ker}}$. 
The remaining steps in the iteration follow that of the S-GKS method. 
The main difference is that the subspace enlargement is performed in the transformed space. 
Specifically, the residual vector is $\mathbf{r}_{\itidx+1} = \overline{\mathbf{A}}_{\itidx+1}^T ( \overline{\mathbf{A}}_{\itidx+1} \mathbf{V}_{\itidx} \mathbf{z}_{\itidx+1} - \overline{\mathbf{b}}) + \regparam_{\itidx+1} \mathbf{V}_{\itidx} \mathbf{z}_{\itidx+1}$ for the proposed PS-GKS method. 
Finally, the PS-GKS iteration is repeated until some convergence criterion is satisfied. 
\cref{alg:ps-gks} summarizes the proposed PS-GKS algorithm. 

\begin{algorithm}[h]
\caption{The PS-GKS method}%
\label{alg:ps-gks}
\begin{algorithmic}[1]
	\Require{$\bA, \boldsymbol{\Psi}, \mathbf{b}, \bx_0, \mathbf{K}$} 
    \Ensure{An approximate solution $\bx_{\itidx+1}$}
	\Function{$\bx_{\itidx+1} = $ PS-GKS }{$\bA, \boldsymbol{\Psi}, \mathbf{b}, \bx_{0}, \mathbf{K}$} \;
        \noindent \State $\mathbf{A} \mathbf{K} = \mathbf{Q}_{\text{ker}} \mathbf{R}_{\text{ker}}$ and $(\mathbf{A} \mathbf{K})^\dagger = \mathbf{R}_{\text{ker}}^{-1} \mathbf{Q}_{\text{ker}}^T$\;\Comment{$(\mathbf{A} \mathbf{K})^\dagger$ via economic QR}
        \State $\mathbf{x}_{\text{ker}} = \mathbf{K} \mathbf{R}_{\text{ker}}^{-1} \mathbf{Q}_{\text{ker}}^T \mathbf{b}$ and $\overline{\mathbf{b}} = \mathbf{b} - \mathbf{A} \mathbf{x}_{\text{ker}}$\;\Comment{Fixed component in $\ker(\boldsymbol{\Psi})$}
	  \State Generate the initial subspace basis $\bV_{0}\in \R^{K \times D_0}$ such that $\bV_{0}^{T}\bV_{0}=\bI_{D_0}$\;
     \FOR {$\itidx=0,1,2,\ldots$ \text{until convergence}}{
\State Update weights $\mathbf{W}_{\itidx+1} = \operatorname{diag}(\bw_{\itidx+1})$ and $\boldsymbol{\Psi}_{\itidx+1} = \mathbf{W}_{\itidx+1} \boldsymbol{\Psi}$ given $\boldsymbol{\Psi} \mathbf{x}_\itidx$
\State Build operators for  $\boldsymbol{\Psi}_{\itidx+1}^\dagger$, $( \boldsymbol{\Psi}_{\itidx+1})_{\mathbf{A}}^\dagger = ( \mathbf{I}_N - \mathbf{K}(\mathbf{A} \mathbf{K})^\dagger \mathbf{A} ) \boldsymbol{\Psi}_{\itidx+1}^\dagger$, and $\overline{\mathbf{A}}_{\itidx+1} = \mathbf{A} ( \boldsymbol{\Psi}_{\itidx+1})_{\mathbf{A}}^\dagger$\;
\State	$\overline{\mathbf{A}}_{\itidx+1} \mathbf{V}_{\itidx} = \mathbf{Q}_{\overline{\mathbf{A}}} \mathbf{R}_{\overline{\mathbf{A}}}$\;\Comment{Compute economic QR}
\State Select $\regparam_{\itidx+1}$ by a heuristic (e.g., DP) on \cref{eq:qr_factored_priorconditioned_projected_problem}\;\Comment{Regularization parameter selection} 
\State $\mathbf{u}_{\itidx+1}$ to satisfy \cref{eq:qr_factored_priorconditioned_projected_problem} with selected $\regparam_{\itidx+1}$\;\Comment{Solve projected problem}
\State $\bx_{\itidx+1} =  (\boldsymbol{\Psi}_{\itidx+1})_{\mathbf{A}}^\dagger  \bV_{\itidx}\mathbf{u}_{\itidx+1} + \mathbf{x}_{\text{ker}}$\;\Comment{Full-scale solution via projection}
\State $\br_{\itidx+1}=\overline{\mathbf{A}}_{\itidx+1}^T( \overline{\mathbf{A}}_{\itidx+1} \bV_{\itidx} \mathbf{u}_{\itidx+1} - \overline{\mathbf{b}} )+\regparam_{\itidx+1} \mathbf{V}_{\itidx} \mathbf{u}_{\itidx+1}$\;\Comment{Full-scale residual}
\State		 $\br_{\itidx+1} = \br_{\itidx+1} - \bV_{\itidx}\bV_{\itidx}^{T}\br_{\itidx+1}$\; \Comment{Reorthogonalize (optional)}
\State		 $\bv_{\rm new} = \frac{\br_{\itidx+1}}{\|\br_{\itidx+1}\|_{2}}$: $\bV_{\itidx+1}=[\bV_{\itidx}, \bv_{\rm new}]$\;\Comment{Enlarge the solution subspace}}
\ENDFOR
\EndFunction
\end{algorithmic}
\end{algorithm}

\begin{remark}[Computational costs of PS-GKS]\label{rem:PSGKS_costs}
    Here, we compare the computational costs of existing S-GKS strategies and our proposed PS-GKS methods. 
    For the S-GKS approach discussed in \Cref{sub:sgks}, obtaining the economic QR factorization of $\mathbf{A} \mathbf{V}_\itidx$, which is required to formulate \cref{eq:l2_gks_qr_factored_problem}, requires a single matvec with $\mathbf{A}$ and $\mathcal{O}( D_\itidx M )$ additional flops. 
    At the same time, building the QR factorization of $\boldsymbol{\Psi}_{\itidx+1} \mathbf{V}_\itidx$ requires $D_\itidx$ matvecs with $\boldsymbol{\Psi}$ and $\mathcal{O}( D_\itidx^2 K)$ additional flops---we have an additional factor $D_\itidx$ because $\boldsymbol{\Psi}_{\itidx+1}$ changes at each iteration.\footnote{We assume $M, K \sim \mathcal{O}(N)$. 
    The matrix $\mathbf{A} \mathbf{V}_\itidx$ may be obtained using $\mathbf{A} \mathbf{V}_{\itidx-1}$ and a single matvec with $\mathbf{A}$. Similarly, the economic QR factorization of $\mathbf{A} \mathbf{V}_{\itidx}$ can be obtained efficiently as column update of $\mathbf{A} \mathbf{V}_{\itidx-1}$.} 
    At the same time, the computational cost per PS-GKS iteration is $\mathcal{O}( D_\itidx^2 \operatorname{max}(M, K) )$ flops---ignoring the cost of matvecs with $\mathbf{A}$ and $\boldsymbol{\Psi}_{\itidx}^\dagger$. 
    Furthermore, the memory requirement to perform $\itidx$ iterations of PS-GKS is the storage of $\mathcal{O}( D_\itidx \operatorname{max}(M,K) )$ floating point numbers. 
    Each iteration of PS-GKS requires $\mathcal{O}(D_\itidx)$ matvecs with both  $\mathbf{A}/\mathbf{A}^T$ and  $\boldsymbol{\Psi}_{\itidx}^\dagger/(\boldsymbol{\Psi}_{\itidx}^\dagger)^T$. 
\end{remark}

\begin{remark}[Incorporating $\mathbf{x}_0$ into $\mathcal{V}_0$]\label{rem:incorporating_x0_priorconditioned} 
    Although the standard initial Krylov subspace $\mathcal{K}_{h}( \overline{\mathbf{A}}_0^T \overline{\mathbf{A}}_0^T, \overline{\mathbf{A}}_0^T \overline{\mathbf{b}})$ incorporates information about $\mathbf{w}_0$, it may neglect further available information contained in $\mathbf{x}_0$. 
    To incorporate this information, we compute the vector $\mathbf{z}_0 = \boldsymbol{\Psi}_0 \mathbf{x}_0$ and take our initial subspace to be $\mathcal{V}_0 = \mathcal{K}_{\ell}( \overline{\mathbf{A}}_0^T \overline{\mathbf{A}}_0, \overline{\mathbf{A}}_0^T \overline{\mathbf{b}}) \cup \operatorname{span}\{ \mathbf{z}_0 \}$. 
    It is straightforward to obtain an associated matrix $\mathbf{V}_0$ whose columns form an orthonormal basis for $\mathcal{V}_0$.
\end{remark}

\subsection{Regularization parameter selection}\label{sub:dp_priorconditioned_projected} 

There are various approaches for selecting regularization parameters \cite{vogel2002computational,hansen2010discrete}, including the discrepancy principle (DP)
\cite{morozov1984methods,reichel2008new}, generalized cross validation (GCV) \cite{golub1997generalized}, and the L-curve \cite{lawson1995solving}. 
We now provide a few details on how the DP can be used to select $\regparam_{\ell+1}$ in the priorconditioned projected problem \cref{eq:qr_factored_priorconditioned_projected_problem}. 
Since $\mathbf{x}_{\ell+1} = (\boldsymbol{\Psi}_{\ell+1})_{\mathbf{A}}^\dagger \mathbf{V}_{\itidx} \mathbf{u}_{\itidx+1} + \mathbf{x}_{\text{ker}}$ and $\mathbf{A} (\boldsymbol{\Psi}_{\itidx+1})_{\mathbf{A}}^\dagger \mathbf{V}_{\itidx} = \mathbf{Q}_{\overline{\mathbf{A}}} \mathbf{R}_{\overline{\mathbf{A}}}$, the DP rule for \cref{eq:qr_factored_priorconditioned_projected_problem} is to select $\regparam_{\itidx+1}$ as the root of
\begin{subequations}\label{eq:priorconditioned_varphi}
\begin{align}
    \varphi(\regparam) &= \| \mathbf{A}( (\boldsymbol{\Psi}_{\itidx+1})_{\mathbf{A}}^\dagger \mathbf{V}_{\itidx} \mathbf{u}_{\itidx+1}^{(\regparam)} + \mathbf{x}_{\text{ker}} )  - \mathbf{b} \|_2^2 - \tau^2 \| \mathbf{e} \|_2^2 \\
    &= \| \mathbf{A} (\boldsymbol{\Psi}_{\itidx+1})_{\mathbf{A}}^\dagger \mathbf{V}_{\itidx} \mathbf{u}_{\itidx+1}^{(\regparam)}  - \overline{\mathbf{b}} \|_2^2 - \tau^2 \| \mathbf{e} \|_2^2 \\
    &= \| \mathbf{R}_{\overline{\mathbf{A}}} \mathbf{u}^{(\regparam)}_{\itidx+1} - \mathbf{Q}_{\overline{\mathbf{A}}}^T  \overline{\mathbf{b}} \|_2^2 + \|(\mathbf{I}_M - \mathbf{Q}_{\overline{\mathbf{A}}} \mathbf{Q}_{\overline{\mathbf{A}}}^T ) \overline{\mathbf{b}} \|_2^2 - \tau^2 \| \mathbf{e} \|_2^2,
\end{align}
\end{subequations}
where $\mathbf{u}^{(\regparam)}_{\itidx+1}$ denotes the solution to \cref{eq:qr_factored_priorconditioned_projected_problem}
for fixed $\regparam$. 
Note that an estimate of $\| \mathbf{e} \|_2$ may not be available in practice.
Hence, we adopt the approximation $\| \mathbf{e} \|_2 \approx \sqrt{M}$. 
With the change of variables $\beta = \regparam^{-1}$, it is possible to find a condition that guarantees the existence and uniqueness of a root of $\psi(\beta) \coloneqq \varphi(\beta^{-1})$ and that Newton's method initialized to the left of the root (e.g., $\beta_0 = 0$) converges. 
Such results have been given in \cite{calvetti2003tikhonov,reichel2008new,hochstenbach2010iterative,gazzola2020krylov}. \cref{thm:root_theorem} below is derived from these existing results. 
While the results in the existing literature typically assume $\boldsymbol{\Psi} = \mathbf{I}$, \cref{thm:root_theorem} applies to any $\mathbf{A}$ and $\boldsymbol{\Psi}$ with $\ker(\mathbf{A}) \cap \ker(\boldsymbol{\Psi}) = \{ \mathbf{0}\}$.

\begin{theorem}\label{thm:root_theorem}
    Let $\mathbf{x}_{\beta}$ denote the solution to 
    \begin{align}\label{eq:root_thm_least_squares}
        \argmin_{\mathbf{x} \in \mathbb{R}^N} \,\, \| \mathbf{A} \mathbf{x} - \mathbf{b} \|_2^2 + \beta^{-1} \| \boldsymbol{\Psi} \mathbf{x} \|_2^2.
    \end{align}
    Then, $\psi(\beta) \coloneqq \| \mathbf{A} \mathbf{x}_{\beta} - \mathbf{b}  \|_2^2 - z$ is strictly decreasing and convex, and has a unique positive root so long as $\| (\mathbf{I}_M - \mathbf{A} \mathbf{A}^\dagger) \mathbf{b} \|_2^2 \leq z \leq \| (\mathbf{I}_M - \mathbf{A} \mathbf{K} (\mathbf{A} \mathbf{K})^\dagger) \mathbf{b} \|_2^2$.
\end{theorem}

\begin{proof}
    The proof is similar to that of \cite[Theorem 2.1]{calvetti2003tikhonov} and is thus omitted.
\end{proof}

Although \cref{thm:root_theorem} is expressed for the full-scale problem without priorconditioning, it is straightforward to derive an analogous result for the priorconditioned projected problem \cref{eq:qr_factored_priorconditioned_projected_problem} by making the substitutions $\mathbf{A} \leftarrow \mathbf{R}_{\overline{\mathbf{A}}}$, $\mathbf{b} \leftarrow \mathbf{Q}_{\overline{\mathbf{A}}}^T \overline{\mathbf{b}}$, $\boldsymbol{\Psi} \leftarrow \mathbf{I}_{D_\itidx}$, and $z \leftarrow \tau^2 \| \mathbf{e} \|_2^2 - \| (\mathbf{I} - \mathbf{Q}_{\overline{\mathbf{A}}} \mathbf{Q}_{\overline{\mathbf{A}}}^T) \overline{\mathbf{b}} \|_2^2$. 
This yields the condition $\| (\mathbf{I}_{D_{\itidx}} - \mathbf{R}_{\overline{\mathbf{A}}} \mathbf{R}_{\overline{\mathbf{A}}}^\dagger) \mathbf{Q}_{\overline{\mathbf{A}}}^T \overline{\mathbf{b}} \|_2^2 \leq \tau^2 \| \mathbf{e} \|_2^2 - \| (\mathbf{I}_M - \mathbf{Q}_{\overline{\mathbf{A}}} \mathbf{Q}_{\overline{\mathbf{A}}}^T) \overline{\mathbf{b}}  \|_2^2 \leq \| \overline{\mathbf{b}}  \|_2^2$
for guaranteeing the existence and uniqueness of the root, which can be efficiently found using a third-order root finder \cite{reichel2008new}.

\subsection{Restarting and recycling PS-GKS}\label{sub:restarting_recycling_psgks} 

As discussed in \cref{rem:restarting}, the storage requirements of performing many iterations of S-GKS or PS-GKS can easily exceed the memory capacity in some applications. 
Furthermore, for PS-GKS, we have to perform $\mathcal{O}(D_\itidx)$ additional matvecs with $\mathbf{A}$ and $(\boldsymbol{\Psi}_{\itidx+1})_{\mathbf{A}}^\dagger$ at each iteration, which can increase the computational costs for large $\itidx$. For these reasons, we proposed to combine the PS-GKS method with restarting/recycling strategies---similar to those described in \cref{rem:restarting} for existing S-GKS methods. The resulting restarted PS-GKS (resPS-GKS) method has memory requirements and computational costs of $\mathcal{O}(  D_\itidx K  )$ and $\mathcal{O}( D_\itidx^2 M )$ flops per iteration, respectively, where $D_{\itidx} = 1 + \itidx \mod D_{\text{max}}$. 

At the same time, the resulting recycling PS-GKS (recPS-GKS) method has memory requirements and computational costs of $\mathcal{O}(  D_\itidx K )$ and $\mathcal{O}( D_\itidx^2 M)$ flops per iteration, respectively, where $D_\itidx = D_{\text{min}} + \itidx  \mod (D_{\text{max}}  +1)$.

Several strategies  may be used to perform the basis compression step in recycled PS-GKS. 
In our implementation, we consider a truncated SVD (tSVD) method based on the matrix $\mathbf{H}_{\itidx+1} = [ \mathbf{R}_{\overline{\mathbf{A}}}^T, \regparam_{\itidx+1}^{1/2} \mathbf{I}_{D_\itidx} ]^T$, which is inspired by \cite{jiang2021hybrid, pasha2023recycling}.  
Suppose we want to compress the basis $\mathbf{V}_{\itidx} \in \mathbb{R}^{K \times D_{\text{max}}}$. 
We begin by computing the rank-($D_{\text{min}}-1)$ truncated SVD  $\mathbf{H}_{\itidx+1} \approx \hat{ \mathbf{U} } \hat{ \mathbf{S} }  \hat{ \mathbf{W} }^T$ with $\hat{\mathbf{U}} \in \mathbb{R}^{2 D_{\text{max}} \times D_{\text{min}}-1 }$, $\hat{\mathbf{S}} \in \mathbb{R}^{D_{\text{min}} - 1 \times D_{\text{min}}  - 1 }$, and $\hat{\mathbf{W}} \in \mathbb{R}^{D_{\text{max}} \times D_{\text{min}}  - 1 }$. 
Next, we compute a new basis matrix $\tilde{\mathbf{V}} = \mathbf{V}_{\itidx} \hat{\mathbf{W}} \in \mathbb{R}^{K \times D_{\text{min}} -1}$. 

We then form $\tilde{\mathbf{z}} = ( \mathbf{z}_{\itidx+1} - \tilde{\mathbf{V}} \tilde{\mathbf{V}}^T \mathbf{z}_{\itidx+1} ) / \| \mathbf{z}_{\itidx+1} - \tilde{\mathbf{V}} \tilde{\mathbf{V}}^T \mathbf{z}_{\itidx+1} \|_2$ and replace the basis with $\mathbf{V}_{\itidx} = [\tilde{\mathbf{V}} \,\, \tilde{\mathbf{z}}] \in \mathbb{R}^{K \times D_{\text{min}}}$. Overall, the above compression routine costs $\mathcal{O}( D_{\text{max}}^3 +  D_{\text{max}}^2 K )$ flops, making it inexpensive compared to the computational cost of a single iteration of PS-GKS when the basis is large.


\section{Computational examples}
\label{sec:numerics}

Next, we examine the performance of the PS-GKS method in two reconstruction tasks. 
For each task, we consider the results obtained using both the MM (with $p = 1$) and IAS weights (with $r = -1, \beta = 1$ and $r = 1/2$, $\beta = 3.01$, corresponding to an approximation of  $\ell_{2/3}$-regularization).
\begin{enumerate}
\item In \Cref{sub:numerics_1d_cosine}, we revisit the 1D undersampled DCT problem in \Cref{sub:illustrative_example}. We provide a thorough comparison with other hybrid methods in the literature, showing the PS-GKS outperforms other methods for this problem. 
\item In \Cref{sub:numerics_tomography}, we apply PS-GKS to an ill-posed 2D tomography problem where sparsity is enforced in the anisotropic 2D gradient. In this experiment we only compare to a subset of existing methods, since most methods \emph{cannot} be applied here due to their limitation of requiring $\boldsymbol{\Psi}$ being invertible. We find that PS-GKS outperforms all other considered methods with either MM or IAS weights. 
\end{enumerate}

Unless otherwise specified, all comparisons to other methods using MM weights use the same MM weight formulations presented in the relevant literature (see \Cref{app:mm_weights}). 
A summary of all methods compared is given in \cref{tab:methods_summary}. 
We standardize the choice of regularization parameter selection method across all methods to be the DP rule with $\tau = 1.01$ (see \cref{sub:dp_priorconditioned_projected}) -- for the regularization parameter we set an upper bound $\mu_{\text{max}} = 10^{7}$, as well as a lower bound $\mu_{\text{min}} = 10^{-7}$ which is defaulted to should a root to the DP root-finding problem fail to be found. In our experiments, we assess the quality of the reconstructions using the relative residual error (RRE) and structural similarity index (SSIM) measures \cite{wang2004image}. 
We also assess the degree of sparsity of $\boldsymbol{\Psi} \mathbf{x}$ using the Gini index \cite{zonoobi2011gini}, which is a normalized measure of sparsity (in the interval $[0,1]$) with higher values indicating greater sparsity. 
For each method, we denote by $n_{\mathbf{A}}$, $n_{\boldsymbol{\Psi}}$, and $n_{\boldsymbol{\Psi}^\dagger}$ the total number of matvecs (and transpose matvecs) required with  $\mathbf{A}$, $\boldsymbol{\Psi}$, and $\boldsymbol{\Psi}_{\itidx+1}^\dagger$, respectively.
We refer to the ratio $\sigma_{\rm NL}=\|\be\|_{2}/\|\bA\bx\|_{2}$ as the noise level. 
\begin{table}[ht]
\centering
\scalebox{0.8}{
\begin{tabular}{|l|l|l|l|l|}
\hline
\textbf{Method} & \textbf{Description} & \textbf{Req. $\boldsymbol{\Psi}^{-1}$?} & \textbf{Weights}  & \textbf{Ref.}  \\ \hline
\textbf{GKS} &  \Cref{sec:sgks_algo} & No  & ---  & \cite{lampe2012large}  \\ \hline
\textbf{S-GKS} &  \Cref{sec:sgks_algo} & No & MM$_2$ & \cite{lanza2015generalized, huang2017majorization}  \\ \hline
\textbf{resS-GKS} & \Cref{sec:sgks_algo}, see \Cref{rem:restarting}    & No & MM$_1$ &  \cite{buccini2023limited}  \\ \hline
\textbf{recS-GKS} & \Cref{sec:sgks_algo}, see \Cref{rem:restarting}   &  No  & MM$_2$ &   \cite{pasha2023recycling}   \\ \hline
\textbf{PS-GKS} & \Cref{alg:ps-gks}  &  No  & MM$_3$ & Here  \\ \hline
\textbf{resPS-GKS} & \Cref{alg:ps-gks}, see \Cref{sub:restarting_recycling_psgks}  &  No  & MM$_3$  & Here    \\  \hline
\textbf{recPS-GKS} & \Cref{alg:ps-gks}, see \Cref{sub:restarting_recycling_psgks}  & No  & MM$_3$ &  Here   \\  \hline
\textbf{PS-GKB} &  \Cref{alg:psgkb}, see \Cref{ssub:gkb}  & No & MM$_3$ & Here \& \cite{gazzola2021flexible} \\ \hline
\textbf{FGK} & See  \Cref{ssub:fgk}  & No &  MM$_4$  & \cite{gazzola2021flexible}   \\ \hline
\textbf{FLSQR-I} & See \cite{chung2019flexible}   & Yes &  MM$_5$  & \cite{chung2019flexible} \\ \hline
\textbf{FLSQR-R} &  See \cite{chung2019flexible}  & Yes & MM$_5$  & \cite{chung2019flexible}   \\ \hline
\textbf{FLSQR-W} &  See \Cref{rem:golub_kahan_approaches}   & Yes & MM$_5$    & Here  \\ \hline
\textbf{IRW-FLSQR} & See \cite{gazzola2021iteratively}  &  Yes   &  MM$_3$  & \cite{gazzola2021iteratively}   \\ \hline
\end{tabular}}
\caption{Summary of the reconstruction methods considered in the numerical experiments.}
\label{tab:methods_summary}
\end{table}

\begin{remark}[Alternative Golub-Kahan approaches]\label{rem:golub_kahan_approaches} The S-GKS and PS-GKS methods both utilize GKS as the subspaces. 
Alternative Golub--Kahan approaches utilizing partial Golub--Kahan bidiagonalizations or flexible Golub--Kahan decompositions have also been proposed \cite{chung2019flexible,  gazzola2021iteratively,  gazzola2021flexible}. 
When $\boldsymbol{\Psi}^{-1}$ exists, we compare our results with those obtained from the FLSQR-I, FLSQR-R, and IRW-FLSQR methods. We also define an additional flexible method FLSQR-W which is the same as FLSQR-R except where the projected problems solved are regularized by $\regparam_{\itidx+1} \| \boldsymbol{\Psi}_{\itidx+1}  \mathbf{x}  \|_2^2$ instead of $\regparam_{\itidx+1} \|  \boldsymbol{\Psi} \mathbf{x}  \|_2^2$. When $\boldsymbol{\Psi}^{-1}$ does not exist, we also consider two additional methods (PS-GKB and FGK \cite{gazzola2021flexible}) which are further discussed in Appendix \ref{ssec:psgkb_and_fgk}.
\end{remark}

\subsection{Test 1: 1D undersampled cosine transform}\label{sub:numerics_1d_cosine}

We revisit the 1D numerical example considered earlier in \cref{sub:illustrative_example} to compare existing methods with our PS-GKS method. Results for all methods considered are shown in \cref{tab:1d_cosine_performance_results}, and records of the RRE, SSIM, and Gini index are shown for selected methods in \cref{fig:1d_cosine_performance}. 
The stopping criteria were set to perform attempt a maximum of 150 iterations for all methods.

\begin{table}[ht]
\centering
\scalebox{0.8}{
\begin{tabular}{llcrrrccrrr}
\toprule
Weights & Method & $\mu$ & $n_{\textrm{iter}}$ & RRE & SSIM & Gini index & $\kappa$ & $n_{\mathbf{A}}$ & $n_{\boldsymbol{\Psi}}$ & $n_{\boldsymbol{ \Psi }^{-1} }$ \\
\midrule
--- & GKS & $2\cdot10^{2}$ & $150$ & $0.112$ & $0.914$ & $0.669$ & $1\cdot10^{1}$ & $315$ & $455$ & $0$ \\
\hline
MM & S-GKS & $3\cdot10^{1}$ & $150$ & $0.076$ & $0.962$ & $0.862$ & $4\cdot10^{1}$ & $315$ & $456$ & $0$ \\
MM & resS-GKS & $2\cdot10^{2}$ & $150$ & $0.112$ & $0.914$ & $0.670$ & $1\cdot10^{1}$ & $318$ & $459$ & $0$ \\
MM & recS-GKS & $3\cdot10^{1}$ & $150$ & $0.092$ & $0.948$ & $0.855$ & $4\cdot10^{1}$ & $525$ & $666$ & $0$ \\
MM & PS-GKS & $5\cdot10^{1}$ & $150$ & $0.059$ & $0.973$ & $0.930$ & $3\cdot10^{1}$ & $12234$ & $1$ & $12235$ \\
MM & resPS-GKS & $5\cdot10^{1}$ & $150$ & $0.059$ & $0.973$ & $0.930$ & $3\cdot10^{1}$ & $3227$ & $153$ & $3378$ \\
MM & recPS-GKS & $5\cdot10^{1}$ & $150$ & $0.059$ & $0.973$ & $0.930$ & $3\cdot10^{1}$ & $3018$ & $1$ & $3019$ \\
MM & PS-GKB & $5\cdot10^{1}$ & $150$ & $0.059$ & $0.973$ & $0.930$ & $3\cdot10^{1}$ & $22651$ & $0$ & $22652$ \\
MM & FGK & $4\cdot10^{1}$ & $49$ & $0.086$ & $0.949$ & $0.928$ & $4\cdot10^{5}$ & $101$ & $101$ & $100$ \\
MM & FLSQR-I & $3\cdot10^{0}$ & $49$ & $0.096$ & $0.927$ & $0.789$ & $1\cdot10^{5}$ & $101$ & $0$ & $102$ \\
MM & FLSQR-R & $2\cdot10^{2}$ & $49$ & $0.104$ & $0.928$ & $0.778$ & $3\cdot10^{5}$ & $101$ & $0$ & $102$ \\
MM & FLSQR-W & $2\cdot10^{1}$ & $49$ & $0.099$ & $0.934$ & $0.855$ & $8\cdot10^{4}$ & $101$ & $0$ & $102$ \\
MM & IRW-FLSQR & $3\cdot10^{1}$ & $49$ & $0.083$ & $0.952$ & $0.901$ & $3\cdot10^{8}$ & $101$ & $0$ & $102$ \\
\hline
IAS & S-GKS & $5\cdot10^{6}$ & $150$ & $0.071$ & $0.969$ & $0.981$ & $3\cdot10^{3}$ & $315$ & $456$ & $0$ \\
IAS & resS-GKS & $1\cdot10^{6}$ & $150$ & $0.072$ & $0.968$ & $0.971$ & $1\cdot10^{3}$ & $318$ & $459$ & $0$ \\
IAS & recS-GKS & $2\cdot10^{5}$ & $150$ & $0.077$ & $0.964$ & $0.945$ & $5\cdot10^{2}$ & $525$ & $666$ & $0$ \\
IAS & PS-GKS & $1\cdot10^{7}$ & $150$ & $0.049$ & $0.985$ & $0.997$ & $2\cdot10^{2}$ & $12234$ & $1$ & $12235$ \\
IAS & resPS-GKS & $1\cdot10^{6}$ & $150$ & $0.049$ & $0.985$ & $0.996$ & $2\cdot10^{2}$ & $3227$ & $153$ & $3378$ \\
IAS & recPS-GKS & $1\cdot10^{7}$ & $150$ & $0.060$ & $0.982$ & $0.997$ & $2\cdot10^{2}$ & $3018$ & $1$ & $3019$ \\
IAS & PS-GKB & $1\cdot10^{7}$ & $150$ & $0.049$ & $0.985$ & $0.997$ & $2\cdot10^{2}$ & $22651$ & $0$ & $22652$ \\
IAS & FGK & $5\cdot10^{1}$ & $49$ & $0.127$ & $0.881$ & $0.698$ & $1\cdot10^{4}$ & $101$ & $101$ & $100$ \\
IAS & FLSQR-I & $7\cdot10^{1}$ & $49$ & $0.112$ & $0.914$ & $0.674$ & $8\cdot10^{5}$ & $101$ & $0$ & $102$ \\
IAS & FLSQR-R & $2\cdot10^{2}$ & $49$ & $0.112$ & $0.914$ & $0.669$ & $7\cdot10^{5}$ & $101$ & $0$ & $102$ \\
IAS & FLSQR-W & $2\cdot10^{1}$ & $49$ & $0.099$ & $0.934$ & $0.855$ & $8\cdot10^{4}$ & $101$ & $0$ & $102$ \\
IAS & IRW-FLSQR & $1\cdot10^{2}$ & $49$ & $0.111$ & $0.916$ & $0.687$ & $6\cdot10^{9}$ & $101$ & $0$ & $102$ \\
\bottomrule
\end{tabular}
}
\caption{Test 1: Performance comparison for the 1D cosine problem.}
\label{tab:1d_cosine_performance_results}
\end{table}

\begin{figure}[tb]
    \centering
    \includegraphics[width=.99\textwidth]{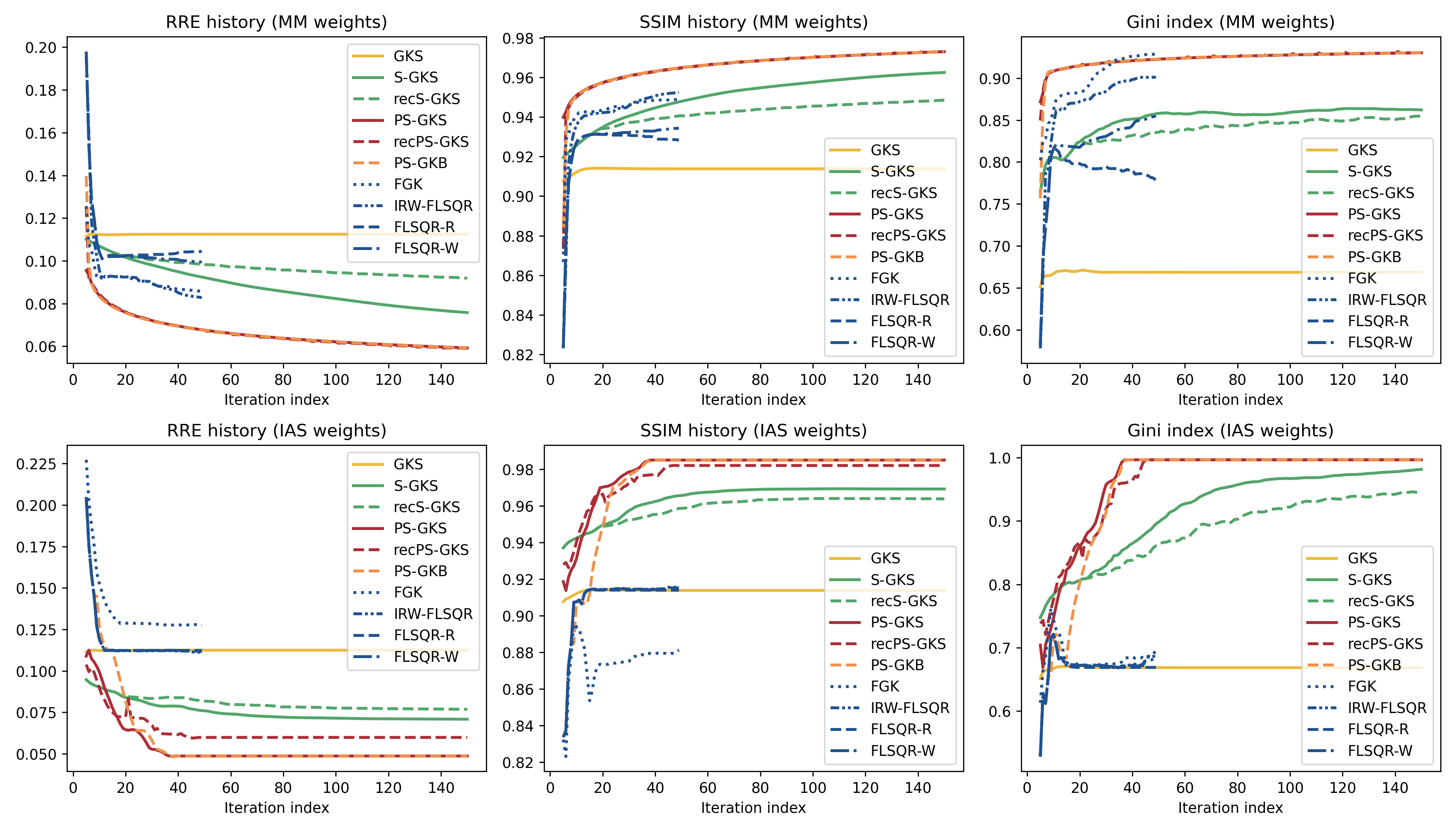}
    \caption{Test 1. The RRE, SSIM, and Gini index for MM (top row) and IAS (bottom row) weights.}
    \label{fig:1d_cosine_performance}
\end{figure}

Note that the methods based on the flexible Golub--Kahan process are subject to breakdown, which is related to the fact that $\operatorname{rank}(\mathbf{A}) = 50$ in this experiment. 
For the restarted and recycled methods, we set $D_{\text{min}} = 15$ and $D_{\text{max}} = 25$. 
We find that the res/recPS-GKS methods behave nearly identically to PS-GKS, with the difference in IAS weights attributed to nonconvexity. 
All three methods outperform both S-GKS and res/recS-GKS. 
The interpretation is that the preconditioning employed by the PS-GKS methods is significantly more effective than that of the S-GKS methods for this experiment. 
The dominant increase in cost for this improved performance is the large number of matvecs required with $\boldsymbol{\Psi}^{-1}$. 
This highlights the importance of the res/recPS-GKS methods for addressing not only the memory concerns of storing a large number of basis vectors but also mitigating the number of matvecs with $\boldsymbol{\Psi}^{-1}$. 
The results obtained by the PS-GKB method closely follow those of PS-GKS, suggesting that there is no significant performance difference between projection methods based on GKS or GKB. 
We also note that the observed faster convergence of PS-GKS using the IAS weights compared to the MM weights should be expected due to our \Cref{thm:main_sparsity_theorem}---since the IAS weights promote sparsity more aggressively than the MM weights. 

Regarding the flexible methods, the results with MM weights roughly follow those of the S-GKS method, until the breakdown occurs. 
However, using the IAS weights, we observe that the flexible methods perform comparably to the unweighted GKS method, which produces a smooth solution. 
This behavior is due to the nonconvexity of the associated regularization penalty and the extremely large condition numbers exhibited by the projected least squares problems solved by the flexible methods, which are significantly higher than for either the S-GKS or PS-GKS methods (see for instance \ref{tab:1d_cosine_performance_results} for more details). 
Unlike the S-GKS and PS-GKS methods, these methods do not project onto orthogonal bases. 
Consequently, the condition numbers of the least squares problems these methods solve cannot be bounded above by that of the full-scale problem. 

\subsubsection{Sensitivity of the MM methods w.r.t.\ $\varepsilon$}
\label{sub:epsilon_sensitivity} 

For methods using MM weights, it is well known that the value of $\varepsilon$ used in defining the smoothed approximation to the $\ell_p$ norm may drastically affect the performance of various methods---a value of $\varepsilon$ that is too large promotes sparsity only weakly, while a value of $\varepsilon$ that is too small may yield numerical instabilities in the method. 
Here, we argue that methods that make use of priorconditioning (including PS-GKS, PS-GKB, and flexible methods) are relatively insensitive to the $\varepsilon$ parameter when compared to S-GKS. 
To demonstrate this, we compare the dependency of the final RRE and Gini index (after 150 iterations) for several methods on the value of the $\varepsilon$ parameter in \cref{fig:1d_cosine_mm_epsilon_sensitivity}. 

\begin{figure}[tb]
    \centering
    \includegraphics[width=.95\textwidth]{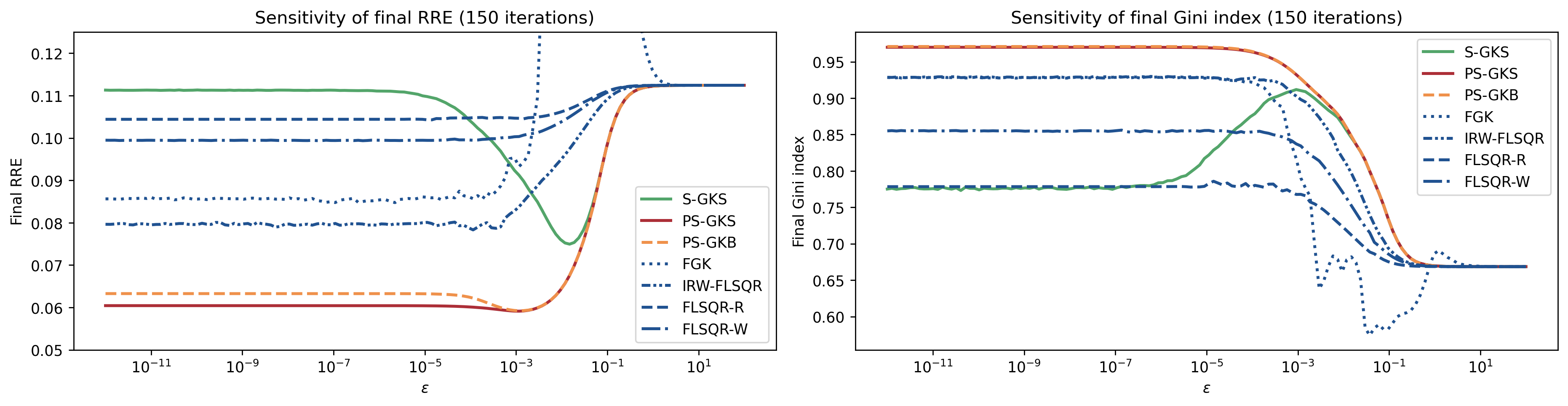}
    \caption{Test 1. Sensitivity of 1D cosine problem results w.r.t.\ the MM $\varepsilon$ parameter in terms of the final RRE (left) and Gini index (right).}
    \label{fig:1d_cosine_mm_epsilon_sensitivity}
\end{figure}

We observe that the quality of the S-GKS solution initially improves as $\varepsilon$ decreases, but then degrades once $\varepsilon$ is reduced below a certain critical threshold. 
Since this threshold is not known a priori, this highlights the role of $\varepsilon$ as an additional tuning parameter for the S-GKS method with MM weights. 
In contrast, we observe that the performance of the PS-GKS, PS-GKB, and flexible methods (except FGK) is comparatively insensitive with respect to  $\varepsilon$, which can thus safely be set to a small value without degrading performance.

\subsection{Test 2: Computerized X-ray tomography problems (CT)}
\label{sub:numerics_tomography} 

In this section we provide a comparison of our PS-GKS with other existing methods for a large-scale limited angle inverse CT problem. 
To this end, we use the TRIPs-Py library \cite{pasha2024trips} to generate the true synthetic CT for the $256 \times 256$ Shepp--Logan phantom image (shown in Figure \ref{fig:2d_tomo_data} (a)) using a parallel beam geometry and contaminated by Gaussian noise with level $\sigma_{\text{NL}} = 1\%$. 
The observational data is shown in \Cref{fig:2d_tomo_data}(b) and consists of 28 view angles in the interval $[0, 2 \pi)$, yielding a measurement operator $\mathbf{A} \in \mathbb{R}^{10136 \times 65536}$ and data $\mathbf{b} \in \mathbb{R}^{10136}$. 
\begin{figure}[tb]
    \centering
    \includegraphics[width=.5\textwidth]{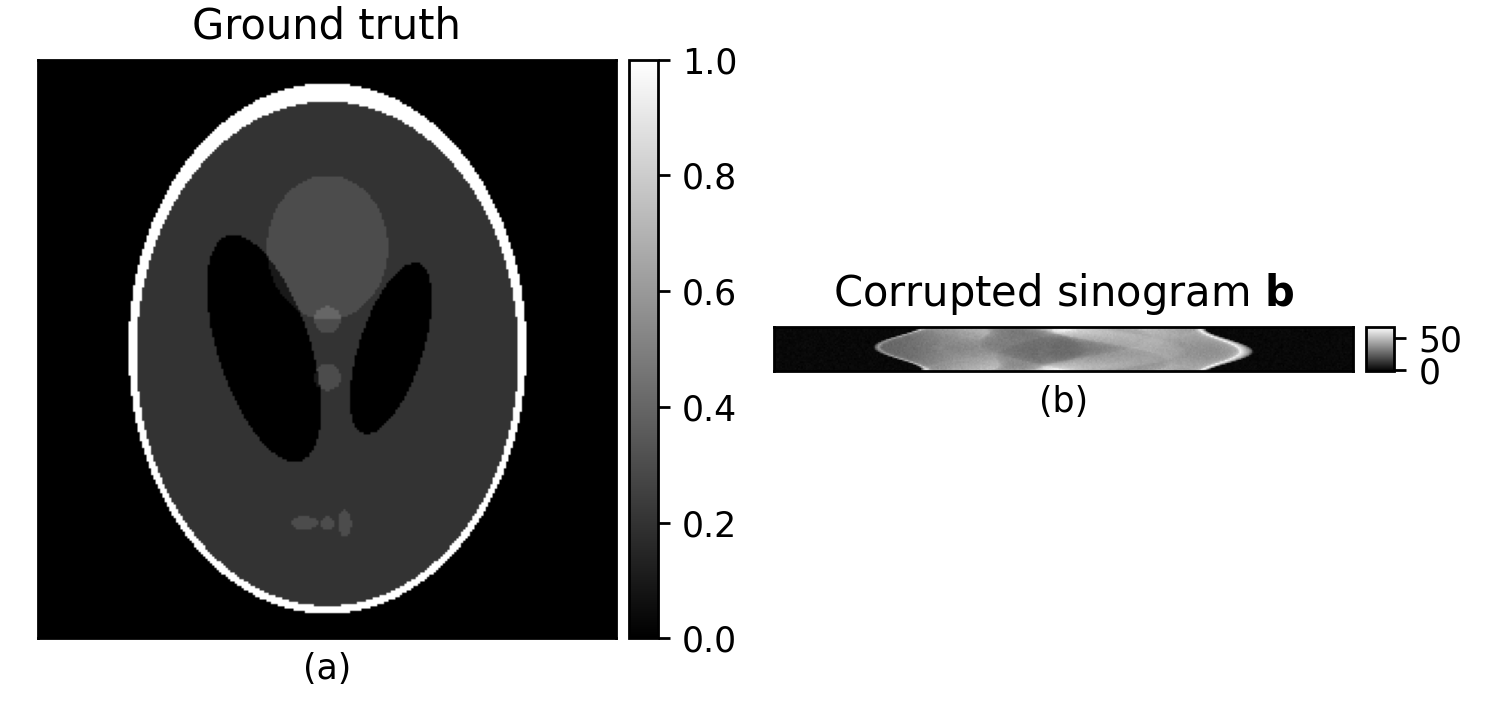}
    \caption{
        Test 2. a) True image of size $256 \times 256$. b) Sinogram obtained from $28$ view angles.   
    }
    \label{fig:2d_tomo_data}
\end{figure}
We use the anisotropic two-dimensional first derivative operator with Neumann boundary conditions for the sparsifying transform $\boldsymbol{\Psi}$. In this case, $\operatorname{ker}(\boldsymbol{\Psi}) = \operatorname{span}\{ \mathbf{1}_N \}$ and the oblique pseudoinverse must be used to employ priorconditioning. 
To compute matvecs with $\boldsymbol{\Psi}_{\itidx}^\dagger$, we utilize a preconditioned CG method with a GPU-accelerated spectral preconditioner based on the unweighted pseudoinverse $\boldsymbol{\Psi}^\dagger$; 
See \cite[Appendix B]{lindbloom2024generalized}. 
Furthermore, we set $D_{\text{min}} = 25$ and $D_{\text{max}} = 40$ for the restarted and recycled methods. We run all methods for a maximum of 100 iterations, except for the restarted and recycled methods which we run longer for 200 iterations since it is feasible to do so due to their mitigated basis sizes. The MM weight formulations are kept the same as in \cref{tab:methods_summary}, but for the IAS weights we instead use nonconvex hyper-prior parameters $r = 1/2$, $\beta = 3.01$, corresponding to an approximation of  $\ell_{2/3}$-norm regularization (see \cite{calvetti2020sparse}). 
We use this choice instead of $r = -1$, $\beta = 1$ because the latter promotes sparsity more aggressively. 
In our tests, all methods struggle to reconstruct solutions that do not default to undesirable smooth local minima. 
We leave the exploration of the $r = -1$ case to future work.

Performance metrics, reconstructions, and error images for the experiment are shown in \cref{tab:2d_tomo_performance_results,fig:2d_tomo_reconstructions_mm,fig:2d_tomo_reconstructions_ias}. We observe that the PS-GKS methods outperform the S-GKS methods using both MM and IAS weights. This should be evaluated against the number of matvecs these methods require with $\boldsymbol{\Psi}_{\itidx+1}^\dagger$. Using the IAS weights we observe somewhat erratic initial convergence behavior, which is attributed to the nonconvexity of the associated regularization penalty.

Since in this example $\boldsymbol{\Psi}^{-1}$ does not exist, many existing methods can not be applied. We consider for comparison FGK. We found particularly poor performance using the MM weight formulation MM$_{4}$ which uses $\varepsilon = 10^{-10}$, so in this example we instead use MM$_{2}$ with $\varepsilon = 10^{-2}$. We find that the FGK method does not perform as well as PS-GKS or even the S-GKS method in this numerical test. A possible explanation for this is seen in that the condition number of the projected least squares problem solved by FGK grows increasingly large as the iterations progress (exceeding $\kappa = 10^7$ with the MM weights, and $\kappa = 10^{10}$ with the IAS weights). Indeed, these large condition numbers are shared by the projection matrix $\mathbf{Z}_{\itidx} = [ (\boldsymbol{\Psi}_{1})_{\mathbf{A}}^\dagger ((\boldsymbol{\Psi}_{1})_{\mathbf{A}}^\dagger)^T \mathbf{v}_1, \ldots, (\boldsymbol{\Psi}_{\ell})_{\mathbf{A}}^\dagger ((\boldsymbol{\Psi}_{\ell})_{\mathbf{A}}^\dagger)^T \mathbf{v}_{\itidx}]$ used to perform the projection (see \cref{ssub:fgk}). We speculate that the large condition number results from the fact that $\mathbf{Z}_{\ell}$ incorporates information from the weights at all previous iterations, and the weights at later iterations may vary drastically from the initial weights $\mathbf{w}_0$. A potential remedy for this issue would be to further incorporate restarting/recycling into the FGK method, but we do not pursue this here.

\begin{table}[ht]
\centering
\scalebox{0.8}{
\begin{tabular}{lllrrrcrrr}
\toprule
Weights & Method & $\regparam$ & $n_{\textrm{iter}}$ & RRE & SSIM & Gini index & $n_{\mathbf{A}}$ & $n_{\boldsymbol{\Psi}}$ & $n_{\boldsymbol{ \Psi }^\dagger }$ \\
\midrule
--- & GKS & $2\cdot10^{2}$ & $100$ & $0.418$ & $0.422$ & $0.538$ & $215$ & $305$ & $0$ \\
\hline
MM & S-GKS & $2\cdot10^{1}$ & $100$ & $0.192$ & $0.808$ & $0.855$ & $215$ & $306$ & $0$ \\
MM & resS-GKS & $2\cdot10^{2}$ & $200$ & $0.417$ & $0.423$ & $0.539$ & $420$ & $611$ & $0$ \\
MM & recS-GKS & $2\cdot10^{1}$ & $200$ & $0.191$ & $0.808$ & $0.855$ & $715$ & $906$ & $0$ \\
MM & PS-GKS & $2\cdot10^{1}$ & $100$ & $0.151$ & $0.934$ & $0.957$ & $5661$ & $103$ & $5760$ \\
MM & resPS-GKS & $2\cdot10^{1}$ & $200$ & $0.151$ & $0.934$ & $0.957$ & $4181$ & $203$ & $4380$ \\
MM & recPS-GKS & $2\cdot10^{1}$ & $200$ & $0.150$ & $0.934$ & $0.957$ & $6308$ & $203$ & $6507$ \\
MM & PS-GKB & $2\cdot10^{1}$ & $100$ & $0.152$ & $0.934$ & $0.958$ & $10103$ & $101$ & $10200$ \\
MM & FGK & $3\cdot10^{0}$ & $100$ & $0.482$ & $0.216$ & $0.557$ & $203$ & $201$ & $200$ \\
\hline
IAS & S-GKS & $7\cdot10^{1}$ & $100$ & $0.371$ & $0.519$ & $0.828$ & $215$ & $306$ & $0$ \\
IAS & resS-GKS & $5\cdot10^{1}$ & $200$ & $0.382$ & $0.501$ & $0.812$ & $715$ & $906$ & $0$ \\
IAS & recS-GKS & $7\cdot10^{1}$ & $200$ & $0.364$ & $0.525$ & $0.848$ & $420$ & $611$ & $0$ \\
IAS & PS-GKS & $3\cdot10^{2}$ & $100$ & $0.117$ & $0.809$ & $0.986$ & $5661$ & $103$ & $5760$ \\
IAS & resPS-GKS & $3\cdot10^{2}$ & $200$ & $0.111$ & $0.974$ & $0.986$ & $4181$ & $203$ & $4380$ \\
IAS & recPS-GKS & $3\cdot10^{2}$ & $200$ & $0.111$ & $0.974$ & $0.987$ & $6308$ & $203$ & $6507$ \\
IAS & PS-GKB & $1\cdot10^{-7}$ & $100$ & $0.431$ & $0.423$ & $0.575$ & $10103$ & $101$ & $10200$ \\
IAS & FGK & $6\cdot10^{0}$ & $100$ & $0.478$ & $0.223$ & $0.491$ & $203$ & $201$ & $200$ \\
\bottomrule
\end{tabular}
}
\caption{Test 2. Summary of performance metrics for MM and IAS weights.}
\label{tab:2d_tomo_performance_results}
\end{table}

\begin{figure}[tb]
    \centering
    \includegraphics[width=.99\textwidth]{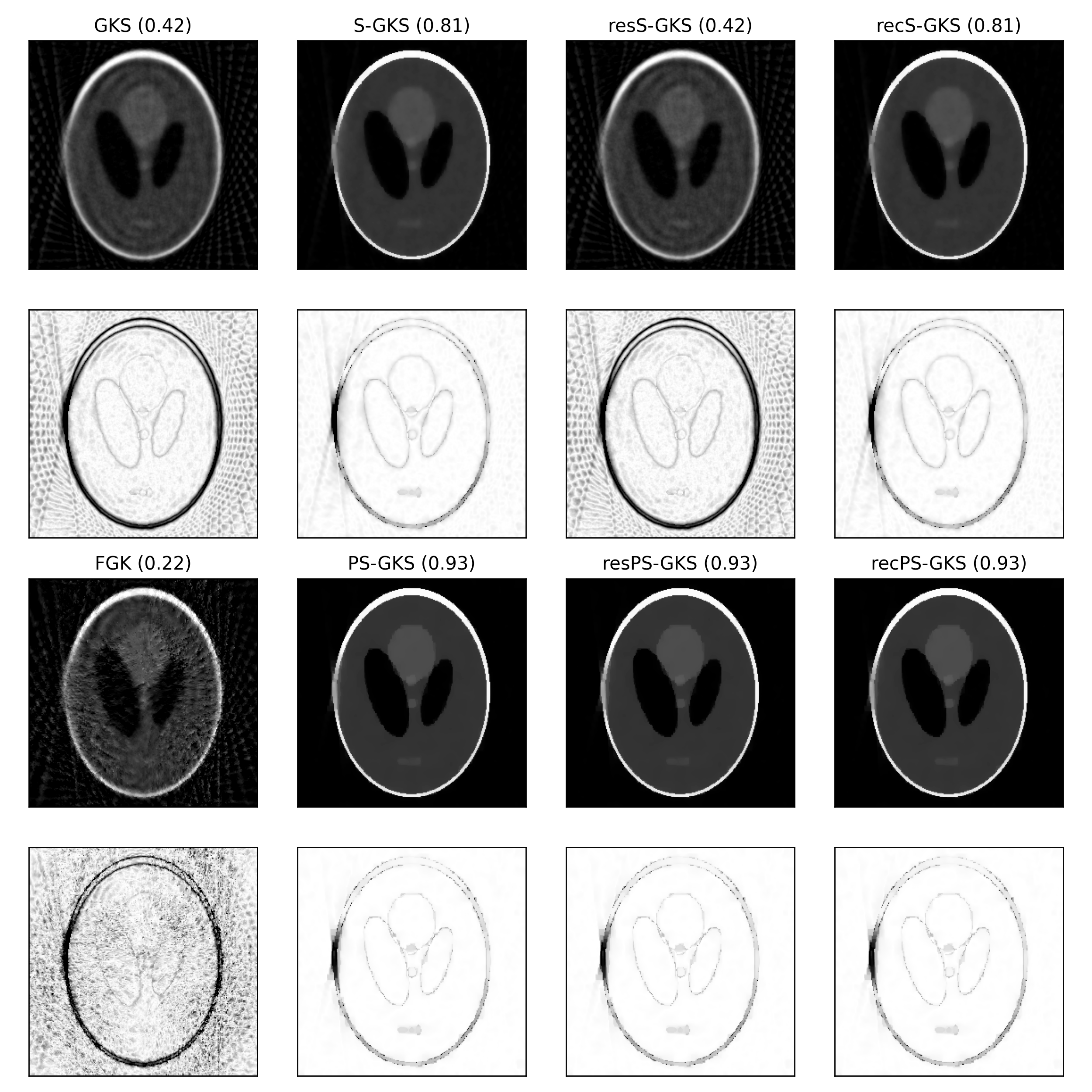}
    \caption{
        Test 2. Reconstructions by different methods using MM weights with the SSIM of the reconstruction shown in parentheses (first and third rows), as well as  error images, shown in the reverse color scale and with a shared range of values across all methods (second and fourth rows).
    }
    \label{fig:2d_tomo_reconstructions_mm}
\end{figure}

\begin{figure}[tb]
    \centering
    \includegraphics[width=.99\textwidth]{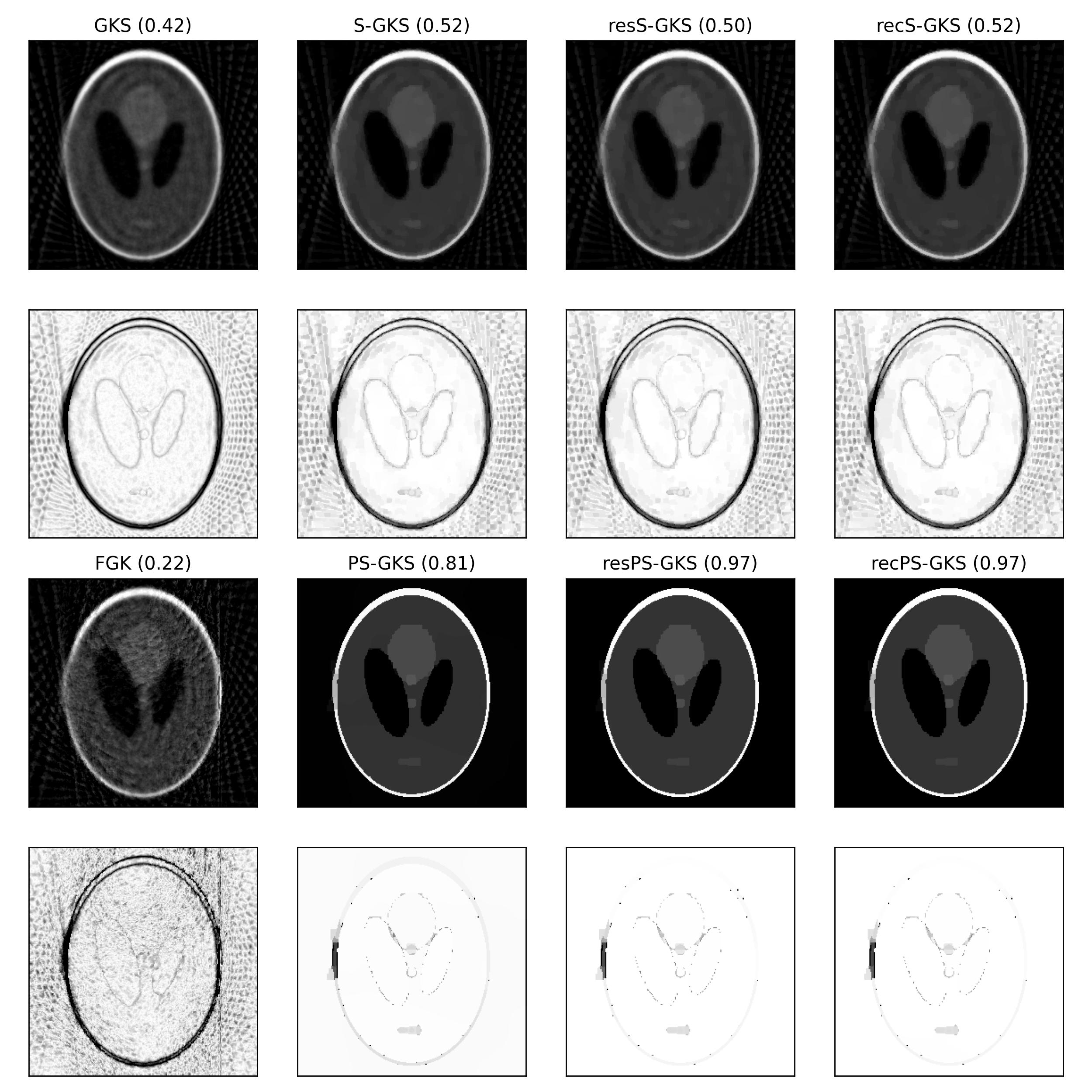}
    \caption{
        Test 2. Reconstructions by different methods using IAS weights with the SSIM of the reconstruction shown in parentheses (first and third rows), as well as  error images, shown in the reverse color scale and with a shared range of values across all methods (second and fourth rows).
    } 
    \label{fig:2d_tomo_reconstructions_ias}
\end{figure}

\begin{figure}[tb]
    \centering
    \includegraphics[width=.99\textwidth]{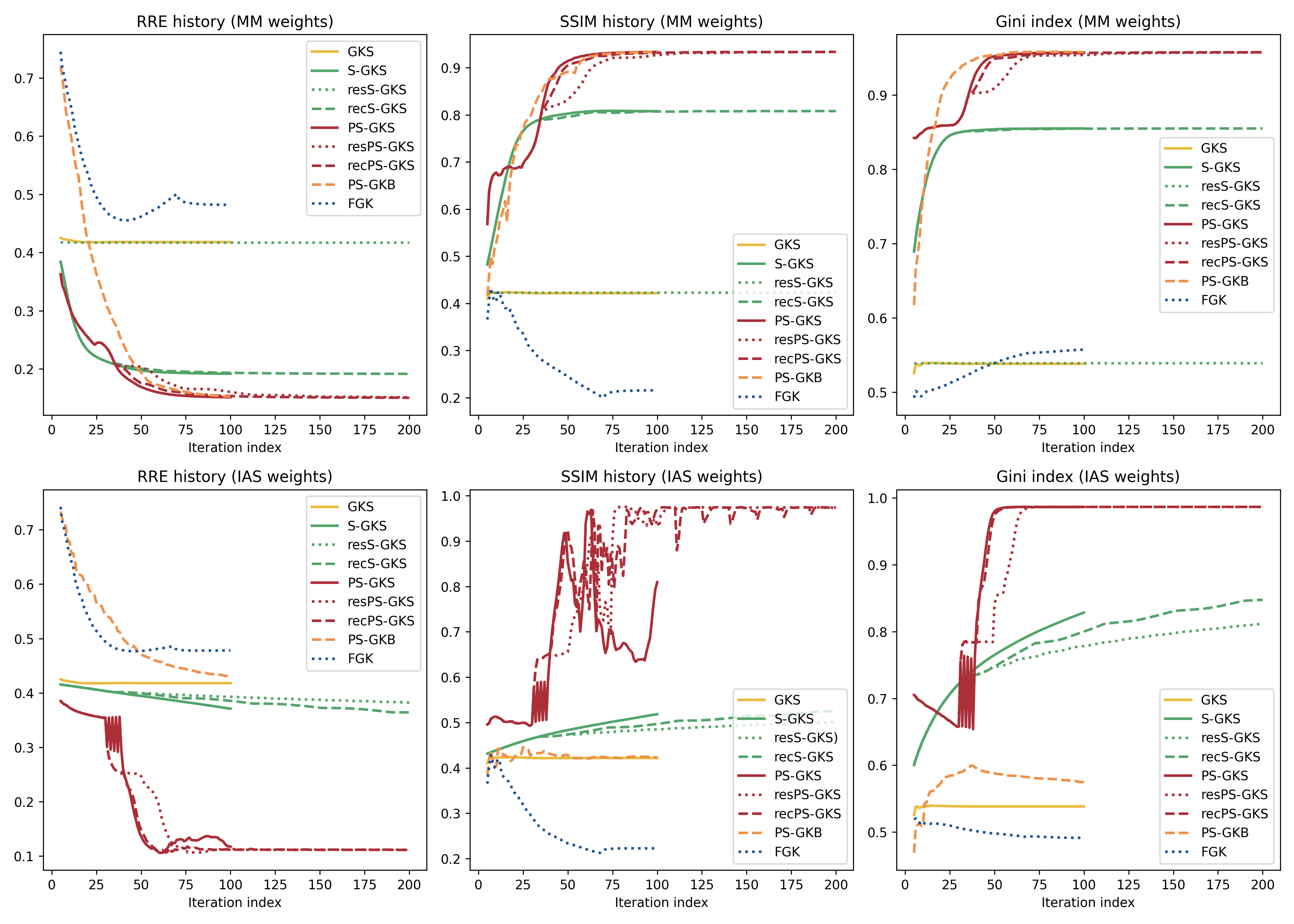}
    \caption{Test 2. Histories of the RRE, SSIM, and Gini index for all methods compared. (top row) results using MM weights; (bottom row) results using IAS weights.}
    \label{fig:2d_tomo_reconstructions_ias_performance}
\end{figure}


\section{Conclusion and outlook}\label{sec:conclusion}

In this paper we propose priorconditioning methods that can be used in both deterministic and Bayesian setting for ill-posed inverse problems with sparsity-promoting priors. Namely, we develop prior conditioned sparsity generalized Krylov subspace  PS-GKS methods that when used in the context of reweighting exhibits superior reconstructions at a relatively small number of iterations. We further propose variations of the PS-GKS method, including restarting and recycling (resPS-GKS and recPS-GKS).
Such methods allow us to improve the computed solution quality and the computational time by reducing memory requirements, and automatically select a regularization parameter at a reduced number of iterations compared to the original S-GKS. 
For the Bayesian setting, we can estimate one of the hyperprior parameters automatically. While we only focus on the GSBL priors, our work has potential to be extended to other hierarchical priors and further can be used for more efficient posterior characterization, i.e., UQ for the sparsity-promoting priors. We provide a theoretical analysis showing the benefits of priorconditioning in for sparsity-promoting inverse problems.

As future work we consider extension to the dynamic inverse problems setting which we anticipate to present computational difficulties especially in efficient estimation of the priorconditioner. Addressing computational concerns that arise from the need to estimate pseudoinverses of large and complex matrices is left as future work. Our preliminary results on multigrid methods show that we can avoid the need for GPU usage and still maintain low computational time -- a direction which we aim to pursue for the dynamic inverse problems setting \cite{pasha2023computational, lan2023spatiotemporal}.


\section*{Acknowledgements}
JL and JG were supported by the US DOD (ONR MURI) grant \#N00014-20-1-2595. MP acknowledges support from NSF DMS 2410699. Any opinions, findings, conclusions, or recommendations expressed in this material are those of the authors and do not necessarily reflect the views of the National Science Foundation.

\appendix

\section{Proof of \cref{thm:main_sparsity_theorem}}\label{app:proof_eigenvalue_thm} The lower bound is trivial. To obtain the second upper bound, a generalization of Ostrowski's theorem (see \cref{thm:ostrowski}) can be applied twice to obtain
    \begin{equation}\label{eq:sparsity_theorem_l3}
    \begin{aligned}
        \lambda_i(\overline{\mathbf{A}}^T \overline{\mathbf{A}}) 
            & \leq \lambda_1(\mathbf{A}^T \mathbf{A}) \lambda_i( ( (\mathbf{W} \boldsymbol{\Psi})^\dagger)^T \mathbf{E}^T \mathbf{E} (\mathbf{W} \boldsymbol{\Psi})^\dagger ) \\
            & \leq \lambda_1(\mathbf{A}^T \mathbf{A}) \lambda_1(\mathbf{E}^T \mathbf{E}) \lambda_i( ((\mathbf{W} \boldsymbol{\Psi})^\dagger)^T  (\mathbf{W} \boldsymbol{\Psi})^\dagger ) \\
            & \leq \lambda_1(\mathbf{A}^T \mathbf{A}) \lambda_i( ( \mathbf{W} \boldsymbol{\Psi} \boldsymbol{\Psi}^T \mathbf{W} )^\dagger ) 
    \end{aligned}
    \end{equation} 
    since $\mathbf{E}$ is a projector satisfying $\mathbf{E}^2 = \mathbf{E}$ with eigenvalues in $\{ 0, 1 \}$. 
    Let $\hat{\boldsymbol{\Psi}}$ be a square matrix such that $\boldsymbol{\Psi} \boldsymbol{\Psi}^T = \hat{\boldsymbol{\Psi}} \hat{\boldsymbol{\Psi}}^T$ (e.g., $\hat{\boldsymbol{\Psi}} = (\boldsymbol{\Psi} \boldsymbol{\Psi}^T)^{1/2}$), and  
    note that $\lambda_i( ( \mathbf{W} \boldsymbol{\Psi} \boldsymbol{\Psi}^T \mathbf{W} )^\dagger   ) = (\lambda_{R-i+1}( \mathbf{W} \boldsymbol{\Psi} \boldsymbol{\Psi}^T \mathbf{W}    ))^{-1} = (\lambda_{R-i+1}(  \hat{\boldsymbol{\Psi}}^T \mathbf{W}^2 \hat{\boldsymbol{\Psi}}    ))^{-1} $ for $i = 1, \ldots, R$, with the remaining $K-R$ eigenvalues being zero. Another application of \cref{thm:ostrowski} gives 
    \begin{align}
        \lambda_{R-i+1}(\hat{\boldsymbol{\Psi}}^T \mathbf{W}^2 \hat{\boldsymbol{\Psi}}) \geq \lambda_R(\hat{\boldsymbol{\Psi}}^T \hat{\boldsymbol{\Psi}}) \lambda_{N-i+1}(\mathbf{W}^2) = \lambda_R(\boldsymbol{\Psi}^T \boldsymbol{\Psi}) / \lambda_{i}(\mathbf{W}^{-2})
    \end{align}
    for $i = 1, \ldots, R$, which combined with \cref{eq:sparsity_theorem_l3} gives
    \begin{align}
\lambda_i(\overline{\mathbf{A}}^T \overline{\mathbf{A}}) \leq \lambda_i(\mathbf{W}^{-2})  \lambda_1(\mathbf{A}^T \mathbf{A}) /\lambda_R(\boldsymbol{\Psi}^T \boldsymbol{\Psi})
    \end{align}
    for $i = 1, \ldots, R$, with the remaining eigenvalues satisfy  $\lambda_i(\overline{\mathbf{A}}^T \overline{\mathbf{A}}) = 0$ for $i = R+1, \ldots, K$. Shifting these eigenvalues by $+\regparam$ yields the second upper bound. 
    The first upper bound is obtained by applying the generalized Ostrowski theorems of \cite{higham1998modifying} to obtain $\lambda_i(\overline{\mathbf{A}}^T \overline{\mathbf{A}}) \leq \lambda_i( \mathbf{A}^T \mathbf{A})  \lambda_1( ((\mathbf{W} \boldsymbol{\Psi})^\dagger)^T  (\mathbf{W} \boldsymbol{\Psi})^\dagger )$ for $i = 1, \ldots, K$, which following the preceding arguments can be bounded as $\lambda_i(\overline{\mathbf{A}}^T \overline{\mathbf{A}}) \leq \lambda_i(\mathbf{A}^T \mathbf{A}) \lambda_1(\mathbf{W}^{-2})/\lambda_R(\boldsymbol{\Psi}^T \boldsymbol{\Psi})$. Again, shifting these eigenvalues by $+\regparam$ gives the desired upper bound.


\section{MM weights}
\label{app:mm_weights} Various choices of the weight matrix \cref{eq:MM_original} for $\ell_p$-regularization have been utilized in the literature on hybrid projection methods. 
Here, we present the different choices that can be used in tandem with \cref{tab:methods_summary} to determine the weighting scheme used for each method in our numerical experiments. 
We define weighting schemes $\text{MM}_{i}$ for $i = 1, \ldots, 4$ as $\mathbf{W}_{\itidx+1} = \operatorname{diag}(\varphi_{\text{MM}_i}(\boldsymbol{\Psi} \mathbf{x}_{\itidx}) )$ with $\varphi_{\text{MM}_i}(z) = \left(z^2 + \varepsilon_i^2 \right)^{\frac{p-2}{4}}$, where $\varepsilon_1 = 1, \varepsilon_2 = 10^{-2}, \varepsilon_3 = 10^{-3}$, and $\varepsilon_4 = 10^{-4}$. 
Additionally, we define a fifth choice $\text{MM}_5$ by $\mathbf{W}_{\itidx+1} = \operatorname{diag}(\varphi_{\text{MM}_5}(\boldsymbol{\Psi} \mathbf{x}_{\itidx} ))$ with $\varphi_{\text{MM}_5}(z) = |z|^{\frac{p-2}{2}}$ if $z \geq \tau_1$ and $\varphi_{\text{MM}_5}(z) = \tau_2^{\frac{p-2}{2}}$ otherwise, where $\tau_1 = 10^{-10}$ and $\tau_2 = 10^{-16}$.

\section{The S-GKS algorithm}\label{sec:sgks_algo} \cref{alg:sgks} provides pseudocode for the S-GKS method.

\begin{algorithm}[!ht]
\caption{The S-GKS method}
\label{alg:sgks}
\begin{algorithmic}[1]
	\Require{$\bA, \boldsymbol{\Psi}, \mathbf{b}, \bx_{0}$} 
	\Ensure{An approximate solution $\bx_{\itidx+1}$}
	\Function{$\bx_{\itidx+1} = $ S-GKS }{$\bA, \boldsymbol{\Psi}, \mathbf{b}, \bx_{0}$}
	\State	Generate initial subspace basis $\bV_{0}\in \R^{N \times D_0}$ s.\,t.\ $\bV_{0}^{T}\bV_{0}=\bI_{D_0}$\; 
\FOR {$\itidx=0,1,2,\ldots$ \text{until convergence}}{
\State Update weights $\mathbf{w}_{\itidx+1}$ given $\boldsymbol{\Psi} \mathbf{x}_{\itidx}$, according to the specific IRLS method\;	
\State $\mathbf{W}_{\itidx+1} = \operatorname{diag}(\bw_{\itidx+1})$ and $\boldsymbol{\Psi}_{\itidx+1} = \mathbf{W}_{\itidx+1} \boldsymbol{\Psi}$
\State $\bA \bV_{\itidx}$ and $\boldsymbol{\Psi}_{\itidx+1} \bV_{\itidx}$\;
\State	$\bA\bV_{\itidx} = \bQ_{\bA}\bR_{\bA}$ and $\boldsymbol{\Psi}_{\itidx+1}  \bV_{\itidx}=\bQ_{\boldsymbol{\Psi}}\bR_{\boldsymbol{\Psi}}$\;\Comment{Compute/update the economic QR}
\State Select $\regparam_{\itidx+1}$ by heuristic (e.g., DP) \;
\State $\bz_{\itidx+1}$ to solve the projected problem with selected $\regparam_{\itidx+1}$\;\Comment{Solve projected problem}
\State $\bx_{\itidx+1}=\bV_{\itidx}\bz_{\itidx+1}$\;\Comment{Full-scale solution via projection}
\State $\br_{\itidx+1}=\bA^T(\bA\bV_{\itidx} \bz_{\itidx+1} -\mathbf{b})+\regparam_{\itidx+1} \boldsymbol{\Psi}_{\itidx+1}^T \boldsymbol{\Psi}_{\itidx+1} \bV_{\itidx}\bz_{\itidx+1}$\;\Comment{Full-scale residual}
\State		 $\br_{\itidx+1} = \br_{\itidx+1} - \bV_{\itidx}\bV_{\itidx}^{T}\br_{\itidx+1}$\; \Comment{Reorthogonalize (optional)}
\State		 $\bv_{\rm new}=\frac{\br_{\itidx+1}}{\|\br_{\itidx+1}\|_{2}}$; $\bV_{\itidx+1}=[\bV_{\itidx}, \bv_{\rm new}]$\;\Comment{Enlarge the solution subspace} }
\ENDFOR
	\EndFunction
\end{algorithmic}
\end{algorithm}

\section{Basis vector comparison}\label{sec:basis_vector_comparisons} \Cref{fig:1d_cosine_results_mm_and_ias_basis_vectors} provides a comparison of the basis vectors generated by the S-GKS and PS-GKS methods for the 1D cosine problem. 

\begin{figure}
    \centering
    \includegraphics[width=.90\textwidth]{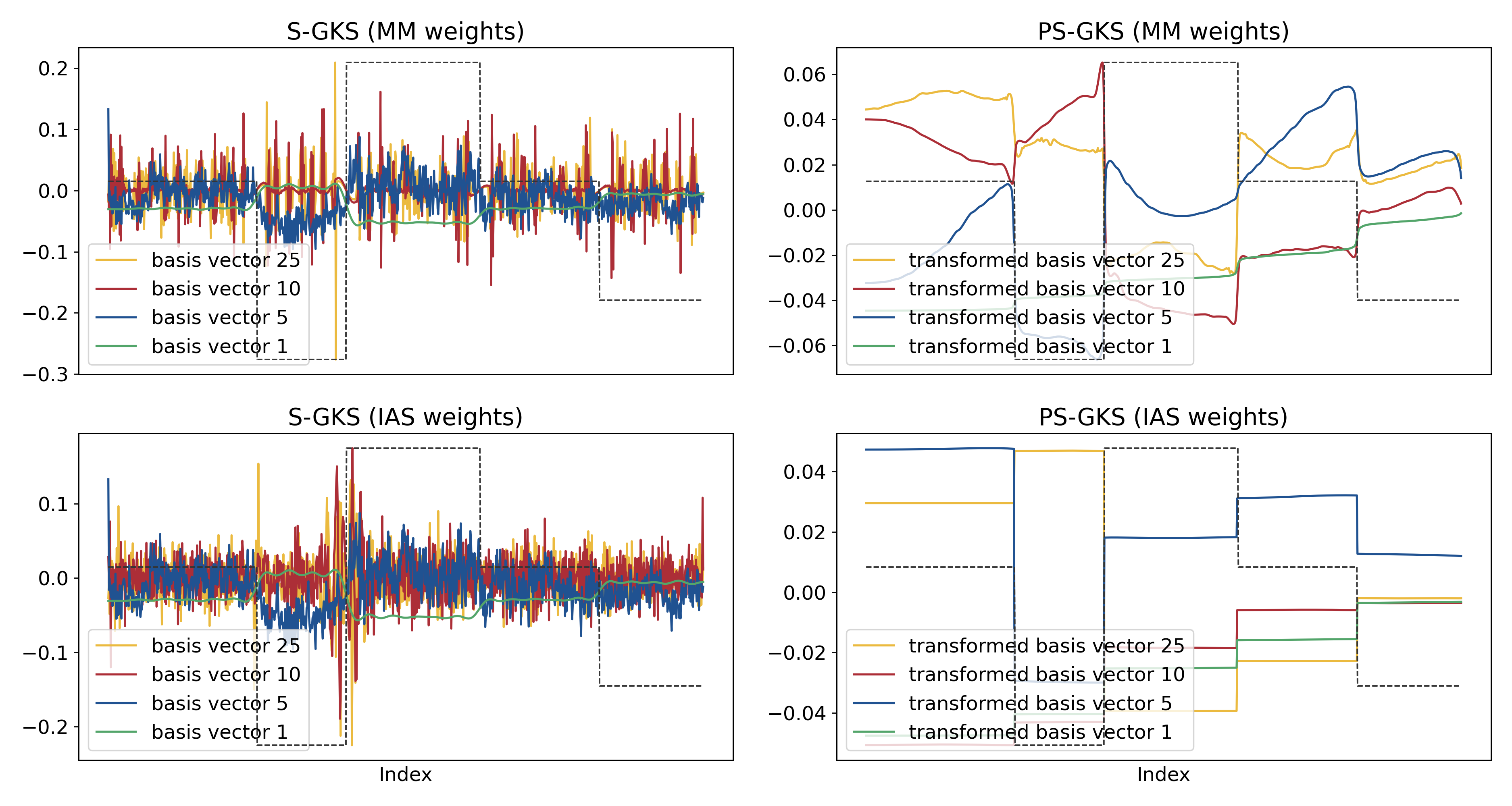}
    \caption{Test 1. Comparison of the basis vectors generated by S-GKS and PS-GKS, with MM weights ($p = 1$, $\varepsilon = 10^{-4}$) and IAS weights ($r = -1$, $\beta = 1$). (left column) select S-GKS basis vectors; (right column) select PS-GKS basis vectors, transformed into the same space as the S-GKS basis vectors to aid the comparison. In all plots, the ground truth vector is overlayed with a dashed black line.
    }
\label{fig:1d_cosine_results_mm_and_ias_basis_vectors}
\end{figure}

\section{Generalized Ostrowski Theorems} The proofs of the eigenvalue bounds for $\mathbf{Q}_{\regparam}^{\text{st}}$ and $\mathbf{Q}_{\regparam}^{\text{pr}}$ given in \Cref{sub:analysis_of_priorconditioning} rely on the following two generalizations of the Ostrowski theorem.    Their proofs rely on the Ostrowski theorem \cite[Corollary 4.5.11]{johnson1985matrix} as well as the Cauchy interlace theorem \cite[Theorem 10.1.1]{parlett1998symmetric}. Additionally, in \Cref{fig:eigenvalue_bounds_plot} we provide a numerical demonstration of our \Cref{thm:main_sparsity_theorem} applied to the 1D cosine problem.

\begin{theorem}[Rank-deficient rectangular Ostrowski theorem]\label{thm:ostrowski}
    Let $\mathbf{C} \in \mathbb{R}^{N \times N}$ be symmetric, and let $\mathbf{X} \in \mathbb{R}^{N \times M}$ with $R = \operatorname{rank}(\mathbf{X})$. Then the eigenvalues of $\mathbf{X}^T \mathbf{C} \mathbf{X}$ satisfy
    \begin{align}
        \lambda_{i + (N-R)}(\mathbf{C}) \lambda_R(\mathbf{X}^T \mathbf{X}) \leq \lambda_i(\mathbf{X}^T \mathbf{C} \mathbf{X}) \leq \lambda_i(\mathbf{C}) \lambda_1(\mathbf{X}^T \mathbf{X}), \quad i = 1, \ldots, R,
    \end{align}
    and the remaining $M-R$ eigenvalues of $\mathbf{X}^T \mathbf{C} \mathbf{X}$ are all zero. Furthermore, if $\mathbf{C}$ is symmetric positive semidefinite (SPSD) then the eigenvalues of $\mathbf{X}^T \mathbf{C} \mathbf{X}$ also satisfy
    \begin{align}
\lambda_i(\mathbf{X}^T \mathbf{C} \mathbf{X}) \leq \lambda_1(\mathbf{C}) \lambda_i(\mathbf{X}^T \mathbf{X}) , \quad i = 1, \ldots, R.
    \end{align}
\end{theorem}
\begin{proof}
   We prove the results for the tall ($N \geq M$) and wide ($N \leq M$) cases for $\mathbf{X}$ separately.

    \emph{Tall case ($N \geq M$):}  Let $\mathbf{X} = \mathbf{U} \begin{bmatrix} \boldsymbol{\Sigma} \\  \mathbf{0}_{(N-M) \times M} \end{bmatrix} \mathbf{V}^T$ be the SVD of $\mathbf{X}$, where $\boldsymbol{\Sigma} = \operatorname{diag}(\sigma_1, \ldots, \sigma_M) \in \mathbb{R}^{M \times M}$ contains the singular values in descending order. Note that we will have $\sigma_{R+1} = \cdots = \sigma_{M} = 0$ (the last $M-R$ singular values are zero). Let $\gamma$ be an arbitrary scalar to be determined later, and note that
    \begin{align}
        \boldsymbol{\Sigma} &= \operatorname{diag}(\overbrace{1, \ldots, 1,}^{R} \underbrace{0, \ldots, 0}_{M-R})  \cdot \operatorname{diag}(\overbrace{\sigma_1, \ldots, \sigma_{R}}^{R}, \underbrace{\gamma, \cdots, \gamma}_{M-R}) \in \mathbb{R}^{M \times M} \\
        &= \mathbf{J} \mathbf{D}_{\gamma}
    \end{align}
    where we have defined $\mathbf{J}$ and $\mathbf{D}_{\gamma}$ as their corresponding factors in the line above. We then find that 
    \begin{align}
        \mathbf{X}^T \mathbf{C} \mathbf{X} &= \mathbf{V} \begin{bmatrix} \boldsymbol{\Sigma}^T & \mathbf{0}_{M \times (N-M)}\end{bmatrix} \mathbf{U}^T \mathbf{C} \mathbf{U} \begin{bmatrix} \boldsymbol{\Sigma} \\ \mathbf{0}_{(N-M) \times M} \end{bmatrix} \mathbf{V}^T \\
        &= \mathbf{V} \mathbf{D}_{\gamma}^T \left( \begin{bmatrix} \mathbf{J}^T & \mathbf{0}_{M \times (N-M)} \end{bmatrix} \mathbf{U}^T \mathbf{C} \mathbf{U} \begin{bmatrix} \mathbf{J} \\ \mathbf{0}_{(N-M) \times M} \end{bmatrix} \right) \mathbf{D}_{\gamma} \mathbf{V}^T \\
        &= \mathbf{V} \mathbf{D}_{\gamma}^T \begin{bmatrix}
            \left( \mathbf{U}^T \mathbf{C} \mathbf{U} \right)_{1:R,1:R} & \mathbf{0}_{R \times (M-R)} \\ \mathbf{0}_{(M-R) \times R} & \mathbf{0}_{(M-R) \times (M-R)}
        \end{bmatrix} \mathbf{D}_{\gamma} \mathbf{V}^T \\
        &= \mathbf{V} \mathbf{D}_{\gamma}^T \mathbf{Z} \mathbf{D}_{\gamma} \mathbf{V}^T
    \end{align}
    where $\mathbf{Z}$ is defined from the previous line and  $(\mathbf{U}^T \mathbf{C} \mathbf{U})_{1:R, 1:R}$ denotes the leading principle submatrix of $\mathbf{U}^T \mathbf{C} \mathbf{U}$ of order $R$. Note that the first $R$ eigenvalues of $\mathbf{Z}$ are the same as those of $(\mathbf{U}^T \mathbf{C} \mathbf{U})_{1:R, 1:R}$ and that the remaining $M-R$ eigenvalues are all zero. For $i = 1, \ldots, R$, an application of the usual (square) Ostrowski theorem gives
  \begin{align}
    \lambda_i(\mathbf{X}^T \mathbf{C} \mathbf{X}) &= \lambda_i( \mathbf{D}_{\gamma}^T \mathbf{Z} \mathbf{D}_{\gamma} ) \label{eq:XtCXl1} \\
        &= \lambda_i( \mathbf{Z} ) \theta_i \\
        &= \lambda_i\left( (\mathbf{U}^T \mathbf{C} \mathbf{U})_{1:R,1:R} \right) \theta_i 
    \end{align}
    where $\theta_i$ is some scalar satisfying $ \lambda_{M}(\mathbf{D}_{\gamma}^T \mathbf{D}_{\gamma} )   \leq \theta_i \leq \lambda_{1} (\mathbf{D}_{\gamma}^T \mathbf{D}_{\gamma})$. Finally, using the Cauchy interlacing theorem we get the bound
    \begin{align}
       \lambda_{i+(N-R)}( \mathbf{C} ) = \lambda_{i+(N - R)}(\mathbf{U}^T \mathbf{C} \mathbf{U})  \leq  \lambda_i\left( (\mathbf{U}^T \mathbf{C} \mathbf{U})_{1:R,1:R} \right) \leq \lambda_i(\mathbf{U}^T \mathbf{C} \mathbf{U}) = \lambda_i(\mathbf{C}).
    \end{align}
    Incorporating the bound for $\theta_i$ yields 
    \begin{align}\label{eq:gamma_dependent_result}
       \lambda_{i+(N-R)}(\mathbf{C})  \lambda_M(\mathbf{D}_{\gamma}^T \mathbf{D}_{\gamma})  \leq \lambda_i(\mathbf{X}^T \mathbf{C} \mathbf{X}) \leq \lambda_i(\mathbf{C}) \lambda_1(\mathbf{D}_{\gamma}^T \mathbf{D}_{\gamma}).
    \end{align}

    Recall that  \eqref{eq:gamma_dependent_result} holds for all choices of the free parameter $\gamma \in \mathbb{R}$. Considering the choice $\gamma = \sigma_R$ gives $\lambda_M(\mathbf{D}_{\gamma}^T \mathbf{D}_{\gamma}) = \sigma_R^2 = \lambda_R(\mathbf{X}^T \mathbf{X})$, and similarly the choice $\gamma = \sigma_1$ gives $\lambda_1(\mathbf{D}_{\gamma}^T \mathbf{D}_{\gamma}) = \sigma_1^2 = \lambda_1(\mathbf{X}^T \mathbf{X})$. Combining \eqref{eq:gamma_dependent_result} for both of these choices gives the desired bounds
    \begin{align}
         \lambda_{i+(N-R)}(\mathbf{C})  \lambda_R(\mathbf{X}^T \mathbf{X})  \leq \lambda_i(\mathbf{X}^T \mathbf{C} \mathbf{X}) \leq \lambda_i(\mathbf{C}) \lambda_1(\mathbf{X}^T \mathbf{X})
    \end{align}
    for $i = 1, \ldots, R$.

    This argument can be slightly modified to produce the second inequality. When $\mathbf{C}$ is SPSD we may make use of $\mathbf{Z}^{\frac{1}{2}}$ to obtain
    \begin{align}
        \lambda_i(\mathbf{X}^T \mathbf{C} \mathbf{X}) &= \lambda_i(\mathbf{D}_{\gamma}^T \mathbf{Z} \mathbf{D}_{\gamma}) \\ 
        &= \lambda_i( ( \mathbf{D}_{\gamma}^T \mathbf{Z}^{\frac{1}{2}} ) ( \mathbf{Z}^{\frac{1}{2}} \mathbf{D}_{\gamma}) ) \\
        &= \lambda_i( ( \mathbf{Z}^{\frac{1}{2}} \mathbf{D}_{\gamma}) ( \mathbf{D}_{\gamma}^T \mathbf{Z}^{\frac{1}{2}} ) ) \\
        &\leq \lambda_i(\mathbf{D}_{\gamma} \mathbf{D}_{\gamma}^T) \lambda_1( \mathbf{Z} ) \\  
         &= \lambda_i(\mathbf{D}_{\gamma} \mathbf{D}_{\gamma}^T) \lambda_1( (\mathbf{U}^T \mathbf{C} \mathbf{U})_{1:R,1:R} ) \\ 
         &\leq \lambda_i(\mathbf{D}_{\gamma} \mathbf{D}_{\gamma}^T) \lambda_1(  \mathbf{C}  )
    \end{align}
    which yields $\lambda_i(\mathbf{X}^T \mathbf{C} \mathbf{X}) \leq \lambda_i(\mathbf{X}^T \mathbf{X}) \lambda_1Z(\mathbf{C})$ for $i = 1, \ldots, R$ with the choice $\gamma = 0$.

    \emph{Wide case ($N \leq M$):}  Let $\mathbf{X} = \mathbf{U} \begin{bmatrix} \boldsymbol{\Sigma} & \mathbf{0}_{N \times (M-N)}  \end{bmatrix} \mathbf{V}^T$ be the SVD of $\mathbf{X}$, where $\boldsymbol{\Sigma} = \operatorname{diag}(\sigma_1, \ldots, \sigma_N) \in \mathbb{R}^{N \times N}$. Note that we will have $\sigma_{R+1} = \cdots = \sigma_{N} = 0$ (the last $N-R$ singular values are zero). Let $\gamma$ be an arbitrary scalar to be determined later, and note that
    \begin{align}
        \boldsymbol{\Sigma} &= \operatorname{diag}(\overbrace{1, \ldots, 1,}^{R} \underbrace{0, \ldots, 0}_{N-R})  \cdot \operatorname{diag}(\sigma_1, \ldots, \sigma_{R}, \underbrace{\gamma, \cdots, \gamma}_{N-R}) \in \mathbb{R}^{N \times N} \\
        &= \mathbf{J} \mathbf{D}_{\gamma}
    \end{align}
    where we have defined $\mathbf{J}$ and $\mathbf{D}_{\gamma}$ as their corresponding factors in the line above. We then find that
    \begin{align}
        \mathbf{X}^T \mathbf{C} \mathbf{X} &= \mathbf{V} \begin{bmatrix} \boldsymbol{\Sigma}^T \\ \mathbf{0}_{(M-N) \times N} \end{bmatrix} \mathbf{U}^T \mathbf{C} \mathbf{U} \begin{bmatrix}
            \boldsymbol{\Sigma} & \mathbf{0}_{N \times (M-N)} 
        \end{bmatrix} \mathbf{V}^T \\
        &= \mathbf{V} \left( \begin{bmatrix} \mathbf{J}^T \\ \mathbf{0}_{(M-N) \times N } \end{bmatrix} \mathbf{D}_{\gamma}^T \mathbf{U}^T \mathbf{C} \mathbf{U} \mathbf{D}_{\gamma} \begin{bmatrix} \mathbf{J} & \mathbf{0}_{N \times (M-N)} \end{bmatrix}  \right) \mathbf{V}^T \\
        &= \mathbf{V} \begin{bmatrix}
            \left( \mathbf{D}_{\gamma}^T \mathbf{U}^T \mathbf{C} \mathbf{U} \mathbf{D}_{\gamma} \right)_{1:R,1:R} & \mathbf{0}_{R \times (M-R)} \\ \mathbf{0}_{(M-R) \times R} & \mathbf{0}_{(M-R) \times (M-R)}
        \end{bmatrix} \mathbf{V}^T \label{eq:XtCX_wide_final}
    \end{align}
    From \eqref{eq:XtCX_wide_final} we observe that 
    \begin{align}
    \lambda_i(\mathbf{X}^T \mathbf{C} \mathbf{X}) = \lambda_i\left( ( \mathbf{D}_{\gamma}^T  \mathbf{U}^T \mathbf{C} \mathbf{U} \mathbf{D}_{\gamma})_{1:R,1:R} \right), \quad i = 1, \ldots, R,
    \end{align} 
    and that the remaining $M-R$ eigenvalues of $\mathbf{X}^T \mathbf{C} \mathbf{X}$ are zero. Applying the Cauchy interlacing theorem yields 
    \begin{align}
         \lambda_{i + (N-R)}\left( \mathbf{D}_{\gamma}^T \mathbf{U}^T \mathbf{C} \mathbf{U} \mathbf{D}_{\gamma} \right)
 \leq  \lambda_i(\mathbf{X}^T \mathbf{C} \mathbf{X}) \leq \lambda_i( \mathbf{D}_{\gamma}^T \mathbf{U}^T \mathbf{C} \mathbf{U} \mathbf{D}_{\gamma}), \quad i = 1, \ldots, R.
    \end{align}
    
To proceed, we further extend the bounds. For the lower bound, using the usual (square) Ostrowski theorem we obtain
\begin{align}
    \lambda_{i+(N-R)}\left( \mathbf{D}_{\gamma}^T \mathbf{U}^T \mathbf{C} \mathbf{U} \mathbf{D}_{\gamma} \right) &= \theta_i \lambda_{i+(N-R)}(\mathbf{U}^T \mathbf{C} \mathbf{U})  \\
    &= \theta_i \lambda_{i+(N-R)}( \mathbf{C} )
\end{align}
for $i = 1, \ldots, R$, where $\theta_i$ is some number satisfying $ \lambda_N(\mathbf{D}_{\gamma}^T \mathbf{D}_{\gamma}) \leq \theta_i \leq \lambda_1(\mathbf{D}_{\gamma}^T \mathbf{D}_{\gamma})$. 
Considering the choice $\gamma = \sigma_R$  yields $\lambda_N(\mathbf{D}_{\gamma}^T \mathbf{D}_{\gamma}) = \sigma_R^2 = \lambda_R(\mathbf{X}^T \mathbf{X})$ and $\theta_i \lambda_{i+(N-R)}(\mathbf{C}) \geq \lambda_R(\mathbf{X}^T \mathbf{X}) \lambda_{i+(N-R)}(\mathbf{C})$. For the upper bound, using the  (square) Ostrowski theorem we obtain
\begin{align}
    \lambda_{i}(\mathbf{D}_{\gamma}^T \mathbf{U}^T \mathbf{C} \mathbf{U} \mathbf{D}_{\gamma}) 
    &= \xi_i \lambda_i(\mathbf{C}), \quad i = 1, \ldots, R,
\end{align}
 where $\xi_i$ is some number satisfying $\lambda_N(\mathbf{D}_{\gamma}^T \mathbf{D}_{\gamma}) \leq \xi_i \leq \lambda_1(\mathbf{D}_{\gamma}^T \mathbf{D}_{\gamma})$.  Considering the choice $\gamma = \sigma_1$ gives $\lambda_1(\mathbf{D}_{\gamma}^T \mathbf{D}_{\gamma}) = \sigma_1^2 = \lambda_1(\mathbf{X}^T \mathbf{X})$ and $\xi_i \lambda_{i}(\mathbf{C}) \leq \lambda_1(\mathbf{X}^T \mathbf{X}) \lambda_{i}(\mathbf{C})$.  Combining the lower and upper bounds gives 
\begin{align}
     \lambda_{i+(N-R)}( \mathbf{C} ) \lambda_{R}(\mathbf{X}^T \mathbf{X})  \leq \lambda_i(\mathbf{Y}) \leq \lambda_i(\mathbf{X}^T \mathbf{C} \mathbf{X}) \lambda_1(\mathbf{X}^T \mathbf{X}), \quad i = 1, \ldots, R,
\end{align}
as desired, 
with the remaining $M-R$ eigenvalues of $\mathbf{X}^T \mathbf{C} \mathbf{X}$ equal to zero. 

This argument can be slightly modified to produce the second inequality. When $\mathbf{C}$ is SPSD we can write 
\begin{align}
    \lambda_i(\mathbf{X}^T \mathbf{C} \mathbf{X}) &\leq \lambda_i(\mathbf{D}_{\gamma}^T \mathbf{U}^T \mathbf{C} \mathbf{U} \mathbf{D}_{\gamma}) \\
    &= \lambda_i( (\mathbf{C}^{\frac{1}{2}} \mathbf{U} \mathbf{D}_{\gamma}) ( \mathbf{D}_{\gamma}^T \mathbf{U}^T \mathbf{C}^{\frac{1}{2}} ) ) \\
    &\leq \lambda_1(\mathbf{C}) \lambda_i(\mathbf{D}_{\gamma} \mathbf{D}_{\gamma}^T)
\end{align}
which gives $\lambda_i(\mathbf{X}^T \mathbf{C} \mathbf{X}) \leq \lambda_1(\mathbf{C}) \lambda_i(\mathbf{X}^T \mathbf{X})$ with the choice $\gamma = 0$.

\end{proof}

\begin{figure}
    \centering
    \includegraphics[width=1.0\textwidth]{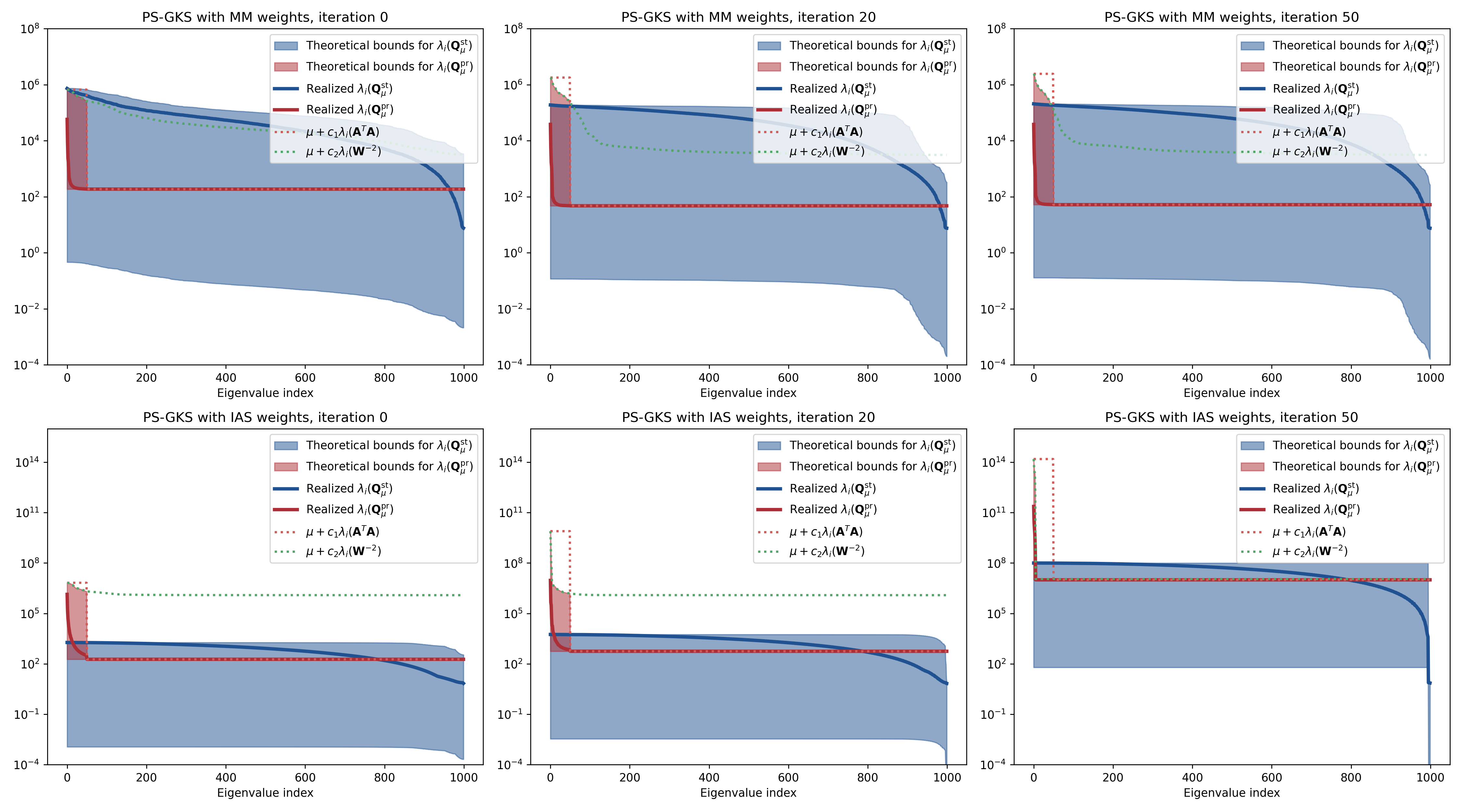}
    \caption{Test 1. Spectrum of $\mathbf{Q}_{\regparam}^{\text{st}}$ and $\mathbf{Q}_{\regparam}^{\text{pr}}$ as the PS-GKS iterations progress. We show the theoretical eigenvalue bounds predicted by the theory developed in \Cref{sub:analysis_of_priorconditioning}, as well as the realized spectra. (top row) Results using MM weights. (bottom row) Results using IAS weights.}
\label{fig:eigenvalue_bounds_plot}
\end{figure}

\section{Computation of pseudoinverses}\label{sec:pseudoinverses} The main computational obstacle introduced by the PS-GKS method when compared with S-GKS is the requirement of computing matvecs with the pseudoinverses $\boldsymbol{\Psi}_{\itidx}^\dagger$ and $(\boldsymbol{\Psi}_{\itidx}^\dagger)^T$. We assume here that $\boldsymbol{\Psi}^{-1}$ is not invertible. Recall that it in general $(\mathbf{W}_{\itidx} \boldsymbol{\Psi_\itidx})^\dagger \neq \boldsymbol{\Psi}_\itidx^\dagger \mathbf{W}_{\itidx}^{-1}$,\footnote{Notable exceptions of when $(\mathbf{W}_{\itidx} \boldsymbol{\Psi_\itidx})^\dagger = \boldsymbol{\Psi}_\itidx^\dagger \mathbf{W}_{\itidx}^{-1}$ holds include when $\boldsymbol{\Psi}$ is invertible or has linearly independent rows.} which means that an offline computation of $\boldsymbol{\Psi}^\dagger$ is not immediately of use. 

For problems of sufficiently low dimension (small $N$), one may employ an approximation to the pseudoinverse such as 
\begin{align}\label{eq:pinv_delta_approx}
    \boldsymbol{\Psi}_{\itidx}^\dagger \approx (\boldsymbol{\Psi}_{\itidx}^T \boldsymbol{\Psi}_{\itidx} + \delta \mathbf{I}_N)^{-1} \boldsymbol{\Psi}_{\itidx}^T, \quad (\boldsymbol{\Psi}_{\itidx}^\dagger)^T \approx \boldsymbol{\Psi}_{\itidx} (\boldsymbol{\Psi}_{\itidx}^T \boldsymbol{\Psi}_{\itidx} + \delta \mathbf{I}_N)^{-1},
\end{align}
for some small $\delta > 0$. This is particularly convenient when $\boldsymbol{\Psi}^T \boldsymbol{\Psi}$ possesses banded structure, since in this case matvecs with the inverses in \cref{eq:pinv_delta_approx} may be applied in $\mathcal{O}(B^2 N)$ flops using a banded Cholesky factorization \cite{rue2005gaussian}, where $B$ denotes the bandwidth. 

For problems of high dimension (large $N$), iterative methods provide an effective means of computing matvecs with the pseudoinverses. One such method was recently proposed in \cite{gazzola2019flexible, gazzola2021flexible},  where the matvec $\boldsymbol{\xi} = \boldsymbol{\Psi}_{\itidx}^\dagger \mathbf{y}$ is computed by applying the LSQR \cite{paige1982lsqr} or LSMR \cite{fong2011lsmr} algorithm to the right-preconditioned least squares problem
\begin{align}
    \min_{\hat{\boldsymbol{\xi}} \in \mathbb{R}^N} \left\{  \| \mathbf{W}_{\itidx} \boldsymbol{\Psi} \mathbf{P}_{\itidx} \hat{\boldsymbol{\xi}} - \mathbf{y}  \|_2^2  \right\}, \quad \boldsymbol{\xi} = \mathbf{P}_{\itidx} \hat{\boldsymbol{\xi}}.
\end{align}
Choices for the preconditioner $\mathbf{P}_{\itidx}$ are suggested as $\mathbf{P}_{\itidx}^{(1)} = \boldsymbol{\Psi}^\dagger$, $\mathbf{P}_{\itidx}^{(2)} = \boldsymbol{\Psi}^\dagger \mathbf{W}_{\itidx}^{-1}$, or a diagonal preconditioned $\mathbf{P}_{\itidx}^{(3)}$ based on row scaling. In our work, we examine  For large 2D imaging inverse problems defined on an $N = N_x \times N_y$ grid, a common choice of sparsifying transformation is one similar to
\begin{align}
    \mathbf{\Psi} = \begin{bmatrix} \boldsymbol{\Psi}^{(N_y)}_{\text{1D}} \otimes \mathbf{I}_{N_x} \\ \mathbf{I}_{N_y} \otimes \boldsymbol{\Psi}^{(N_x)}_{\text{1D}}
    \end{bmatrix} \in \mathbb{R}^{ 2N \times N}, \quad  
   \boldsymbol{\Psi}_{\text{1D}}^{(L)} \coloneqq  \begin{bmatrix}
        1 & -1 & & \\ 
        & \ddots & \ddots & \\ 
        & & 1 & -1 \\
        & & & 0
    \end{bmatrix} \in \mathbb{R}^{L \times L},
\end{align}
which corresponds to an anisotropic two-dimensional discrete gradient operator with Neumann boundary conditions. Note that here $\mathbf{K}$ may be chosen as $\mathbf{K} = \mathbf{1}_N$. For such $\boldsymbol{\Psi}$, using $\boldsymbol{\Psi}^\dagger$ as a preconditioner according to the method of \cite{gazzola2021flexible} requires an upfront expense of $\mathcal{O}(N^{3/2})$ flops in order to build an expression for $\boldsymbol{\Psi}^\dagger$ in terms of the kronecker products and the SVD of $\boldsymbol{\Psi}_{\text{1D}}^{(N_y)}$/$\boldsymbol{\Psi}_{\text{1D}}^{(N_x)}$. 

Here, we use a different method to compute matvecs with $\boldsymbol{\Psi}_{\itidx}^\dagger$ which utilizes a singular CG method with a spectral preconditioner. The advantage of our method is essentially cheaper computation of the pseudoinverse $\boldsymbol{\Psi}^\dagger$ which we evaluate with cost $\mathcal{O}(N \log N)$ at each instance. The matrix  $\boldsymbol{\Psi}^T \boldsymbol{\Psi}$ can be expressed as the sum of specially-structured matrices,\footnote{Specifically, in the Neumann boundary condition case $\boldsymbol{\Psi}^T \boldsymbol{\Psi}$ can be written as the sum of block Toeplitz with Toeplitz blocks (BTTB), block Toeplitz with Hankel blocks (BTHB), block Hankel with Toeplitz blocks (BHTB), and block Hankel with Hankel blocks (BHHB) matrices \cite{hansen2006deblurring}. In the Dirichlet boundary condition case, $\boldsymbol{\Psi}^T \boldsymbol{\Psi}$ is a block Toeplitz with Toeplitz blocks (BTTB) matrix and can be diagonalized by the type I discrete sine transform (DST). In the periodic boundary condition case, $\boldsymbol{\Psi}^T \boldsymbol{\Psi}$ is a block circulant with circulant blocks (BCCB) matrix and can be diagonalized by the discrete Fourier transform (DFT).} such that it can be diagonalized \emph{a priori} by the (orthonormal, type II) two-dimensional discrete cosine transform (DCT) (e.g., see \cite{hansen2006deblurring, Strang1999DCT, Makhoul1980FastCosine}). Specifically, letting $\mathbf{B}$ denote the DCT for a $N_x \times N_y$ grid, it holds that
\begin{align}\label{eq:dct_diagonalization}
  \mathbf{M} \coloneqq \boldsymbol{\Psi}^T \boldsymbol{\Psi} = \mathbf{B}^T \boldsymbol{\Lambda} \mathbf{B},
\end{align}
where $\boldsymbol{\Lambda}$ is a diagonal matrix with nonnegative entries containing the eigenvalues of $\boldsymbol{\Psi}^T \boldsymbol{\Psi}$, and $\mathbf{B}^T = \mathbf{B}^{-1}$ denotes the inverse DCT. Note that the eigenvalues are quickly computed as $\boldsymbol{\Lambda} = \operatorname{diag}( (\mathbf{B} \boldsymbol{\Psi}^T \boldsymbol{\Psi} \mathbf{B}^T \mathbf{y}) \oslash \mathbf{y})$
for a vector $\mathbf{y} \in \mathbb{R}^N$ with nonzero entries and $\oslash$ denoting component-wise division. To compute the pseudoinverse matvec $\boldsymbol{\Psi}_{\itidx} ^\dagger \mathbf{y}$, we recall the identity $\mathbf{C}^\dagger = (\mathbf{C}^T \mathbf{C})^\dagger \mathbf{C}^T$, so the problem reduces to computing a matvec with $(\boldsymbol{\Psi}^T \mathbf{W}_\itidx^2 \boldsymbol{\Psi})^\dagger$. The matvec $\boldsymbol{\xi} = (\boldsymbol{\Psi}^T \mathbf{W}_{\itidx}^2 \boldsymbol{\Psi})^\dagger \mathbf{y}$ may be obtained by applying the CG method with an initialization $\boldsymbol{\xi}_0 \in \operatorname{col}(\boldsymbol{\Psi}^T)$ to the symmetric semipositive-definite system
\begin{align}
    (\boldsymbol{\Psi}^T \mathbf{W}_{\itidx}^2 \boldsymbol{\Psi}) \boldsymbol{\xi} = \mathbf{y}, 
\end{align}
which may be preconditioned using  $\mathbf{M}$ given in \cref{eq:dct_diagonalization} as a DCT preconditioner. See \cite[Appendix B]{lindbloom2024generalized} for details. 

We emphasize that our method for computing $\boldsymbol{\Psi}^\dagger$ requires $\mathcal{O}(N \log N)$ flops at each instance, especially for large-scale imaging problems. In practice, our method can be massively accelerated by using a highly efficient GPU implementation of the DCT provided by a library such as \texttt{clFFT} \cite{clfft_github}, \texttt{cuFFT} \cite{nvidia_cufft}, or \texttt{CuPy} \cite{nishino2017cupy}.

\section{Priorconditioned GKB and FGK methods}\label{ssec:psgkb_and_fgk} The methods discussed in the paper discussed so far fall under the umbrella of GKS-type hybrid projection methods; our PS-GKS method is the first of GKS-type that utilizes priorconditioned subspaces. There exists a competing class of priorconditioned Golub-Kahan (GK) -type hybrid projection methods based on the Golub-Kahan bidiagonalization (GKB) algorithm or the flexible Golub-Kahan (FGK) decomposition \cite{chung2019flexible}.

\subsection{GKB method}\label{ssub:gkb} It is straightforward to devise an analogue of our PS-GKS based on the GKB process, which we will refer to as PS-GKB. Such a method was first proposed in \cite{gazzola2021flexible}. At the $\itidx$th iteration of PS-GKB we form the priorconditioned problem
\begin{align}\label{eq:priorconditioned_fullscale_GKB}
    \argmin_{ \mathbf{x} \in \mathbb{R}^N } \left\{ \| \mathbf{A}\mathbf{x} - \mathbf{b} \|_2^2 + \regparam_{\itidx} \| \boldsymbol{\Psi}_{\itidx} \mathbf{x}  \|_2^2 \right\} = (\boldsymbol{\Psi}_{\itidx+1})_{\mathbf{A}}^\dagger \left( \argmin_{\mathbf{z} \in \mathbb{R}^{K}} \left\{ \| \overline{\mathbf{A}}_{\itidx} \mathbf{z} - \overline{\mathbf{b}} \|_2^2 + \regparam_{\itidx} \| \mathbf{z} \|_2^2 \right\} \right) + \mathbf{x}_{\text{ker}}
\end{align}
where $\overline{\mathbf{b}} = \mathbf{b} - \mathbf{A} \mathbf{x}_{\text{ker}}$ and $\mathbf{x}_{\text{ker}} = \mathbf{K} (\mathbf{A} \mathbf{K})^\dagger \mathbf{b}$. Next, instead of projecting the problem in the RHS of \cref{eq:priorconditioned_fullscale_GKB} onto a generalized Krylov subspace as in PS-GKS, we instead project it onto $\operatorname{col}(\mathbf{V}_{\itidx})$ where $\mathbf{V}_\itidx$ arises from the Golub-Kahan bidiagonalization
\begin{align}\label{eq:GKB}
    \overline{\mathbf{A}}_{\itidx} \mathbf{V}_{\itidx} = \mathbf{U}_{\itidx+1} \mathbf{B}_{\itidx}, \quad \overline{\mathbf{A}}_{\itidx}^T \mathbf{U}_{\itidx+1} = \mathbf{V}_{\itidx+1} \hat{\mathbf{B}}_{\itidx+1}^T
\end{align}
for matrices $\mathbf{U}_{\itidx+1} \in \mathbb{R}^{M \times (\itidx+1)}$,  $\mathbf{V}_{\itidx+1} \in \mathbb{R}^{N \times (\itidx+1)}$ with orthonormal columns and lower bidiagonal matrices $\mathbf{B}_{\itidx} \in \mathbb{R}^{(\itidx+1)\times \itidx}, \hat{\mathbf{B}}_{\itidx+1} \in \mathbb{R}^{(\itidx+1)\times(\itidx+1)}$. Using \cref{eq:GKB}, the projected problem can be expressed as
\begin{align}\label{eq:projected_problem_GKB}
    \mathbf{y}_\itidx = \argmin_{\mathbf{y} \in \mathbb{R}^{\itidx} } \left\{ \| \mathbf{B}_{\itidx} \mathbf{y} - \| \overline{\mathbf{b}} \|_2 \mathbf{e}_1 \|_2^2 + \regparam_{\itidx} \| \mathbf{y}  \|_2^2  \right\}
\end{align}
where $\mathbf{e}_1 = [1, 0, \ldots, 0]^T \in \mathbb{R}^{\itidx+1}$ and the solution at the $\itidx$th iteration is recovered as $\mathbf{x}_{\itidx} = \mathbf{x}_{\text{ker}} +  (\boldsymbol{\Psi}_{\itidx})_{\mathbf{A}}^\dagger \mathbf{V}_{\itidx} \mathbf{y}_{\itidx}$. The parameter $\regparam_{\itidx}$ can be selected using a regularization parameter selection method such as DP. This process is repeated for increasing $\itidx$, where a new weighting matrix $\mathbf{W}_{\itidx}$ is computed at each iteration. We note that the $\itidx$th iteration of PS-GKB requires $\mathcal{O}(\itidx)$ matvecs with $\mathbf{A}/\mathbf{A}^T$ and $(\boldsymbol{\Psi_{\itidx}})_{\mathbf{A}}^\dagger/((\boldsymbol{\Psi_{\itidx}})_{\mathbf{A}}^\dagger)^T$ which is the same as the $\itidx$th iteration of PS-GKS.

\begin{algorithm}[!ht]
\caption{The PS-GKB method}%
\label{alg:psgkb}
\begin{algorithmic}[1]
	\Require{$\bA, \boldsymbol{\Psi}, \mathbf{b},  \mathbf{K}$} 
 \Ensure{An approximate solution $\bx_{\itidx+1}$}
	\Function{$\bx_{\itidx+1} = $ PS-GKB }{$\bA, \boldsymbol{\Psi}, \mathbf{b}, \mathbf{K}$} \;
        \noindent \State $\mathbf{A} \mathbf{K} = \mathbf{Q}_{\text{ker}} \mathbf{R}_{\text{ker}}$ and $(\mathbf{A} \mathbf{K})^\dagger = \mathbf{R}_{\text{ker}}^{-1} \mathbf{Q}_{\text{ker}}^T$\;\Comment{$(\mathbf{A} \mathbf{K})^\dagger$ via economic QR}
        \State $\mathbf{x}_{\text{ker}} = \mathbf{K} \mathbf{R}_{\text{ker}}^{-1} \mathbf{Q}_{\text{ker}}^T \mathbf{b}$ and $\overline{\mathbf{b}} = \mathbf{b} - \mathbf{A} \mathbf{x}_{\text{ker}}$\; \Comment{Fixed component in $\ker(\boldsymbol{\Psi})$}
\FOR {$\itidx=1,2,\ldots$ \text{until convergence}}{
\State Update weights $\mathbf{W}_{\itidx} = \operatorname{diag}(\mathbf{w}_{\itidx})$ and $\boldsymbol{\Psi}_{\itidx} = \mathbf{W}_{\itidx} \boldsymbol{\Psi}$ given $\boldsymbol{\Psi} \mathbf{x}_{\itidx-1}$\;
\State Build operators for  $\boldsymbol{\Psi}_{\itidx}^\dagger$, $( \boldsymbol{\Psi}_{\itidx})_{\mathbf{A}}^\dagger = ( \mathbf{I}_N - \mathbf{K}(\mathbf{A} \mathbf{K})^\dagger \mathbf{A} ) \boldsymbol{\Psi}_{\itidx}^\dagger$, and $\overline{\mathbf{A}}_{\itidx} = \mathbf{A} ( \boldsymbol{\Psi}_{\itidx})_{\mathbf{A}}^\dagger$\;
\State Compute the GKB $\overline{\mathbf{A}}_{\itidx} \mathbf{V}_{\itidx} = \mathbf{U}_{\itidx+1} \mathbf{B}_{\itidx}$, $\overline{\mathbf{A}}_{\itidx}^T \mathbf{U}_{\itidx+1} = \mathbf{V}_{\itidx+1}\; 
\hat{\mathbf{B}}_{\itidx+1}^T$\;
\State Select $\regparam_{\itidx}$ by heuristic (e.g., DP) on \cref{eq:projected_problem_GKB}\;\Comment{Regularization parameter selection} 
\State $\mathbf{y}_{\itidx}$ to satisfy \cref{eq:projected_problem_GKB} with selected $\regparam_{\itidx}$\;
\State $\mathbf{x}_{\itidx} =  (\boldsymbol{\Psi}_{\itidx})_{\mathbf{A}}^\dagger \mathbf{V}_{\itidx} \mathbf{y}_{\itidx} + \mathbf{x}_{\text{ker}}$ }
\ENDFOR
	\EndFunction
\end{algorithmic}
\end{algorithm}

\subsection{FGK methods}\label{ssub:fgk} FGK hybrid projection methods are based on the flexible preconditioning framework of \cite{notay2000flexible, saad1993flexible, simoncini2007recent}. The main differences between FGK methods and the PS-GKB method are that FGK methods require only a single matvec with $\mathbf{A}/\mathbf{A}^T$ and $(\boldsymbol{\Psi_{\itidx}})_{\mathbf{A}}^\dagger/((\boldsymbol{\Psi_{\itidx}})_{\mathbf{A}}^\dagger)^T$ in each iteration, and that the approximation subspace of FGK methods depends on the entire history of weight matrices $\{ \mathbf{W}_j  \}_{j \leq \itidx}$, while the approximation subspace of PS-GKB depends on only the latest weight matrix $\mathbf{W}_\itidx$.

Assuming for the moment that $\boldsymbol{\Psi}^{-1}$ exists, the FGK process produces the factorization
\begin{align}\label{eq:FGK_process}
    \mathbf{A} \mathbf{Z}_{\itidx} = \mathbf{U}_{\itidx+1} \mathbf{M}_{\itidx}, \quad \mathbf{A}^T \mathbf{U}_{\itidx+1} = \mathbf{V}_{\itidx+1} \mathbf{S}_{\itidx+1},
\end{align}
where the columns of $\mathbf{U}_{\itidx+1}$ and $\mathbf{V}_{\itidx+1}$ are orthonormal, $\mathbf{M}_{\itidx}$ is upper Hessenberg, and $\mathbf{S}_{\itidx+1}$ is upper triangular. The approximation subspace for the solution at the $\itidx$th iteration is taken to be $\operatorname{col}(\mathbf{Z}_{\itidx})$, where several choices of $\mathbf{Z}_{\itidx}$ appear in the literature. When $\boldsymbol{\Psi}^{-1}$ exists, the choice $\mathbf{Z}_{\itidx}^{(a)} = [ \boldsymbol{\Psi}_1^{-T} \mathbf{v}_1, \ldots, \boldsymbol{\Psi}_{\itidx}^{-T} \mathbf{v}_{\itidx} ]$ corresponds to that of \cite{chung2019flexible}. For general $\boldsymbol{\Psi}$, the choice 
\begin{align}
    \mathbf{Z}_{\itidx}^{(b)} = \begin{bmatrix} (\boldsymbol{\Psi}_1)_{\mathbf{A}}^\dagger ( (\boldsymbol{\Psi}_1)_{\mathbf{A}}^\dagger)^T \mathbf{v}_1 \quad  \cdots \quad (\boldsymbol{\Psi}_{\itidx})_{\mathbf{A}}^\dagger ( (\boldsymbol{\Psi}_{\itidx})_{\mathbf{A}}^\dagger)^T \mathbf{v}_{\itidx} \end{bmatrix}
\end{align}
has been proposed in  \cite{gazzola2021flexible, gazzola2021iteratively} and is what we employ for the FGK methods in this investigation. We assume throughout that the FGK process is break-down free, i.e., that $\operatorname{rank}(\mathbf{Z}_\itidx) = \itidx$ for each $\itidx$. 

In this investigation, we only consider FGK methods corresponding to the ``first-regularize-then-project'' framework \cite[\S 6.4]{hansen2010discrete}. Included in this framework are the IRW-LSQR method of \cite{gazzola2021iteratively} and the F-TV method method of \cite{gazzola2021flexible}. Such FGK methods utilize the projected problem
\begin{align}\label{eq:FGK_iteration}
    \mathbf{x}_{\itidx} =  \mathbf{x}_{\text{ker}} + \argmin_{\mathbf{x} \in \operatorname{col}(\mathbf{Z}_{\itidx})} \left\{ \|\mathbf{A} \mathbf{x} - \mathbf{b} \|_2^2 + \regparam_{\itidx} \| \boldsymbol{\Psi}_{\itidx} \mathbf{x}  \|_2^2 \right\}, 
\end{align}
in the $\itidx$th iteration. Inserting \cref{eq:FGK_process} into the minimization in \cref{eq:FGK_iteration} produces the equivalent problem
\begin{align}\label{eq:fgk_projected_problem_4}
    \mathbf{y}_{\itidx} = \argmin_{\mathbf{y} \in \mathbb{R}^{\itidx} } \,\, \| \mathbf{M}_{\itidx} \mathbf{y} - \| \mathbf{b} \|_2 \mathbf{e}_1 \|_2^2 + \regparam_{\itidx} \| \mathbf{R}_{\boldsymbol{\Psi}} \mathbf{y}  \|_2^2
\end{align}
where $\boldsymbol{\Psi}_{\itidx} \mathbf{Z}_{\itidx} = \mathbf{Q}_{\boldsymbol{\Psi}} \mathbf{R}_{\boldsymbol{\Psi}}$ denotes an economic QR factorization. The solution at the $\itidx$th iteration is recovered as $\mathbf{x}_{\itidx} = \mathbf{x}_{\text{ker}} + \mathbf{Z}_{\itidx} \mathbf{y}_{\itidx}$.

\begin{algorithm}[!ht]
\caption{The flexible Golub-Kahan (FGK) method}%
\label{alg:fgk}
\begin{algorithmic}[1]
	\Require{$\bA, \boldsymbol{\Psi}, \mathbf{b}, \bx_0, \mathbf{K}$} 
 \Ensure{An approximate solution $\bx_{k+1}$}
	\Function{$\bx_{k+1} = $ PS-GKS }{$\bA, \boldsymbol{\Psi}, \mathbf{b}, \bx_{0}, \mathbf{K}$} \;
        \noindent \State $\mathbf{A} \mathbf{K} = \mathbf{Q}_{\text{ker}} \mathbf{R}_{\text{ker}}$ and $(\mathbf{A} \mathbf{K})^\dagger = \mathbf{R}_{\text{ker}}^{-1} \mathbf{Q}_{\text{ker}}^T$\;\Comment{$(\mathbf{A} \mathbf{K})^\dagger$ via economic QR}
        \State $\mathbf{x}_{\text{ker}} = \mathbf{K} \mathbf{R}_{\text{ker}}^{-1} \mathbf{Q}_{\text{ker}}^T \mathbf{b}$ and $\overline{\mathbf{b}} = \mathbf{b} - \mathbf{A} \mathbf{x}_{\text{ker}}$\;\Comment{Fixed component in $\ker(\boldsymbol{\Psi})$}
    \State $\mathbf{u}_1 = \overline{\mathbf{b}}/ \| \overline{\mathbf{b}} \|_2$,  $\mathbf{U}_1 = [\mathbf{u}_1]$, $\mathbf{V}_0 = [\,]$, $\mathbf{Z}_0 = [\,]$
\FOR {$\itidx=1,2,\ldots$ \text{until convergence}}{
\State $\overline{\mathbf{v}}_{\itidx} = \mathbf{A}^T \mathbf{u}_{\itidx}$, $\mathbf{v}_{\itidx} = (\mathbf{I} - \mathbf{V}_{\itidx-1} \mathbf{V}_{\itidx-1}^T) \overline{\mathbf{v}}_\itidx$, $\mathbf{v}_{\itidx} = \mathbf{v}_{\itidx}/\|\mathbf{v}_{\itidx}\|_2$, $\mathbf{V}_{\itidx} = [\mathbf{V}_{\itidx-1} \,\, \mathbf{v}_{\itidx}]$\;
\State Update weights $\mathbf{W}_{\itidx} = \operatorname{diag}(\mathbf{w}_{\itidx})$ and $\boldsymbol{\Psi}_{\itidx} = \mathbf{W}_{\itidx} \boldsymbol{\Psi}$ given $\boldsymbol{\Psi}\mathbf{x}_{\itidx-1}$\;
\State Build operators for  $\boldsymbol{\Psi}_{\itidx}^\dagger$, $( \boldsymbol{\Psi}_{\itidx})_{\mathbf{A}}^\dagger = ( \mathbf{I}_N - \mathbf{K}(\mathbf{A} \mathbf{K})^\dagger \mathbf{A} ) \boldsymbol{\Psi}_{\itidx}^\dagger$, and $\overline{\mathbf{A}}_{\itidx} = \mathbf{A} ( \boldsymbol{\Psi}_{\itidx})_{\mathbf{A}}^\dagger$\;
\State $\mathbf{z}_{\itidx} = ((\boldsymbol{\Psi}_{\itidx})_{\mathbf{A}}^\dagger)^T (\boldsymbol{\Psi}_{\itidx})_{\mathbf{A}}^\dagger \mathbf{v}_{\itidx}$, $\mathbf{Z}_{\itidx} = [\mathbf{Z}_{\itidx-1} \,\, \mathbf{z}_{\itidx}]$\;
\State $\overline{\mathbf{u}}_{\itidx+1} = \mathbf{A} \mathbf{z}_{\itidx}$, $\mathbf{u}_{\itidx+1} = (\mathbf{I} - \mathbf{U}_{\itidx} \mathbf{U}_{\itidx}^T) \overline{\mathbf{u}}_{\itidx+1}$, $\mathbf{u}_{\itidx+1} = \mathbf{u}_{\itidx+1}/\| \mathbf{u}_{\itidx+1} \|_2$, $\mathbf{U}_{\itidx+1} = [ \mathbf{U}_{\itidx} \,\, \mathbf{u}_{\itidx+1}]$\;
\State Select $\regparam_{\itidx}$ by heuristic (e.g., DP) \;\Comment{Regularization parameter selection}
\State $\mathbf{y}_{\itidx}$ to satisfy \cref{eq:fgk_projected_problem_4} with selected $\regparam_{\itidx}$\;
\State $\mathbf{x}_{\itidx} = \mathbf{x}_{\text{ker}} + \mathbf{Z}_{\itidx} \mathbf{y}_{\itidx}$ }
\ENDFOR
	\EndFunction
\end{algorithmic}
\end{algorithm}

\bibliographystyle{siamplain}
\bibliography{Arxiv}

\begin{thebibliography}{10}

\bibitem{andrews1974scale}
{\sc D.~F. Andrews and C.~L. Mallows}, {\em Scale mixtures of normal distributions}, Journal of the Royal Statistical Society: Series B (Methodological), 36 (1974), pp.~99--102.

\bibitem{beale1959scale}
{\sc E.~M.~L. Beale and C.~L. Mallows}, {\em Scale mixing of symmetric distributions with zero means}, The Annals of Mathematical Statistics,  (1959), pp.~1145--1151.

\bibitem{beck2017first}
{\sc A.~Beck}, {\em First-Order Methods in Optimization}, SIAM, 2017.

\bibitem{buccini2020modulus}
{\sc A.~Buccini, M.~Pasha, and L.~Reichel}, {\em {Modulus-based iterative methods for constrained $\ell_p-\ell_q$ minimization}}, Inverse Problems, 36 (2020), p.~084001.

\bibitem{buccini2023limited}
{\sc A.~Buccini and L.~Reichel}, {\em Limited memory restarted $\ell_p$-$\ell_q$ minimization methods using generalized {K}rylov subspaces}, Advances in Computational Mathematics, 49 (2023), p.~26.

\bibitem{calvetti2018bayes}
{\sc D.~Calvetti, F.~Pitolli, E.~Somersalo, and B.~Vantaggi}, {\em Bayes meets {K}rylov: Statistically inspired preconditioners for {CGLS}}, SIAM Review, 60 (2018), pp.~429--461.

\bibitem{calvetti2020sparsity}
{\sc D.~Calvetti, M.~Pragliola, and E.~Somersalo}, {\em Sparsity promoting hybrid solvers for hierarchical {B}ayesian inverse problems}, SIAM Journal on Scientific Computing, 42 (2020), pp.~A3761--A3784.

\bibitem{calvetti2020sparse}
{\sc D.~Calvetti, M.~Pragliola, E.~Somersalo, and A.~Strang}, {\em Sparse reconstructions from few noisy data: analysis of hierarchical {B}ayesian models with generalized gamma hyperpriors}, Inverse Problems, 36 (2020), p.~025010.

\bibitem{calvetti2003tikhonov}
{\sc D.~Calvetti and L.~Reichel}, {\em {T}ikhonov regularization of large linear problems}, BIT Numerical Mathematics, 43 (2003), pp.~263--283.

\bibitem{calvetti2007gaussian}
{\sc D.~Calvetti and E.~Somersalo}, {\em A {G}aussian hypermodel to recover blocky objects}, Inverse Problems, 23 (2007), p.~733.

\bibitem{calvetti2023bayesian}
{\sc D.~Calvetti and E.~Somersalo}, {\em Bayesian Scientific Computing}, vol.~215, Springer Nature, 2023.

\bibitem{calvetti2019hierachical}
{\sc D.~Calvetti, E.~Somersalo, and A.~Strang}, {\em Hierachical {B}ayesian models and sparsity: $\ell_2$-magic}, Inverse Problems, 35 (2019), p.~035003.

\bibitem{carvalho2009handling}
{\sc C.~M. Carvalho, N.~G. Polson, and J.~G. Scott}, {\em Handling sparsity via the horseshoe}, in Artificial Intelligence and Statistics, PMLR, 2009, pp.~73--80.

\bibitem{chung2019flexible}
{\sc J.~Chung and S.~Gazzola}, {\em Flexible {K}rylov methods for $\ell_p$ regularization}, SIAM Journal on Scientific Computing, 41 (2019), pp.~S149--S171.

\bibitem{chung2024computational}
{\sc J.~Chung and S.~Gazzola}, {\em Computational methods for large-scale inverse problems: a survey on hybrid projection methods}, Siam Review, 66 (2024), pp.~205--284.

\bibitem{ciaramella2022iterative}
{\sc G.~Ciaramella and M.~J. Gander}, {\em Iterative methods and preconditioners for systems of linear equations}, SIAM, 2022.

\bibitem{clfft_github}
{\sc {clMathLibraries}}, {\em {clFFT} library {GitHub} homepage}.
\newblock \url{https://github.com/clMathLibraries/clFFT}.
\newblock [n. d.]. Accessed: 2025-05-02.

\bibitem{dong2023inducing}
{\sc Y.~Dong and M.~Pragliola}, {\em Inducing sparsity via the horseshoe prior in imaging problems}, Inverse Problems, 39 (2023), p.~074001.

\bibitem{elden1977algorithms}
{\sc L.~Eld{\'e}n}, {\em Algorithms for the regularization of ill-conditioned least squares problems}, BIT Numerical Mathematics, 17 (1977), pp.~134--145.

\bibitem{flock2024continuous}
{\sc R.~Flock, Y.~Dong, F.~Uribe, and O.~Zahm}, {\em Continuous {G}aussian mixture solution for linear {B}ayesian inversion with application to {L}aplace priors}, arXiv preprint arXiv:2408.16594,  (2024).

\bibitem{fong2011lsmr}
{\sc D.~C.-L. Fong and M.~Saunders}, {\em {LSMR}: An iterative algorithm for sparse least-squares problems}, SIAM Journal on Scientific Computing, 33 (2011), pp.~2950--2971.

\bibitem{gazzola2021iteratively}
{\sc S.~Gazzola, J.~G. Nagy, and M.~S. Landman}, {\em Iteratively reweighted {FGMRES} and {FLSQR} for sparse reconstruction}, SIAM Journal on Scientific Computing, 43 (2021), pp.~S47--S69.

\bibitem{gazzola2019flexible}
{\sc S.~Gazzola and M.~Sabat{\'e}~Landman}, {\em Flexible {GMRES} for total variation regularization}, BIT Numerical Mathematics, 59 (2019), pp.~721--746.

\bibitem{gazzola2020krylov}
{\sc S.~Gazzola and M.~Sabat{\'e}~Landman}, {\em Krylov methods for inverse problems: Surveying classical, and introducing new, algorithmic approaches}, GAMM-Mitteilungen, 43 (2020), p.~e202000017.

\bibitem{gazzola2021flexible}
{\sc S.~Gazzola, S.~J. Scott, and A.~Spence}, {\em Flexible {K}rylov methods for edge enhancement in imaging}, Journal of Imaging, 7 (2021), p.~216.

\bibitem{glaubitz2024leveraging}
{\sc J.~Glaubitz and A.~Gelb}, {\em Leveraging joint sparsity in hierarchical {B}ayesian learning}, SIAM/ASA Journal on Uncertainty Quantification, 12 (2024), pp.~442--472.

\bibitem{glaubitz2023generalized}
{\sc J.~Glaubitz, A.~Gelb, and G.~Song}, {\em Generalized sparse {B}ayesian learning and application to image reconstruction}, SIAM/ASA Journal on Uncertainty Quantification, 11 (2023), pp.~262--284.

\bibitem{golub2013matrix}
{\sc G.~H. Golub and C.~F. Van~Loan}, {\em Matrix Computations}, JHU Press, 2013.

\bibitem{golub1997generalized}
{\sc G.~H. Golub and U.~Von~Matt}, {\em Generalized cross-validation for large-scale problems}, Journal of Computational and Graphical Statistics, 6 (1997), pp.~1--34.

\bibitem{hansen2010discrete}
{\sc P.~C. Hansen}, {\em Discrete Inverse Problems: Insight and Algorithms}, SIAM, 2010.

\bibitem{hansen2013oblique}
{\sc P.~C. Hansen}, {\em Oblique projections and standard-form transformations for discrete inverse problems}, Numerical Linear Algebra with Applications, 20 (2013), pp.~250--258.

\bibitem{hansen2006deblurring}
{\sc P.~C. Hansen, J.~G. Nagy, and D.~P. O'leary}, {\em Deblurring Images: Matrices, Spectra, and Filtering}, SIAM, 2006.

\bibitem{higham1998modifying}
{\sc N.~J. Higham and S.~H. Cheng}, {\em Modifying the inertia of matrices arising in optimization}, Linear Algebra and its Applications, 275 (1998), pp.~261--279.

\bibitem{hochstenbach2010iterative}
{\sc M.~E. Hochstenbach and L.~Reichel}, {\em An iterative method for {T}ikhonov regularization with a general linear regularization operator}, The Journal of Integral Equations and Applications,  (2010), pp.~465--482.

\bibitem{huang2017majorization}
{\sc G.~Huang, A.~Lanza, S.~Morigi, L.~Reichel, and F.~Sgallari}, {\em Majorization--minimization generalized {K}rylov subspace methods for $\ell_p-\ell_q$ optimization applied to image restoration}, BIT, 57 (2017), pp.~351--378.

\bibitem{hunter2004tutorial}
{\sc D.~R. Hunter and K.~Lange}, {\em A tutorial on {MM} algorithms}, The American Statistician, 58 (2004), pp.~30--37.

\bibitem{jiang2021hybrid}
{\sc J.~Jiang, J.~Chung, and E.~De~Sturler}, {\em Hybrid projection methods with recycling for inverse problems}, SIAM Journal on Scientific Computing, 43 (2021), pp.~S146--S172.

\bibitem{johnson1985matrix}
{\sc C.~R. Johnson and R.~A. Horn}, {\em Matrix Analysis}, Cambridge University Press, 1985.

\bibitem{lampe2012large}
{\sc J.~Lampe, L.~Reichel, and H.~Voss}, {\em Large-scale {T}ikhonov regularization via reduction by orthogonal projection}, Linear Algebra and its Applications, 436 (2012), pp.~2845--2865.

\bibitem{lan2023spatiotemporal}
{\sc S.~Lan, M.~Pasha, S.~Li, and W.~Shen}, {\em Spatiotemporal {B}esov priors for {B}ayesian inverse problems}, arXiv preprint arXiv:2306.16378,  (2023).

\bibitem{lanza2015generalized}
{\sc A.~Lanza, S.~Morigi, L.~Reichel, and F.~Sgallari}, {\em A generalized {K}rylov subspace method for $\ell_p$-$\ell_q$ minimization}, SIAM Journal on Scientific Computing, 37 (2015), pp.~S30--S50.

\bibitem{lawson1995solving}
{\sc C.~L. Lawson and R.~J. Hanson}, {\em Solving least squares problems}, SIAM, 1995.

\bibitem{lindbloom2024generalized}
{\sc J.~Lindbloom, J.~Glaubitz, and A.~Gelb}, {\em Efficient sparsity-promoting {MAP} estimation for {B}ayesian linear inverse problems}, Inverse Problems, 41 (2025), p.~025001.

\bibitem{Makhoul1980FastCosine}
{\sc J.~Makhoul}, {\em A fast cosine transform in one and two dimensions}, IEEE Transactions on Acoustics, Speech, and Signal Processing, 28 (1980), pp.~27--34.

\bibitem{morozov1984methods}
{\sc V.~A. Morozov}, {\em Methods for Solving Incorrectly Posed Problems}, Springer, New York, 1984.

\bibitem{nishimura2023prior}
{\sc A.~Nishimura and M.~A. Suchard}, {\em Prior-preconditioned conjugate gradient method for accelerated {Gibbs} sampling in “large n, large p” {B}ayesian sparse regression}, Journal of the American Statistical Association, 118 (2023), pp.~2468--2481.

\bibitem{nishino2017cupy}
{\sc R.~Nishino and S.~H.~C. Loomis}, {\em {CuPy}: A {NumPy}-compatible library for {NVIDIA} {GPU} calculations}, 31st conference on neural information processing systems, 151 (2017).

\bibitem{notay2000flexible}
{\sc Y.~Notay}, {\em Flexible conjugate gradients}, SIAM Journal on Scientific Computing, 22 (2000), pp.~1444--1460.

\bibitem{nvidia_cufft}
{\sc {NVIDIA Corporation}}, {\em cu{FFT} library documentation}.
\newblock \url{https://docs.nvidia.com/cuda/cufft/index.html}.
\newblock [n. d.]. Accessed: 2025-05-02.

\bibitem{paige1982lsqr}
{\sc C.~C. Paige and M.~A. Saunders}, {\em {LSQR}: An algorithm for sparse linear equations and sparse least squares}, ACM Transactions on Mathematical Software (TOMS), 8 (1982), pp.~43--71.

\bibitem{park2008bayesian}
{\sc T.~Park and G.~Casella}, {\em The {B}ayesian lasso}, Journal of the American Statistical Association, 103 (2008), pp.~681--686.

\bibitem{parlett1998symmetric}
{\sc B.~N. Parlett}, {\em The symmetric eigenvalue problem}, SIAM, 1998.

\bibitem{pasha2023recycling}
{\sc M.~Pasha, E.~de~Sturler, and M.~E. Kilmer}, {\em Recycling {MMGKS} for large-scale dynamic and streaming data}, arXiv preprint arXiv:2309.15759,  (2023).

\bibitem{pasha2024trips}
{\sc M.~Pasha, S.~Gazzola, C.~Sanderford, and U.~O. Ugwu}, {\em {TRIPs-Py}: Techniques for regularization of inverse problems in {P}ython}, Numerical Algorithms,  (2024), pp.~1--38.

\bibitem{pasha2023computational}
{\sc M.~Pasha, A.~K. Saibaba, S.~Gazzola, M.~I. Espa{\~n}ol, and E.~de~Sturler}, {\em A computational framework for edge-preserving regularization in dynamic inverse problems}, Electronic Transactions on Numerical Analysis, 58 (2023), pp.~486--516.

\bibitem{reichel2008new}
{\sc L.~Reichel and A.~Shyshkov}, {\em A new zero-finder for {T}ikhonov regularization}, BIT Numerical Mathematics, 48 (2008), pp.~627--643.

\bibitem{rue2005gaussian}
{\sc H.~Rue and L.~Held}, {\em Gaussian Markov Random Fields: Theory and Applications}, CRC press, 2005.

\bibitem{saad1993flexible}
{\sc Y.~Saad}, {\em A flexible inner-outer preconditioned {GMRES} algorithm}, SIAM Journal on Scientific Computing, 14 (1993), pp.~461--469.

\bibitem{simoncini2007recent}
{\sc V.~Simoncini and D.~B. Szyld}, {\em Recent computational developments in {K}rylov subspace methods for linear systems}, Numerical Linear Algebra with Applications, 14 (2007), pp.~1--59.

\bibitem{Strang1999DCT}
{\sc G.~Strang}, {\em The discrete cosine transform}, SIAM Review, 41 (1999), pp.~135--147.

\bibitem{stuart2010inverse}
{\sc A.~M. Stuart}, {\em Inverse problems: A {B}ayesian perspective}, Acta Numerica, 19 (2010), pp.~451--559.

\bibitem{tipping2001sparse}
{\sc M.~E. Tipping}, {\em Sparse {B}ayesian learning and the relevance vector machine}, Journal of Machine Learning Research, 1 (2001), pp.~211--244.

\bibitem{uribe2023horseshoe}
{\sc F.~Uribe, Y.~Dong, and P.~C. Hansen}, {\em Horseshoe priors for edge-preserving linear {B}ayesian inversion}, SIAM Journal on Scientific Computing, 45 (2023), pp.~B337--B365.

\bibitem{vogel2002computational}
{\sc C.~R. Vogel}, {\em Computational Methods for Inverse Problems}, SIAM, 2002.

\bibitem{wang2004image}
{\sc Z.~Wang, A.~C. Bovik, H.~R. Sheikh, and E.~P. Simoncelli}, {\em Image quality assessment: from error visibility to structural similarity}, IEEE Transactions on Image Processing, 13 (2004), pp.~600--612.

\bibitem{west1987scale}
{\sc M.~West}, {\em On scale mixtures of normal distributions}, Biometrika, 74 (1987), pp.~646--648.

\bibitem{wright2015coordinate}
{\sc S.~J. Wright}, {\em Coordinate descent algorithms}, Mathematical Programming, 151 (2015), pp.~3--34.

\bibitem{xiao2023sequential}
{\sc Y.~Xiao and J.~Glaubitz}, {\em Sequential image recovery using joint hierarchical {B}ayesian learning}, Journal of Scientific Computing, 96 (2023), p.~4.

\bibitem{zonoobi2011gini}
{\sc D.~Zonoobi, A.~A. Kassim, and Y.~V. Venkatesh}, {\em Gini index as sparsity measure for signal reconstruction from compressive samples}, IEEE Journal of Selected Topics in Signal Processing, 5 (2011), pp.~927--932.

\end{thebibliography}

\end{document}